\documentclass{m2an}
\usepackage{etoolbox,tikz,pgfplots,array,hyperref,amssymb,thmtools,xcolor,mathtools,enumitem,bm}
\hypersetup{
    colorlinks=true,
    citecolor=blue,
    linkcolor=blue
    }
\pdfpagesattr{ /CropBox [0 40 555 760] }
\usepackage[a4,center]{crop}
\pgfplotsset{compat=1.17}
\usepackage[capitalise,nameinlink,noabbrev]{cleveref}

\newtheorem{theorem}{Theorem}[section]

\newtheorem{lemma}[theorem]{Lemma}

\newtheorem{assumption}[theorem]{Assumption}
\newtheorem{problem}[theorem]{Problem}
\newtheorem{experiment}[theorem]{Experiment}

\numberwithin{equation}{section}
\crefname{subsection}{Subsection}{Subsections}
\crefformat{equation}{(#2#1#3)}
\crefformat{align}{(#2#1#3)}
\crefformat{enumi}{(#2#1#3)}
\crefname{figure}{Figure}{Figures}
\newcommand{\ddt}{\frac{\dd}{\dd\mathrm{t}}}

\crefname{theorem}{Theorem}{Theorems}
\crefname{definition}{Definition}{Definitions}
\crefname{lemma}{Lemma}{Lemmas}
\crefname{assumption}{Assumption}{Assumptions}
\crefname{problem}{Problem}{Problems}
\crefname{experiment}{Experiment}{Experiments}

\def\phm{\phantom{-}}
\newcommand{\pt}{\partial_t}
\newcommand{\eps}{\varepsilon}

\newcommand{\ds}{\,\textup{d}s}
\newcommand{\dx}{\,\textup{d}x}
\newcommand{\dy}{\,\textup{d}y}
\newcommand{\dt}{\,\textup{d}t}

\newcommand{\dd}{\mathop{}\!\mathrm{d}}

\renewcommand{\ddt}{\frac{\dd}{\dd\mathrm{t}}}
\renewcommand{\dt}{\,\textup{d}t}
\newcommand{\con}{\hookrightarrow}
\newcommand{\longweak}{\relbar\joinrel\rightharpoonup}
\newcommand{\com}{\mathrel{\mathrlap{{\mspace{4mu}\lhook}}{\hookrightarrow}}}
\newcommand{\ifs}{{\text{if }}}

\newcommand{\bbR}{\mathbb{R}}

\renewcommand{\div}{\operatorname{div}}
\renewcommand{\epsilon}{\varepsilon}
\def\norm #1{\left\|#1\right\|}
\let\originalleft\left
\let\originalright\right
\renewcommand{\left}{\mathopen{}\mathclose\bgroup\originalleft}
\renewcommand{\right}{\aftergroup\egroup\originalright}

\renewcommand{\rho}{\varrho}

\newcommand{\R}{\bbR}
\renewcommand{\tilde}{\widetilde}

\newcommand{\N}{\mathbb{N}}
\newcommand{\ring}{\mathring}
\newcommand{\inv}{(-\Delta)^{-1}}
\newcommand{\p}{\partial}
\renewcommand{\bar}[1]{\mkern 1.5mu\overline{\mkern-1.5mu#1\mkern-1.5mu}\mkern 1.5mu}

\def\Th{\mathcal{T}_h}
\def\Itau{\mathcal{I}_\tau}
\def\Vh{\mathcal{V}_h}
\def\Vhm{\tilde{\mathcal{V}}_h}
\def\dtau{d^{n+1}_\tau}
\usepackage[misc]{ifsym}
\mathcode`\;="603B
\mathcode`\,="613B
\newcolumntype{M}[1]{>{\centering\arraybackslash}m{#1}}
\allowdisplaybreaks
\begin{document}
\title{Analysis and discretization of the Ohta--Kawasaki equation with forcing and degenerate mobility}
\author{Aaron Brunk}\address{Institute of Mathematics, Johannes Gutenberg University, Mainz, Germany}
\author{Marvin Fritz}\address{Computational Methods for PDEs, Johann Radon Institute for Computational and Applied Mathematics, Linz, Austria}
%
%
\begin{abstract} The Ohta--Kawasaki equation models the mesoscopic phase separation of immiscible polymer chains that form diblock copolymers, with applications in directed self-assembly for lithography. We perform a mathematical analysis of this model under degenerate mobility and an external force, proving the existence of weak solutions via an approximation scheme for the mobility function. Additionally, we propose a fully discrete scheme for the system and demonstrate the existence and uniqueness of its discrete solution, showing that it inherits essential structural-preserving properties. Finally, we conduct numerical experiments to compare the Ohta--Kawasaki system with the classical Cahn--Hilliard model, highlighting the impact of the repulsion parameter on the phase separation dynamics. \end{abstract}
\date{\hphantom{.}}
\mbox{}\vspace{-1cm}

\maketitle

\section{Introduction}
Diblock copolymers are linear polymer chains made up of two bonded diblocks of mutually repellent polymer species.
At lower temperatures, a combination of such diblock copolymers tends to separate into discrete phases because of their repulsion.
However, due to the connection between the two polymer diblocks, phase separation occurs on the polymer length scales, resulting in a process known as microphase separation \cite{leibler1980theory}.
Depending on the composition of the diblock copolymer, this results in the development of periodic nanoscale patterns in the form of lamellae, spheres, cylinders, and gyroids; see \cite{abetz2005phase,bates2000block,hamley1998physics}. 
Directed self-assembly of diblock copolymers is one of the most promising advancements in the fabrication of nanoscale devices at an economical price. 
Such self-assembly procedures are researched in the semiconductor industry as a promising means of expanding the capabilities of standard lithographic techniques to build more fine-scale patterned structures for electronic circuits; see \cite{chen2020directed}.
In addition, it is scalable and cost-effective for mass production and has been used in the production of photonic metamaterials \cite{stefik2015block}, membranes for viral filtering \cite{yang2008virus}, and functional materials \cite{kim2010functional}.

In this study, the Ohta--Kawasaki model for the microphase separation of diblock copolymers is investigated.
First presented in \cite{ohta1986equilibrium}, the Ohta--Kawasaki model is part of the density functional theory derived from the self-consistent field theory model. Moreover, it belongs to a class of nonlocal Cahn--Hilliard models and is well accepted in the scientific literature, see \cite{Choksi2009,Choksi2011,Choksi2012,qin2013evolutionary,Fredrickson2005,Cao2022,muller2018continuum,ohta1986equilibrium,Choksi2003,fritz2022time,zou2022fully}. While the Ohta--Kawasaki model introduces a nonlocal term in the free energy to capture repulsive interactions, other nonlocal Cahn--Hilliard models frequently incorporate nonlocal diffusion for long-range interactions in place of standard diffusion, a feature particularly relevant in applications such as cell biology; see, e.g., \cite{fritz2018unsteady,fritz2023tumor}. Furthermore, there are also variants of the Cahn--Hilliard model that are nonlocal-in-time in contrast to the nonlocal-in-space Ohta--Kawasaki model, see, e.g., \cite{fritz2021equivalence}. 

Numerous investigations, see \cite{van_den_Berg_2017,Choksi2012,Choksi2009,Choksi2011}, have shown that the Ohta--Kawasaki model successfully portrays periodic patterns of the diblock copolymer microphase separation that correspond to those found in experiments. In the previous works, it has been exploited that in the case of a constant mobility, the system can be transformed into the Cahn--Hilliard equation with an additional term on the right-hand side. However, it was stated \cite{xia2007local,gomez2011provably} that degenerate mobility in phase-field models can more accurately describe the physics of phase separation, as pure phases must have vanishing mobility. In this work, we study the mathematical analysis of the Ohta--Kawasaki model with degenerating mobilities. In addition, external forces are allowed our analysis, which has also not been done for the Ohta--Kawasaki model, see \cite{Ebenbeck21} for the Cahn-Hilliard case with one-sided mobility. As the external force disrupts the mass conservation of the model, the non-locality becomes more difficult, especially for the numerical discretization of the model.

In \cref{Sec:Mod}, we introduce the Ohta--Kawasaki model and derive it from an underlying energy formulation. In addition, we introduce the functionals and parameters involved.
In \cref{Sec:Analysis}, we present the complete analysis of the model. We begin in \cref{Sec:Pre} with various well-known analytical results that we require in the following subsections to prove the existence of weak solutions.
In \cref{Sec:Nondeg}, we prove the existence of solutions in the non-degenerate case by using the Galerkin method and discretizing the PDE in space. We derive an energy result and pass to the limit in the formulation.
In \cref{Sec:Deg}, we consider the case of degenerate mobility and introduce a further regularization of the PDE. In \cref{Sec:Num}, we present a fully discrete scheme for the Ohta--Kawasaki system. We establish the existence and uniqueness of a discrete solution that preserves key structural quantities of the model, such as mass balance and energy dissipation. We show several simulations to highlight the microphase separation in \cref{Sec:NumIll} when comparing different values of the crucial repulsion parameter including the case of the standard Cahn--Hilliard equation. Finally, we mention some conclusions to the work in \cref{Sec:Concl}.
\section{Mathematical model } \label{Sec:Mod}

We model the microphase separation of a block copolymer mixture under the influence of a substrate chemical field. In this regard, we consider diblock copolymers, which are block copolymers consisting of two chemically distinct polymer species. We refer to one of the two species as phase $A$ and the other as phase $B$.  Given some bounded domain $\Omega \subset \mathbb{R}^{d}$ in dimension $d\in \{2,3\}$,  the normalized segment densities of the monomers $A$ and $B$ inside the domain are denoted by $\phi_A$,$\phi_B:\Omega\to[0,1]$. They satisfy the incompressibility constraint $\phi_A+\phi_B = 1$ and are governed by the generalized mass conservation equation.
We define their difference $\phi=\phi_A-\phi_B:\Omega \to [-1,1]$ that represents the normalized local density difference between the two phases and, in this sense, the value of $\phi$ at $x \in \Omega$ indicates the local composition of the mixture, taking a value of $\phi(x) = 1$ if locally the mixture is purely phase $A$, $\phi(x) = -1$  for purely phase $B$, and $\phi(x) \in  (-1,1)$ when interpolating between the pure phases. 

Based on a microforce balance \cite{gurtin1996generalized}, the phase-field $\phi$ is governed by the partial differential equation
\begin{equation} \label{Eq:Mass}
\pt \phi = \div(m(\phi) \nabla\mu) + f(\phi),\end{equation}
where $\mu$ is given as the G\^ateaux derivative of the underlying energy $\mathcal{E}$, $m$ a mobility definition depending on the phase field $\phi$ such as $m(\phi)=\frac14 (1-\phi^2)^2$, and $f$ a forcing that may depend on $\phi$.
The Ohta--Kawasaki energy functional $\mathcal{E}$ accounts for the effects of the substrate chemical and is defined as, see \cite{ohta1986equilibrium},
	\begin{equation} 
 \label{Eq:Energy}
 \mathcal{E}(\phi)=\int_\Omega \Psi(\phi) + \frac{\eps^2}{2} |\nabla \phi|^2 +\frac{\kappa}{2} (\phi-\bar\phi)(-\Delta)^{-1} (\phi-\bar\phi) \, \dd x,
	\end{equation}
where $\Psi$ is a double-barrier or double-well potential such as $\Psi(\phi)=\frac14 (1-\phi^2)^2$. The parameter $\varepsilon$ describes the strength of repulsion between the two polymer phases, while $\kappa$ describes the strength of the nonlocal effect provided by the bonding of the polymer chain. In particular, $\kappa$ measures the strength of the repulsion between bubbles and favours small domains, therefore, the larger $\kappa$ is, the more bubbles the assembly has. The mass average $$\bar\phi:=\frac{1}{|\Omega|} \int_\Omega \phi(x) \dx$$ represents the ratio between the masses of the two polymer phases in the overall mixture. 

As the chemical potential $\mu$ is defined as the first variation of the Ohta--Kawasaki energy \eqref{Eq:Energy} with respect to $\phi$, we calculate its derivative as follows:
$$\begin{aligned}
\int_\Omega \mu \xi \dx &=\frac{\dd}{\dd\theta}\bigg\vert_{\theta=0}  \int_\Omega \! \Psi(\phi+\theta \xi) + \frac{\eps^2}{2} |\nabla(\phi+\theta \xi)|^2  + \frac{\kappa}{2} (\phi+\theta\xi-\bar{\phi+\theta\xi})(-\Delta)^{-1} (\phi+\theta \xi-\bar{\phi+\theta\xi}) \dx 
\\ &= \int_\Omega \! \big( \Psi'(\phi) - \eps^2 \Delta \phi\big) \xi  +\kappa (\xi-\bar\xi) (-\Delta)^{-1} (\phi -\bar\phi) \big) \dx \end{aligned} $$
for any test function $\xi \in C_0^\infty(\Omega)$. By the fundamental lemma of calculus of variations, we obtain the chemical potential $\mu$ as follows:
$$\mu = \Psi'(\phi) - \eps^2 \Delta \phi + \kappa \inv(\phi-\bar\phi) - \kappa \bar{\inv(\phi-\bar\phi)}$$ 
As the inverse Laplacian of a mean-free function is again mean-free, the last term is zero.
Consequently, it yields  the Ohta--Kawasaki system 
\begin{equation} \label{Eq:Ohta}
	\begin{aligned}
	\p_t\phi &= \div( m(\phi) \nabla\mu) + f(\phi), \\ 
	\mu &= \Psi'(\phi) - \eps^2 \Delta \phi + \kappa \inv(\phi-\bar\phi).
	\end{aligned}
\end{equation}
As it is favourable for our analysis, we introduce a further variable $\nu$ to the system and seek the triple $(\phi,\mu,\nu)$ such that it holds
\begin{equation} \label{Eq:Ohtaequiv}
	\begin{aligned}
	\p_t\phi &= \div( m(\phi) \nabla\mu) + f(\phi),\\
	\mu &= \Psi'(\phi) - \eps^2 \Delta \phi + \kappa \nu,\\
 	-\Delta \nu &= \phi - \bar\phi,
	\end{aligned}
\end{equation}
and we equip the system with the homogeneous Neumann boundary data $\partial_n \phi = \partial_n \mu=\partial_n \nu = 0$ on $\p\Omega$. \medskip

\noindent\textbf{Special case 1:} If one assumes a constant mobility function, wlog $m(\phi)=1$, \eqref{Eq:Ohta} reduces to the simplified system
\begin{equation}\label{Eq:CHOhta}
	\begin{aligned}
	\p_t\phi &= \Delta\mu  -\kappa (\phi-\bar\phi)+f(\phi),\\
	\mu &= \Psi'(\phi) - \eps^2 \Delta \phi,
	\end{aligned}
\end{equation}
i.e., the nonlocal term appears as a source function in the first equation of the Cahn--Hilliard system. \medskip

\noindent\textbf{Special case 2:} If we assume in addition that there is no external force, that is $f=0$, or that the force is mean-free, that is $\bar f=0$, then we observe that the system is mass conserving by integrating the first equation is space. Thus, we have $\bar\phi=\bar\phi_0$ with $\phi_0=\phi(0)$ being the phase field's initial. We conclude that the model becomes local-in-space.

\section{Analysis} \label{Sec:Analysis}

We prove the existence of weak solutions using the Galerkin method and compactness theorems; see \cref{Thm:PosMob} below for the statement of the result. Moreover, we prove the uniqueness and continuous dependence on the data in the case of constant mobility. In both cases, we assume a potential with certain properties that are fulfilled by the standard double-well potential. In \cref{Thm:Reg} in \cref{Sub:Reg} below, we prove higher regularity results of the solutions. For more general cases, for example, the degenerate mobility $m(x)=(1-x^2)^2$ and potentials of double-obstacle and Flory--Huggins type, we refer to \cref{Thm:DegMob} in \cref{Sec:Deg} below.

\subsection{Preliminaries} \label{Sec:Pre}
We denote a generic constant simply by $C>0$ and for brevity, we may write $x \lesssim y$ instead of $x \leq Cy$.
We recall the Poincar\'e--Wirtinger inequality \cite{brezis2010functional} 
\begin{equation} \begin{aligned}
		\Vert  u-\bar u\Vert _{L^p(\Omega)} & \lesssim \Vert \nabla u\Vert _{L^p(\Omega)} && \forall u \in W^{1,p}(\Omega).
	\end{aligned} \label{Eq:SobolevInequality} \end{equation}
Further, we introduce the Hilbert--Sobolev space of first order with zero mean by
$$\mathring{H}^1(\Omega)=\{v \in H^1(\Omega): \bar v=0 \},$$
and its dual $\mathring{H}^{-1}(\Omega)$ can be naturally extended to a subspace of $H^{-1}(\Omega):=(H^1(\Omega))'$ by taking the constant $H^1(\Omega)$-function as a kernel, i.e.,
$$\mathring{H}^{-1}(\Omega)=\{f \in H^{-1}(\Omega): \langle f,v \rangle_{H^1(\Omega)} = 0 \quad\forall v\in H^1(\Omega) \text{ with } v=\text{const a.e.}\},$$
where $\langle a,b\rangle_X$ denotes the dual paring between $X'\times X$.
We define the inverse Laplace operator $\inv : \ring H^{-1}(\Omega) \to \ring H^1(\Omega)$ by
$$(\nabla\inv f,\nabla v)_{L^2(\Omega)}=\langle f,v\rangle_{H^1(\Omega)} \quad \forall f \in \ring H^{-1}(\Omega), v\in H^1(\Omega),$$
which is linear, bounded, bijective, and self-adjoint. We equip the mean-free space $\ring H^{-1}(\Omega)$ with the graph norm $\|\cdot\|_*=\|\nabla(-\Delta)^{-1} \cdot \|_{L^2(\Omega)}$ and the scalar product 
$$\begin{aligned}(f,g)_*&=((-\Delta)^{-1} f,(-\Delta)^{-1} g)_{H^1(\Omega)}=(\nabla(-\Delta)^{-1} f,\nabla(-\Delta)^{-1} g)_{L^2(\Omega)}=(f,(-\Delta)^{-1} g)_{L^2(\Omega)}.\end{aligned}$$

For notational simplicity, we define the spaces $H, \mathring{H}$ and $V, \mathring{V}$ in the Gelfand triples
\begin{align*}
 V=H^1(\Omega)&\hookrightarrow  \hookrightarrow H=L^2(\Omega) \hookrightarrow V', \\
\mathring{V}=\mathring{H}^1(\Omega)&\hookrightarrow  \hookrightarrow \mathring{H}=\mathring{L}^2(\Omega) \hookrightarrow \mathring{V}'.
\end{align*}

\subsection{Nondegenerating case} \label{Sec:Nondeg} First, we consider the case of positive and bounded mobility, e.g., the continuous function $$m(x)=\delta+\beta \chi_{[-1,1]}(x) (1-x^2)^2,$$ for $\delta,\beta>0$.  Here, $\chi$ denotes the characteristic function and $\beta=0$ gives constant mobility $m(x)=\delta$.
We make the following assumption regarding the well-posedness theorem below. We want to highlight that the dimension $d$ of the spatial domain can be arbitrarily high.
\begin{assumption} \label{As:Pos} ~\\[-0.3cm]
\begin{enumerate}[label=(A\arabic*), ref=A\arabic*, leftmargin=.9cm] \itemsep.1em
    \item $\Omega \subset \R^d$ bounded $C^{1,1}$-domain with $d \geq 1$, $T>0$ finite time horizon, $Q_T=(0,T)\times \Omega$.
    \item $\phi_0 \in V$.
    \item $m \in C^0(\R)$ such that $0< m_0 \leq m(x) \leq m_\infty$ for all $x \in \R$ for some constants $m_0,m_\infty<\infty$. \label{Ass:Pos:Mob}
    \item $\Psi \in C^{1,1}(\R;\R_{\geq 0})$ such that $\Psi(0)=\Psi'(0)=0$, $\Psi''(x) \geq -C_\Psi$, and $|\Psi'(x)|\leq C_\Psi' (1+\Psi(x))$ for all $x \in \R$ for some $C_\Psi,C_\Psi' < \infty$.
    \label{Ass:Pos:Pot}
    \item $f \in C^{0}(\R;\R_{\geq 0})\cap L^\infty(\R)$. \label{Ass:Pos:Fun}
    \item $\eps,\kappa$ are positive constants. \label{Ass:Pos:param}
\end{enumerate}
\end{assumption}

We state the existence and uniqueness theorem as follows:

\begin{theorem} \label{Thm:PosMob} Let \cref{As:Pos} hold. Then there exists a weak solution $(\phi,\mu,\nu)$ with 
	$$\begin{aligned}
	\phi &\in H^1(0,T;V') \cap C([0,T];H) \cap L^\infty(0,T;V),\\
	\mu &\in L^2(0,T;V), \\
    \nu &\in L^\infty(0,T;\mathring{V}),
	\end{aligned}$$
	to the Ohta--Kawasaki equation \eqref{Eq:Ohta} in the sense that $\phi(0)=\phi_0$ in $H$ and
	\begin{equation} \label{Eq:Weak}
	\begin{aligned}
	\langle \pt \phi,\xi_1\rangle_V+ (m(\phi) \nabla \mu,\nabla \xi_1)_H &= (f(\phi),\xi_1)_H, \\
	 (\Psi'(\phi),\xi_2)_H + \eps^2 (\nabla \phi,\nabla \xi_2)_H + \kappa (\nu,\xi_2)_H & =(\mu,\xi_2)_H , \\
    (\nabla \nu, \nabla \xi_3)_H &= (\phi-\bar\phi,\xi_3)_H,
	\end{aligned}
	\end{equation}
	for any $\xi_1,\xi_2 \in V$ and $\xi_3\in\mathring{V}$; the triple satisfies the a priori bounds
	\begin{equation} \label{Eq:EnergySolution} \begin{aligned}
	&\|\phi\|_{L^\infty(V)}^2 +  \|\Psi(\phi)\|_{L^\infty(L^1(\Omega))} + \|\nu\|_{L^\infty(V)}^2 +\|\mu\|_{L^2(V)}^2+\big\|\sqrt{m(\phi)}\nabla \mu\big\|_{L^2(H)}^2 \lesssim \norm{f(\phi)}_{L^2(H)} + \|\phi_0\|_{V}^2.
	\end{aligned} \end{equation}
 Furthermore, the integrated energy inequality holds for almost all $t\in(0,T)$ given by
 \begin{align} \label{Eq:EnergyIneq}
  \mathcal{E}(\phi(t)) + \int_0^t m(\phi(s))|\nabla\mu(s)|^2 \ds \leq \mathcal{E}(\phi(0))  + \int_0^t (f(\phi(s)),\mu(s))_H \ds.
 \end{align}
 If the mobility $m$ is constant and $f=0$, then the solution is unique and depends continuously on the data $\phi_0$.
\end{theorem}

\noindent \textit{Proof.} We employ the Galerkin method to reduce the system to ODEs, which admit a solution $(\phi^k,\mu^k,\nu^k)$ due to classical existence results such as the theorem of Cauchy--Peano. We derive energy estimates that imply the existence of weakly convergent subsequences. We pass to the limit $k\to \infty$ and apply compactness methods to return to the variational formulation of the Ohta--Kawasaki equation \eqref{Eq:Ohta}. \medskip
	
\noindent\textbf{Step 1: Discrete approximation.} Let $\{h_k\}_{k \in \N}$ be the eigenfunctions of the Neumann--Laplace operator with corresponding eigenvalues $\{\lambda_k\}_{k\in \N}$. The eigenfunctions form an orthonormal basis  in $H$ and an orthogonal basis in $V$ with $(\nabla h_i,\nabla h_j)_H=\lambda_j \delta_{ij}$, see \cite{brezis2010functional}. We choose $h_1=|\Omega|^{-1/2}$ with $\lambda_1=0$. 
	 We pursue functions $(\phi^k,\mu^k,\nu^k)$ which take their values in the finite-dimensional space $H_k=\{h_1,\dots,h_k\}$, i.e., is of the form
	\begin{equation} \label{Eq:Ansatz}
	    \begin{aligned}
\phi^k(t) = \sum_{j=1}^k \phi^k_j(t) h_j, \qquad \mu^k(t) = \sum_{j=1}^k \mu^k_j(t) h_j, \qquad \nu^k(t) = \sum_{j=1}^k \nu^k_j(t) h_j,
	    \end{aligned}
	\end{equation}
	with coefficient functions $\phi_j^k,\mu_j^k, \nu_j^k : (0,T) \to \R$, $j \in \{1,\dots,k\}$, that solve the Galerkin system 
	\begin{equation} \begin{aligned}
	(\p_t \phi^k,u)_H + (m(\phi^k) \nabla \mu^k, \nabla u)_H &= (f(\phi^k),u)_H,  \\
    (\Psi'(\phi^k),v)_H + \eps^2 (\nabla \phi^k, \nabla v)_H +\kappa (\nu^k,v)_H & =(\mu^k,v)_H , \\
    (\nabla \nu^k, \nabla w)_H &= (\phi^k-\bar\phi^k,w)_H,
	\end{aligned} \label{Eq:FaedoUV}
	\end{equation}
	for all $u,v,w \in H_k$. We equip the system with the initial data $\phi^k(0)=\Pi_k \phi_0$ with $\Pi_k: H \to H_k$, $h \mapsto \sum_{i=1}^k (h,h_j)_H h_j$, denoting the orthogonal projection onto $H_k$. In particular, we exploit its key properties $\|\Pi_k\|_{\mathcal{L}(H)}\leq 1$, $\|\Pi_k\|_{\mathcal{L}(V)}\leq 1$ and
	\begin{equation} \label{Eq:Projection}
	\phi^k(0) = \Pi_k \phi_0 \longrightarrow \phi_0 \quad \text{in $V$ as }  k \to \infty,\end{equation}
	e.g., see \cite[Lemma 7.5]{robinson2001infinite}.

	Since the test functions $u,v,w \in H_k$ are spanned by the eigenfunctions $h_j$, $j \in \{1,\dots,k\}$, we can equivalently write the Galerkin system as
		\begin{subequations} \begin{align}
	(\p_t \phi^k,h_j)_H + (m(\phi^k) \nabla \mu^k, \nabla h_j)_H &= (f(\phi^k),h_j)_H, \label[equation]{Eq:FaedoPhi} \\
	(\Psi'(\phi^k),h_j)_H + \eps^2 (\nabla \phi^k, \nabla h_j)_H + \kappa (\nu^k,h_j)_H &= (\mu^k,h_j)_H, \label[equation]{Eq:FaedoMu} \\
  (\nabla \nu^k,\nabla h_j)_H &= (\phi^k-\bar\phi^k,h_j)_H, \label[equa3tion]{Eq:FaedoNu} 
	\end{align} \label[equation]{Eq:Faedo}
	\end{subequations}
	for all $j \in \{1,\dots, k\}$.
     Inserting the ansatz functions \eqref{Eq:Ansatz} into this system  and exploiting the orthonormality of the eigenfunctions in $H$ and their orthogonality in $V$, yields
		\begin{equation} \label{Eq:DiscreteCoeff} \begin{aligned}
	\p_t \phi^k_j &=- \lambda_j  \mu^k_j \sum_{i=1}^k \big(m\big( \textstyle \sum_{j=1}^k \phi^k_j(t) h_j\big) \nabla h_i,\nabla h_j \big)_H + \big(f\big(\textstyle \sum_{j=1}^k \phi^k_j(t) h_j\big),h_j\big)_H,  \\
	\mu^k_j &= \big(\Psi'\big(\textstyle \sum_{j=1}^k \phi^k_j(t) h_j\big),h_j\big)_H + \eps^2 \lambda_j \phi^k_j + \kappa \nu_j^k,  \\
 \lambda_j \nu_j^k &= \phi_j^k-\phi_1^k \delta_{1j},
	\end{aligned} 
	\end{equation}
 for all $j \in \{1,\dots,k\}$, and the initial data $\phi_j^k(0)=(\phi_0,h_j)_H$. Here, we have used that
 $$(\bar \phi^k,h_j)_H=\frac{1}{|\Omega|} (1,\phi^k)_H (1,h_j)_H=(h_1,\phi^k)_H (h_1,h_j)_H=\phi_1^k \delta_{1j}.$$
 We may plug $\nu_j^k$ from \cref{Eq:DiscreteCoeff}$_3$ into \cref{Eq:DiscreteCoeff}$_1$ and $\mu^k_j$ from \cref{Eq:DiscreteCoeff}$_2$ into \cref{Eq:DiscreteCoeff}$_1$. Thus, the right-hand side in \cref{Eq:DiscreteCoeff}$_1$ is continuously dependent on $\phi_1^k,\dots,\phi_k^k$.
By the existence theory of ODEs, there exists a solution $(\phi_j^k,\mu_j^k,\nu_j^k)$ to the ODE on a time interval $[0,T_k)$ with either $T_k=\infty$ or $T_k<\infty$ and $\limsup_{t \to T_k} |(\phi_1^k,\dots,\phi_k^k)|_{\ell^2}=\infty$. 
Therefore, we have shown the existence of a solution tuple $$(\phi^k,\mu^k,\nu^k) \in H^1(0,T_k;H_k)\times L^2(0,T_k;H_k)\times H^1(0,T_k;H_k),$$
	to the Galerkin system \eqref{Eq:Faedo}. \medskip

	\noindent\textbf{Step 2: Energy estimates.} After we have proved the existence of a solution to the ODE, we can begin to test the Galerkin system \eqref{Eq:FaedoUV} with suitable test functions.
	First, we take the test functions $u=\mu^k$, 
 $v=\p_t \phi^k$ and 
 $w=-\kappa\pt (-\Delta_k)^{-1} (\phi^k-\bar\phi^k)$ in \eqref{Eq:FaedoUV}. We note that $w$ is indeed in $H_k$ as we define the inverse Laplacian of a mean-free function $w$ as
\begin{equation}
\langle \nabla w_k, \nabla \xi_k \rangle = \langle w, \xi_k \rangle \text{ for all } \xi_k \in H_k. \smallskip
\end{equation}
Then we find $K\vec w_k = M\vec w$ and 
inverting $K$ gives 
$\vec w_k = K^{-1}\vec w$
and we define  
$w_k = (-\Delta_k)^{-1} w$. Thus, selecting the test functions as written above gives
	\begin{equation*}
	    \begin{aligned}
	    (\partial_t \phi^k,\mu^k)_H  + (m(\phi^k) \nabla \mu^k, \nabla \mu^k)_H&=(f(\phi^k),\mu^k)_H,  \\
	    (\Psi'(\phi^k) ,\p_t \phi^k)_H + \eps^2 (\nabla \phi^k, \nabla \p_t \phi^k)_H+\kappa(\nu^k,\pt \phi^k)_H&=(\mu^k,\p_t \phi^k)_H, \\-\kappa(\nu^k,\pt \phi^k-\pt\bar\phi^k)_H +\kappa(\phi^k-\bar\phi,\pt (\phi^k-\bar\phi^k))_* &=0.
	    \end{aligned}
	\end{equation*}
 It holds
$$|\Omega|\pt \bar\phi^k=\pt (\phi^k,1)_H=(\pt \phi^k,1)_H=(f,1)_H,\smallskip $$
from which we observe $(\nu^k,\pt\bar\phi^k)_H=0$.
Adding the tested equations then yields
\begin{equation} \label{Eq:EEstimate2} 
	\begin{aligned}
&(m(\phi^k),|\nabla \mu^k|^2)_H+\eps^2 (\nabla \phi^k, \p_t
\nabla \phi^k)_H + (\Psi'(\phi^k),\p_t \phi^k)_H    + \kappa (\pt (\phi^k-\bar\phi^k),\phi^k-\bar\phi^k)_{*}  = (f,\mu^k)_H . 
	\end{aligned}
	\end{equation}
By using the test function $v=1$, we observe $$|\Omega|\bar\mu^k=(\mu^k,1)_H = (\Psi'(\phi^k),1)_H.$$ and consequently, the Poincar\'e--Wirtinger inequality  \eqref{Eq:SobolevInequality} gives
	\begin{equation} \label{Eq:MuEstimate}
	\|\mu^k\|_H \leq \big\|\mu^k - \bar\mu^k \big\|_H +\|\bar\mu^k\|_H \lesssim \|\nabla \mu^k\|_H + \|\Psi'(\phi^k)\|_{L^1(\Omega)} .
	\end{equation}
	Due to assumption \cref{Ass:Pos:Pot} we can bound the derivative of the potential by the potential itself, that is,
	\begin{equation}\label{Eq:EEstimate_mu}
	    \|\mu^k\|_H \lesssim 1+\|\nabla \mu^k\|_H+\|\Psi(\phi^k)\|_{L^1(\Omega)}.
	\end{equation}
We use the lower bound of $m$, see \cref{Ass:Pos:Mob}, insert \eqref{Eq:EEstimate_mu} on the right-hand side in \eqref{Eq:EEstimate2} and use the $\eps$-Young inequality to make the prefactors of $\|\nabla \mu^k\|_H$ sufficiently small to absorb them by the terms on the left-hand side of the inequality. This procedure gives an estimate
\begin{equation} \label{Eq:EEstimate1} 
	\begin{aligned}
&\ddt \left[  \frac{\eps^2}{2} \|\nabla \phi^k\|_H^2 + \|\Psi(\phi^k)\|_{L^1(\Omega)} + \frac{\kappa}{2} \|\phi^k-\bar\phi^k\|_*^2 \right] +  \frac{m_0}{2} \|\nabla \mu^k\|_H^2   \lesssim 1+\|f\|_H^2 +  \|\Psi(\phi^k)\|_{L^1(\Omega)}.
	\end{aligned}
	\end{equation}
 Integrating both sides over $(0,t)$, $t \in (0,T_k)$, and applying the 
applying the Gronwall inequality yields
\begin{equation} \label{Eq:FFinal2} \begin{aligned} &\|\nabla \phi^k(t)\|_{H}^2 +  \|\Psi(\phi^k(t))\|_{L^1(\Omega)} +  \|\phi^k(t)-\bar\phi^k(t)\|_*^2  + \|\nabla \mu^k\|_{L^2(0,t;H)}^2 \\     &\lesssim 1+\|f\|_{L^2(H)}^2+\|\nabla \phi^k_0\|_H^2+\|\Psi(\phi_0^k)\|_{L^1(\Omega)}+ \|\phi^k_0-\bar\phi^k_0\|_*^2,
\end{aligned}
\end{equation}
for almost all $t\in (0,T_k)$.  We estimate the last term on the right-hand side as
 $$\|\phi^k_0-\bar\phi^k_0\|_*^2 
 \leq \|\phi^k_0-\bar\phi^k_0\|_H^2 
 \leq 2\|\phi^k_0\|_H^2 + 2(\bar\phi^k_0)^2 |\Omega| \lesssim \|\phi^k_0\|_H^2 \lesssim \|\phi_0\|_H^2,$$
 where we used the boundedness of the projection operator. Likewise, we can bound the third term on the right-hand side and the fourth term follows by the Lipschitz continuity of $\Psi$ and $\Psi(0)=0$, see \cref{Ass:Pos:Pot}, as follows
 $$\|\Psi(\phi_0^k)\|_{L^1(\Omega)} \lesssim \|\phi_0^k\|_{L^1(\Omega)} \lesssim 1+\|\phi_0\|_H^2.$$
 Thus, the right-hand side of \cref{Eq:FFinal2} becomes independent of $k$, we can argue with a blow-up criterion and extend the time interval by setting $T_k=T$ for all $k$.

We repeat the arguments and use the derived bounds to obtain a uniform bound of $\textstyle\sqrt{m(\phi^k)} \nabla \mu^k$ in $L^2(0,T;H)$. In addition, we get a full bound of $\mu^k$ in $L^2(0,T;H)$ by \cref{Eq:MuEstimate} and of $\phi^k$ in $L^\infty(0,T;V)$ by using the test function $u=\phi^k$.
We obtain the energy inequality
\begin{equation}   \label{Eq:FFinalEnergy} \begin{aligned}
\mathcal{E}(\phi^k(t)) + \int_0^t m(\phi^k)|\nabla\mu^k|^2  \ds = \mathcal{E}(\phi^k(0)) + \int_0^T (f(\phi^k),\mu^k)_H \ds 
\end{aligned} \end{equation}
These finally yield the energy inequality
\begin{equation}   \label{Eq:FFinalbounds} \begin{aligned} &\|\phi^k\|_{L^\infty(V)}^2+  \|\Psi(\phi^k)\|_{L^\infty(L^1(\Omega))} + \|\phi^k-\bar\phi^k\|_{L^\infty(\ring V)}^2 + \|\mu^k\|_{L^2(V)}^2 +\smash{\big\|\textstyle\sqrt{m(\phi^k)} \nabla \mu^k\big\|_{L^2(H)}^2}  \\ &  \lesssim 1+\|f\|_{L^2(H)}^2+ \|\phi_0\|_V^2.
\end{aligned} \end{equation}
We obtain an estimate on $\nu^k$ by choosing $w=\nu^k+(-\Delta_k)^{-1}\nu^k$, which gives
$$\|\nabla \nu^k\|_H^2+\|\nu^k\|_H^2 \leq (\phi^k-\bar\phi^k,\nu^k)_H + (\phi^k-\bar\phi^k,\nu^k)_*$$
Thus
$\|\nu^k\|_V^2 \lesssim \|\phi^k-\bar\phi^k\|_*^2$
and therefore, $\nu^k$ is bounded in $L^\infty(0,T;\ring V)$. \medskip
	
	\noindent\textbf{Step 3: Estimate on the  time-derivative} The energy estimate \eqref{Eq:FFinalEnergy} already gives the existence of converging subsequences. Since the Faedo--Galerkin system \eqref{Eq:Faedo} involves the nonlinear functions $\Psi$ and $m$, we need to derive an estimate on the time derivative of $\phi$ in order to apply the Aubin--Lions compactness lemma and achieve strong convergence. 
	
	Let $u \in L^2(0,T;V)$. Then we have $\Pi_k u = \sum_{j=1}^k u_j^k h_j$ for time-dependent coefficient functions $u_j^k : (0,T) \to \R$, $j \in \{1,\dots,k\}$. We multiply equation \eqref{Eq:FaedoPhi} by $u_j^k$, take the sum from $j=1$ to $k$, and integrate over the interval $(0,T)$, which yields 
	\begin{equation*}
	\begin{aligned}\left|\int_0^T(\p_t \phi^k, u)_H \dt\right| = \left|\int_0^T(\p_t \phi^k,\Pi_k u)_H\dt\right| &\leq m_\infty \|\nabla \mu^k\|_{L^2(H)} \|\nabla \Pi_k u\|_{L^2(H)} + \|f\|_{L^2({Q_T})} \|\Pi_k u\|_{L^2(V)} \lesssim \|u\|_{L^2(V)},
	\end{aligned}
	\end{equation*}
	where we used the energy estimate \eqref{Eq:FFinalEnergy} to bound the terms on the right-hand side. Since $u$ was chosen arbitrarily, we have
	\begin{equation}	\label{Eq:Energy3}
	    \| \p_t \phi^k \|_{L^2(V')} = \sup_{\|u\|_{L^2(V)} \leq 1} \left|\int_0^T(\p_t \phi^k, u)_H \dt\right| \leq C(T,f,\phi_0).
	\end{equation}
	
	\noindent\textbf{Step 4: Limit process} A bounded sequence in a reflexive Banach space admits a weakly/weakly-$*$ convergent subsequence. By a standard abuse of notation, we drop the subsequence index. From the energy estimates \eqref{Eq:FFinalEnergy} and \eqref{Eq:Energy3}, we obtain the existence of limit functions $(\phi,\mu,\nu)$  such that
	$$\begin{aligned}	
	\phi^k &\longweak \phi &&\text{weakly-$*$ in } L^\infty(0,T;V), \\
 \phi^k-\bar\phi^k &\longweak \xi &&\text{weakly-$*$ in } L^\infty(0,T;\ring V), \\
	\p_t \phi^k &\longweak \p_t \phi &&\text{weakly\phantom{-*} in } L^2(0,T;V'), \\
		\phi^k &\longrightarrow \phi &&\text{strongly\hspace{1mm} in }
 L^p(0,T;H), \\
	\mu^k &\longweak \mu &&\text{weakly\phantom{-*} in } L^2(0,T;V),\\
 \nu^k &\longweak \nu &&\text{weakly-$*$ in } L^\infty(0,T;\mathring{V}),
	\end{aligned}$$
for all $p \in [1,\infty)$ as $k \to \infty$, where we applied the Aubin--Lions lemma to achieve the strong convergence of $\phi^k$. In addition, $\xi=\phi-\bar\phi$ holds as $\phi^k$ and $\bar\phi^k$ converge by themselves in $L^\infty(0,T;H)$ to $\phi$ and $\bar \phi$, respectively.

In the next step, we prove that the limit functions $\phi$ and $\mu$ satisfy the weak form of the Ohta-Kawasaki equation \eqref{Eq:Weak}. By multiplying the Faedo--Galerkin system \eqref{Eq:Faedo} by a test function $\eta \in C^\infty_c(0,T)$ and integrating over the time interval $(0,T)$, we find
	\begin{equation} \begin{aligned}
\int_0^T (\p_t \phi^k,h_j)_H \eta(t) \dt   + \int_0^T (m(\phi^k) \nabla \mu^k, \nabla h_j)_H \eta(t) \dt &= \int_0^T (f,h_j)_H \eta(t) \dt,  \\
\int_0^T (\Psi'(\phi^k),h_j)_H \eta(t) \dt +\eps^2 \int_0^T  (\nabla \phi^k, \nabla h_j)_H \eta(t) \dt + \kappa \int_0^T (\nu^k,h_j)_H \eta(t) \dt &= \int_0^T (\mu^k,h_j)_H \eta(t) \dt, \\
\int_0^T (\nabla\nu^k,\nabla h_j)_H \eta(t) \dt   &= \int_0^T (\phi_k-\bar\phi_k,h_j)_H \eta(t) \dt,  \\
\end{aligned} \label{Eq:FaedoTime}
\end{equation}
for all $j \in \{1,\dots,k\}$. We take the limit $k \to \infty$ in these two equations. The linear terms follow directly from the weak/weak-$*$ convergences. It remains to treat the integrals involving the nonlinear functions $m$ and $\Psi'$. Since $m$ and $\Psi'$ are continuous functions, see \cref{Ass:Pos:Mob}, we have by the strong convergence $\phi^k \to \phi$ in $L^2({Q_T})$ also $m(\phi^k) \to m(\phi)$ a.e. in ${Q_T}$. By the boundedness of $m$, we infer from the Lebesgue dominated convergence theorem $m(\phi^k) \nabla h_j \eta \to m(\phi) \nabla h_j \eta$ in $L^2({Q_T})^d$. By the weak convergence of $\nabla \mu^k$, we conclude as $k \to \infty$
$$m(\phi^k) \eta \nabla h_j \cdot \nabla \mu^k \longweak m(\phi) \eta \nabla h_j \cdot \nabla \mu \quad \text{ in } L^1({Q_T}).$$
The continuity of $\Psi'$ gives then $\Psi'(\phi^k) \to \Psi'(\phi)$ a.e. in ${Q_T}$. Further, from the assumption \cref{Ass:Pos:Pot} on the potential function $\Psi$, we infer the bound $$\|\Psi'(\phi^k(t)) \eta(t) h_j\|_{L^1(\Omega)} \leq C \|\eta\|_{L^\infty(0,T)} \cdot \|h_j\|_{H^2(\Omega)} \big( 1+\|\Psi(\phi^k(t))\|_{L^1(\Omega)}\big),$$
for almost every $t \in (0,T)$, and the right-hand side is bounded by the energy estimate \eqref{Eq:FFinalEnergy}. Consequently, the Lebesgue dominated convergence theorem gives for $k \to \infty$
$$\int_0^T (\Psi'(\phi^k),h_j)_H \eta(t) \dt \longrightarrow \int_0^T (\Psi'(\phi),h_j)_H \eta(t) \dt. $$

After having taken care of the nonlinear functions, we are ready to take the limit $k \to \infty$ in the equations \eqref{Eq:FaedoTime} and use the density of $H_k$ in $V$, which yields 
	\begin{equation} \label{Eq:WeakInitial2} \begin{aligned}
\int_0^T \langle \pt \phi,\xi_1 \rangle_V \eta(t) \dt   + \int_0^T (m(\phi) \nabla \mu, \nabla \xi_1)_H \eta(t) \dt &= \int_0^T (f(\phi),\xi_1)_H \eta(t) \dt,  \\
\int_0^T (\Psi(\phi),\xi_2)_H \eta(t) \dt + \int_0^T \eps^2 (\nabla \phi, \nabla \xi_2)_H \eta(t) \dt + \int_0^T \kappa(\nu,\xi_2)_H \eta(t) \dt &= \int_0^T (\mu,\xi_2)_H \eta(t) \dt, \\
\int_0^T (\nabla\nu,\nabla \xi_3)_H \eta(t) \dt   &= \int_0^T (\phi-\bar\phi,\xi_3)_H \eta(t) \dt,
\end{aligned} 
\end{equation}
for all $\xi_1,\xi_2\in V$, $\xi_3 \in \ring V$ and $\eta \in C_c^\infty(0,T)$. Applying the fundamental lemma of calculus of variations, we infer that $(\phi,\mu,\nu)$ satisfies the weak form \eqref{Eq:Weak} of the Ohta--Kawasaki equation. \medskip

\noindent\textbf{Step 5: Initial condition.} From the estimate \eqref{Eq:Energy3} we have $\p_t \phi\in L^2(0,T;V')$. By the continuous embedding $L^2(0,T;V)\cap H^1(0,T;V')\hookrightarrow C^0([0,T];H)$ we have $\phi \in C^0(0,T;H)$. Thus, we have $\phi(0)=\phi_0$ in $H$ by the uniqueness of limits. \medskip

\noindent\textbf{Step 6: Energy inequality.} We prove that the solution tuple $(\phi,\mu,\nu)$ satisfies the energy inequality \eqref{Eq:EnergySolution}. First, we note that norms are weakly/weakly-$*$ lower semi-continuous. Hence, the two norms contained in $\mathcal{E}(\phi_k)$ converge for almost all $t\in(0.T)$. Furthermore we have $\mu^k \rightharpoonup \mu$ in $L^2(0,T;V)$ and therefore for almost every $t\in(0,T)$, we infer
$$\| \mu \|_{L^2(0,t;V)} \leq \liminf_{k \to \infty} \|\mu^k\|_{L^2(0;t;V)}.$$  
We apply the Fatou lemma on the continuous and non-negative function $\Psi$ to obtain
$$\int_\Omega \Psi(\phi) \dx  \leq \liminf_{k \to \infty} \int_\Omega \Psi(\phi^k) \dx.$$ 
For the force term by continuity and boundedness of $f$, we obtain strong convergence $f(\phi^k)\to f(\phi)$ in $L^2(Q_T)$ and by weak convergence we find
\begin{align*}
 \int_0^T (f(\phi_k),\mu_k)_H\eta(t) \dt\longrightarrow \int_0^T (f(\phi),\mu)_H\eta(t)     \dt.
\end{align*}
Consequently, passing to the limit $k \to \infty$ in the discrete energy inequality \eqref{Eq:FFinalEnergy} leads to \eqref{Eq:EnergySolution}.  \medskip

\noindent\textbf{Step 7: Uniqueness.}
We assume the case of constant mobility $m=M>0$ and absence of force $f=0$. 
By a simple computation, it holds $\bar\phi=\bar\phi_0$. Consider two pairs of weak solutions $(\phi_1,\mu_1,\nu_1)$ and $(\phi_2,\mu_2,\nu_2)$, and we denote their differences by $\phi=\phi_1-\phi_2$, $\mu=\mu_1-\mu_2$ and $\nu=\nu_1-\nu_2$. Thus, $\bar\phi=\bar\phi(0)=0$ holds because the same initial is chosen. Each pair fulfills the weak form, and we find for $(\phi,\mu)$
	\begin{equation*} 
	\begin{aligned}
	\langle \p_t \phi,u \rangle_V + M (\nabla \mu,\nabla u)_H &= 0, \\
	 (\Psi'(\phi_1)-\Psi'(\phi_2),v)_H + \eps^2 (\nabla \phi,\nabla v)_H + \kappa(\nu,v)_H & =(\mu,v)_H, \\
    (\nabla\nu,\nabla w)_H &= (\phi-\bar\phi,w)_H
	\end{aligned}
	\end{equation*}
for test functions $u,v \in V$, $w \in \ring V$.
Taking $u=(-\Delta)^{-1} \phi$, $v=M\phi$ and $w=\nu$, yields
	\begin{equation*} 
	\begin{aligned}
	\langle \p_t \phi,(-\Delta)^{-1} \phi\rangle_V + M (\nabla \mu,\nabla (-\Delta)^{-1} \phi)_H &= 0, \\
	 M(\Psi'(\phi_1)-\Psi'(\phi_2),\phi)_H + M\eps^2 (\nabla \phi,\nabla \phi)_H + M\kappa (\nu,\phi)_H & =M(\mu,\phi)_H, \\
     (\nabla\nu,\nabla\nu)_H & = (\phi,\nu)_H
	\end{aligned}
	\end{equation*} 
Exploiting the property $(\nabla \mu, \nabla (-\Delta)^{-1} \phi)_H = (\mu,\phi)_H$ of the Neumann--Laplace operator, gives after adding the equations and cancelling,
\begin{equation} \label{Eq:UniqueDifference}
\begin{aligned}
&(\p_t \phi, \phi)_* + M \eps^2 \|\nabla \phi\|_H^2+M\kappa \|\nabla\nu\|_H^2 = M(\Psi'(\phi_2)-\Psi'(\phi_1),\phi)_H.
\end{aligned}
\end{equation}
Using the $(-C_\Psi)$-convexity of $\Psi$, we have by the mean value theorem
$$(\Psi'(\phi_1)-\Psi'(\phi_2),\phi)_H \geq - C_\Psi \|\phi\|_H^2,
$$
and consequently, we obtain by the $\eps$-Young inequality 
$$\begin{aligned}(\Psi'(\phi_2)-\Psi'(\phi_1),\phi)_H  \leq C_\Psi \|\phi\|_H^2 &= C_\Psi (\nabla (-\Delta)^{-1} \phi,\nabla \phi)_H \leq \frac{\eps^2}{2} \|\nabla \phi\|^2_H + \frac{C_\Psi^2}{2\eps^2} \|\nabla (-\Delta)^{-1} \phi\|_H^2.
\end{aligned}$$
Integration and applying the Gronwall inequality gives
		\begin{equation} \label{Eq:WeakDifference}
	\begin{aligned}
	\frac12 \| \phi(t)\|_*^2 + \frac{M \eps^2}{2} \|\nabla \phi\|_{L^2(H)}^2 \leq C(T) \cdot \|\phi_0 \|_*^2 = 0,
	\end{aligned}
	\end{equation}
hence $\phi_1=\phi_2$ in the sense $\|(\phi_1-\phi_2)(t)\|_{*}=0$ for a.e. $t \in [0,T]$ and $\norm{\phi_1-\phi_2}_{L^2(V)}=\norm{\nu_1-\nu_2}_{L^2(V)}=0$, and consequently $\mu_1=\mu_2$.  \qed

\subsection{Higher spatial regularity} \label{Sub:Reg}
We require higher spatial regularity in the upcoming proof of existence to the Ohta--Kawasaki system with degenerating mobility. Therefore, we state and prove the following theorem: 
\begin{theorem} \label{Thm:Reg}
	Let the assumption of \cref{Thm:PosMob} hold. Then there exists a weak solution $(\phi,\mu,\nu)$ to the Ohta--Kawasaki equation in the sense
	$$\begin{aligned} \pt \phi &= \div(m(\phi) \nabla \mu) +f(\phi) &&\text{ in } L^2(0,T;V'), \\ \mu&=\Psi'(\phi)-\eps^2 \Delta \phi + \kappa\nu &&\text{ a.e. in } {Q_T},  \\
 -\Delta\nu&= \phi-\bar\phi &&\text{ a.e. in } {Q_T},\end{aligned}$$ with the additional regularity $\phi \in L^2(0,T;H^2(\Omega))$ and $\nu \in L^\infty(0,T;H^2(\Omega))$. Moreover, we obtain the additional a priori bound
		\begin{equation} \label{Eq:EnergyExtended} 
	\|\phi\|_{L^2(H^2(\Omega))}^2  \leq  C(T,f,\phi_0). \end{equation}
	Additionally, if $\Psi \in C^2(\R)$ satisfies the growth estimate 
	\begin{equation} \label{Eq:Growth} |\Psi''(x)|\leq C(1+|x|^r) \quad \text{for all } x \in \R \text{ where } \begin{cases} r=\frac{2}{d-2}, &d >2, \\
	r \geq 2, &d=2,
	\end{cases}\end{equation} for some constant $C<\infty$, then it holds $\Psi'(\phi) \in L^2(0,T;V)$ and $\phi \in L^2(0,T;H^3(\Omega))$.
\end{theorem}
\begin{proof}
In the proof of \cref{Thm:PosMob}, we have in the Galerkin setting $\phi^k(t) \in H_k \subset H^2(\Omega)$. 
We take the test function $\Delta \phi^k(t) \in H_k$ in the equation for $\mu^k$, which gives
$$\eps^2 \|\Delta \phi^k\|_H^2 = (\nabla \mu^k,\nabla \phi^k)_H - (\Psi''(\phi^k),|\nabla \phi^k|^2)_H - \kappa (\nabla \nu^k,\nabla \phi^k)_H.$$
Using the additional assumption $\Psi \in C^2(\R)$ it holds by the semiconvexity $\Psi''(x) \geq -C_\Psi$ for all $x \in \R$, and thus, we arrive after integrating from $0$ to $T$ at
$$\begin{aligned} \eps^2 \|\Delta \phi^k\|_{L^2(H)}^2 &\leq \|\mu^k\|_{L^2(V)} \|\phi^k\|_{L^2(V)} + C_\Psi \| \phi^k\|_{L^2(V)}^2 + C \|\nu^k\|_{L^2(V)}\|\phi^k\|_{L^2(V)} \leq C(T,\phi_0).
\end{aligned}$$
Since $(\|\cdot\|_H^2+\|\Delta \cdot\|_H^2)^2$ is an equivalent norm on $H^2(\Omega)$, see \cite[III.Lemma 4.2]{temam2012infinite}, it yields the uniform boundedness of $\phi^k$ in $L^2(0,T;H^2(\Omega))$ and consequently, by the reflexivity of the Hilbert space it holds for the limit $\phi \in L^2(0,T;H^2(\Omega))$.

Similarly, we have $\nu^k(t) \in H_k \subset H^2(\Omega)$ and 
taking the test function $\Delta \nu^k(t) \in H_k$ in the equation for $\nu^k$ yields for a.e. $t \in (0,T)$
$$\|\Delta \nu^k(t)\|_H^2 = (\nabla \nu^k(t),\nabla \phi^k(t))_H \leq \|\nabla \nu^k\|_{L^\infty(H)} \|\nabla \phi^k\|_{L^\infty(H)}\leq C(T,\phi_0),$$
from which we obtain $\nu^k,\nu \in L^\infty(0,T;H^2(\Omega))$.

Inserting $\mu^k = \Pi_k \Psi'(\phi^k)-\eps^2 \Delta \phi^k+\kappa\nu^k$ and $\nu^k=(-\Delta)^{-1}(\phi^k-\bar\phi^k)$ into the equation of $\phi^k$ and considering the Galerkin system 
$$\begin{aligned} &(\pt \phi^k,\xi)_H+(\nabla \Pi_k\Psi'(\phi^k),m(\phi^k)\nabla \xi)_H - (\nabla \Delta \phi^k,m(\phi^k) \nabla \xi)_H + (\nabla (-\Delta)^{-1}(\phi^k-\bar\phi^k),m(\phi^k)\nabla \xi)_H = (f(\phi^k),\xi)_H,
\end{aligned}$$
for all $\xi \in H_k$, and taking the test function $u=-\Delta \phi^k$, we get
$$\begin{aligned} (\pt\nabla\phi^k,\nabla\phi^k)_H  + m_0 \|\nabla\!\Delta \phi^k\|_H^2  &\leq \|f\|_H \|\Delta \phi^k\|_H + m_\infty \|\nabla \Psi'(\phi^k)\|_H  \|\nabla\!\Delta \phi^k\|_H + m_\infty \|\phi^k-\bar\phi^k\|_* \|\nabla\!\Delta\phi^k\|_H \\ &\leq C(T,\phi_0,f) + C\|\nabla \Psi'(\phi^k)\|_H^2 + C\|\phi^k-\bar\phi^k\|_*^2 + \frac{m_0}{2} \|\nabla\!\Delta \phi^k\|_H^2.\end{aligned}$$
By assumption \eqref{Eq:Growth} it holds the growth estimate $|\Psi''(x)| \leq C(1+|x|^r)$ for $r=\frac{2}{d-2}$ for all $x \in \R$ in the case of $d>2$ (for $d=2$ choose any exponent $r \geq 2$ and use the embedding $V \hookrightarrow L^r(\Omega)$).
Therefore, we apply the H\"older and Sobolev inequalities \eqref{Eq:SobolevInequality} to obtain the bound
$$\begin{aligned} \|\nabla \Psi'(\phi^k)\|_H = \|\Psi''(\phi^k) \nabla \phi^k\|_H &\leq \|\Psi''(\phi^k)\|_{L^d(\Omega)} \|\nabla \phi^k\|_{L^{2d/(d-2)}(\Omega)}\lesssim \|1+\phi^k\|_{V}^{2/(d-2)} \|\nabla \phi^k\|_{V}.
\end{aligned}$$
Taking the square on both sides and integrating over the interval $[0,T]$, it yields
$$\|\nabla \Psi'(\phi^k)\|_{L^2(H)}  \lesssim  \|1+\phi^k\|_{L^\infty(V)}^{4/(d-2)} \| \phi^k\|_{L^2(H^2(\Omega))}\leq C(T,\phi_0),$$
and thus, it follows from typical estimates $\nabla\!\Delta \phi^k \in L^2(Q_T)$ and elliptic regularity theory \cite{agmon1959estimates} gives $\phi \in L^2(0,T;H^3(\Omega))$.
\end{proof}
\subsection{Degenerating case} \label{Sec:Deg}

We approximate and extend the mobility function $m \in W^{1,\infty}(-1,1)$ with $m(x)>0$ for all $x\in (-1,1)$ and  $m(\pm 1)=0$ by a strictly positive function $m_\delta$ in the following way:
	$$m_\delta(x) = \begin{cases} m(\delta-1), &\ifs x \leq \delta-1, \\ m(x), &\ifs |x| \leq 1-\delta, \\ m(1-\delta), &\ifs x \geq 1-\delta, \end{cases}$$
where $\delta \in (0,1)$. We extend $m$ by zero outside $[-1,1]$ and denote the extension by $\bar m \in W^{1,\infty}(\R)$. Note that $m_\delta'=\bar m'$ on $[-1+\delta,1-\delta]$. The approximation $m_\delta$ is positive and admits regularity in $W^{1,\infty}(\R)$ with the upper and lower bounds (for $\delta$ sufficiently small)
	$$0<\min\{m(-1+\delta),m(1-\delta)\} \leq m_\delta(x) \leq \max_{y \in [-1,1]} m(y)  \quad \forall x \in \R.$$
	Further, we consider the potential $\Psi:(-1,1) \to \R_{\geq 0}$ and assume the splitting $\Psi=\Psi_1+\Psi_2$ with $\Psi_1 \in C^2(-1,1)$ convex and $\Psi_2 \in C^2([-1,1])$ being $(-C_\Psi)$-convex. We define its regularization $\Psi_\delta : \R \to \R$ as $\Psi_\delta = \Psi_{1,\delta} + \bar\Psi_2$ where $\Psi_{1,\delta}\in C^2(\R)$ is the unique function with $\Psi_{1,\delta}(0)=\Psi_1(0)$, $\Psi_{1,\delta}'(0)=\Psi_1'(0)$, and 
		$$(\Psi_{1,\delta})''(x) = \begin{cases} (\Psi_1)''(\delta-1), &\ifs x \leq \delta-1, \\ (\Psi_1)''(x), &\ifs |x| \leq 1-\delta, \\ (\Psi_1)''(1-\delta), &\ifs x \geq 1-\delta. \end{cases}$$
		In particular, $\Psi_{1,\delta}$ is convex on $\R$ since $\Psi_1$ itself is assumed to be convex on $(-1,1)$. Moreover, we introduce the extension $\bar \Psi_2\in C^2(\R)$ of $\Psi_2$ to the reals by setting
		$$\bar\Psi_2(x)=\begin{cases} \Psi_2(-1)+\Psi_2'(-1)(x+1) + \frac12 \Psi_2''(-1) (x+1)^2, &\ifs x < -1, \\
		\Psi_2(x), &\ifs |x|\leq 1, \\
		\Psi_2(1)+\Psi_2'(1)(x-1) + \frac12 \Psi_2''(1) (x-1)^2, &\ifs x >1. 
		\end{cases}$$
		It holds $\|\bar{\Psi}_2''\|_{C(\R)} \leq \| \Psi_2''\|_{C([-1,1])} \leq C$ and $\Psi''_\delta(x) \geq - C_\Psi$ for all $x \in \R$. By definition we have $\Psi_\delta=\Psi$ and $m_\delta = m$ on the interval $[-1+\delta,1-\delta]$ for $\delta \in (0,1)$. \medskip
	 
	\noindent\textbf{Auxiliary problem.}
	After having defined the approximations $m_\delta$ and $\Psi_\delta$, we consider the auxiliary problem
	\begin{equation}\begin{aligned}
	\partial_t \phi_\delta  &=  \div ( m_\delta(\phi_\delta) \nabla \mu_\delta)+f(\phi_\delta),  \\
	\mu_\delta &= \Psi_\delta'(\phi_\delta) - \eps^2 \Delta \phi_\delta + \kappa \nu_\delta, \\
    -\Delta\nu_\delta &= \phi_\delta-\bar\phi_\delta
	\end{aligned} \label{Eq:CHdelta} \end{equation}
	with initial data $\phi_{\delta,0}=\phi_0 \in (-1,1)$ a.e. in $\Omega$, which has a weak solution $(\phi_\delta,\mu_\delta,\nu_\delta)$ according to the result from before and it satisfies
	\begin{equation}\begin{aligned}
	\langle \partial_t \phi_\delta,\xi_1\rangle_V  &=  -(m_\delta(\phi_\delta) \nabla \mu_\delta,\nabla \xi_1)_H + (f(\phi_\delta),\xi_1)_H,  \\
	(\mu_\delta,\xi_2)_H &= (\Psi_\delta'(\phi_\delta),\xi_2)_H - \eps^2 (\Delta \phi_\delta,\xi_2)_H + \kappa (\nu,\xi_2)_H, \\
    -(\Delta\nu_\delta,\xi_3)_H &= (\phi_\delta-\bar\phi_\delta,\xi_3)_H
	\end{aligned} \label{Eq:CHdeltaWeak} \end{equation}
	for all $\xi_1\in V$, $\xi_2,\xi_3 \in H$. Indeed, the mobility $m_\delta$ is positive, continuous and bounded, and the potential $\Psi_\delta=\Psi_{1,\delta}+\bar \Psi_2$ is $(-C_\Psi)$-convex as discussed before and fulfils the required growth estimates due to the definitions of $\Psi_{1,\delta}$ and $\bar \Psi_2$.

	We multiply the variational form by a smooth test function $\eta \in C_c^\infty(0,T)$ and exploit the density of the tensor space $C_c^\infty(0,T) \otimes V$ in $L^2(0,T;V)$ (and analogously for $H$) to formulate the weak form in terms of time-dependent test functions, i.e., 
	\begin{subequations}\begin{align}
	\label[equation]{Eq:CHdeltaWeakTimePhi} \int_0^T \langle \partial_t \phi_\delta,\xi_1\rangle_V \dt  &=  -\int_0^T (m_\delta(\phi_\delta) \nabla \mu_\delta,\nabla \xi_1)_H \dt + (f(\phi_\delta),\xi_1)_H,  \\
	\int_0^T (\mu_\delta,\xi_2)_H \dt &= \int_0^T (\Psi_\delta'(\phi_\delta),\xi_2)_H - \eps^2 (\Delta \phi_\delta,\xi_2)_H +\kappa(\nu_\delta,\xi_2)_H \dt,
	\label[equation]{Eq:CHdeltaWeakTimeMu}\\
    -\int_0^T(\Delta\nu_\delta,\xi_3)_H \dt &=  \int_0^T(\phi_\delta-\bar\phi_\delta,\xi_3)_H \dt
    \label[equation]{Eq:CHdeltaWeakTimenu}
	\end{align}\label[equation]{Eq:CHdeltaWeakTime}\end{subequations}
	for all $\xi_1 \in L^2(0,T;V)$ and $\xi_1, \xi_2 \in L^2(Q_T)$.
	We derive $\delta$-uniform estimates and pass to the limit $\delta \to 0$. 
	
	We make the following assumptions for the following proofs.
	
	\begin{assumption} \label{As:Deg} ~\\[-0.4cm]
	\begin{enumerate}[label=(B\arabic*), ref=B\arabic*, leftmargin=.9cm] \itemsep.1em
    \item $\Omega \subset \R^d$ bounded $C^{1,1}$-domain with $d \geq 2$, $T>0$ finite time horizon, $Q_T=(0,T)\times \Omega$.
    \item $\phi_0 \in V$ with $\Psi(\phi_0) \in L^1(\Omega)$, $\Phi(\phi_0) \in L^1(\Omega)$ (see \cref{Lem:DegMobEst}), and $|\phi_0(x)|\leq 1$ for a.e. $x \in \Omega$. \label{Ass:Deg:Ini}
    \item $\Psi=\Psi_1+\Psi_2$ with $\Psi_1 \in C^2(-1,1)$ convex and $\Psi_2 \in C^2([-1,1])$ is $(-C_\Psi)$-convex for some $C_\Psi<\infty$. \label{Ass:Deg:Pot}
    \item $m \in W^{1,\infty}(-1,1)$ such that $m(x) >0$ for all $x\in (-1,1)$, $m(\pm 1)=0$, and $m\Psi'' \in C^0([-1,1])$.\label{Ass:Deg:Mob}
    \item $f\in C^0(\R,\R_{\geq 0})$ and $f(x)=0$ for $x \in \R\setminus (-1,1)$, and for $x \in [-1,1]$ we require
    \begin{equation*}
        1-|x| \lesssim f(x) \lesssim \begin{cases}
            |1-x| \quad x\geq 0,\\
            |1+x| \quad x\leq 0.
        \end{cases}
    \end{equation*} \label{Ass:Deg:f1}
    \item In the case $f(x)\neq 0$ we require $\Psi_1''(x)=\max\{0,1-x^2\}^{-p}F(x)$ and $m(x)=\max\{0,1-x^2\}^{q}m_0(x)$ for $2 \geq q \geq p \geq 0$. The functions $F$ and $m_0$ are positive bounded by $F_0$. \label{Ass:Deg:f2}
\end{enumerate}
	\end{assumption}

We note that the function $f(x)=\max\{0,1-x^2\}$ is the typical example that fulfills the assumptions.
	Further, we make the following remarks regarding \cref{As:Deg}. \medskip
	\begin{enumerate} 
	\item We assume in \cref{Ass:Deg:Mob} a mobility that degenerates at $\pm 1$. For the general case of degeneracy at points $a,b\in \R$, one has to shift the interval $[-1,1]$ by an operator $A:[-1,1] \to [a,b]$, see \cite{abels2013incompressible}. We assume that mobility compensates for an eventual blow-up of $\Psi''$ at $\pm 1$ by assuming $m\Psi'' \in C^0([-1,1])$ in \cref{Ass:Deg:Mob}. For example, the Flory--Huggins potential has the second derivative $\Psi''(x)=\theta/(1-x^2)-\theta_0$ for $x \in (-1,1)$ and therefore, degenerates at $x=\pm 1$. Then with the typical mobility $m(x)=(1-x^2)^\ell$, $\ell\geq 1$, $x \in [-1,1]$, we have indeed $m \Psi'' \in C^0([-1,1])$. For the double-obstacle potential one chooses $\Psi_1=0$ and $\Psi_2(x)=1-x^2$, since $\Psi_1$ does not have to be defined on the boundary $\pm 1$ in the assumption \cref{Ass:Deg:Pot} of \cref{Thm:DegMob}.\medskip
	
	\item We remark that we assume $|\phi_0(x)|\leq 1$ a.e. in $\Omega$ in \cref{Ass:Deg:Ini} instead of excluding the values $\pm 1$ to guarantee $\Psi_\delta(\phi_0)=\Psi(\phi_0)$ for $\delta$ sufficiently small. We use the same argument as in \cite{abels2013incompressible}. The assumption $|\phi_0(x)| \leq 1$ a.e. implies $\bar \phi_0 \in [-1,1]$. In the case of $\bar \phi_0=\pm 1$ it holds $\phi_0 = \pm 1$ a.e. in $\Omega$, which readily gives the existence of a weak solution $(\phi,J,\nu)=(\pm 1,0,0)$. Therefore, in the proof, we solely consider the case $|\phi_0|<1$ almost everywhere.\medskip
	\end{enumerate}

	We formulate and prove two lemmata, which will be needed in the proof of existence theorem in case of a degenerate mobility. First, we will derive an estimate of the so-called entropy function $\Phi$. Second, we will prove a key inequality that allows us to achieve the result $\phi(t,x) \in [-1,1]$ for a.e. $(t,x) \in Q_T$.

		\begin{lemma} \label{Lem:DegMobEst} Let \cref{As:Deg} hold. Further, let $\Phi:(-1,1) \to \R_{> 0}$ be the unique function, which is given by $\Phi''(x)=1/m(x)$, $\Phi'(0)=\Phi(0)=0$. Further, its approximation $\Phi_\delta:\R \to \R_{>0}$ is defined by $\Phi''_\delta(x)=1/m_\delta(x)$ and $\Phi_\delta'(0)=\Phi_\delta(0)=0$. Then the following $\delta$-uniform bounds hold
	\begin{equation}\begin{aligned} 
        &\|\Phi_\delta(\phi_\delta)\|_{L^\infty(L^1(\Omega))} + \norm{\phi_\delta}_{L^\infty(V)}^2 + \norm{\nu_\delta}_{L^\infty(V)}^2 + \norm{\Psi_\delta(\phi_\delta)}_{L^\infty(L^1(\Omega))}+ \|\Delta \phi_\delta\|_{L^2(H)}^2  \\
        & + \norm{\sqrt{m_\delta(\phi_\delta)}\nabla\mu_\delta}_{L^2(H)}^2+ \big( \Psi_{\delta}''(\phi_\delta),|\nabla\phi_\delta|^2\big)_{L^2(H)} \leq C(T,\phi_0).  
	\end{aligned} \label{Eq:EnergyH2} \end{equation}
	\end{lemma}
	
	\begin{proof}
    \noindent\textbf{Step 1: Energy bounds.} 
    Using the energy inequality \cref{Eq:EnergyIneq} at $t=T$, we find 
    \begin{align*}
  \mathcal{E}(\phi_\delta(T)) + \int_0^T m_\delta(\phi_\delta)|\nabla\mu_\delta| \ds \leq \mathcal{E}(\phi_\delta(0))  + \int_0^T(f(\phi_\delta),\mu_\delta)_H  \ds
 \end{align*}
    Hence, we begin by estimating $(f(\phi_\delta),\mu_\delta)_H$ suitably. Using $\xi_2=f(\phi_\delta)$ as a test function in \cref{Eq:CHdeltaWeakTimeMu}, we find
    \begin{equation*}
        \int_0^T(\mu_\delta,f(\phi_\delta))_H \ds = \int_0^T (\Psi_\delta'(\phi_\delta),f(\phi_\delta))_H -\eps^2 (\Delta\phi_\delta,f(\phi_\delta))_H + \kappa(\nu_\delta,f(\phi_\delta))_H \ds.
    \end{equation*}
    Regarding the first inner product on the right-hand side, we use the splitting $\Psi_\delta=\Psi_{1,\delta} + \bar\Psi_2$, see \cref{Ass:Deg:Pot}, and apply the H\"older's and Young's inequalities, which yields
    \begin{equation*}
    \int_0^T(\mu_\delta,f(\phi_\delta))_H \ds \leq \delta\norm{\Delta\phi_\delta}_{L^2(H)}^2 + \delta\|\nu_\delta\|_{L^2(H)}^2 + C + \int_0^T (\Psi_{1,\delta}'(\phi_\delta),f(\phi_\delta))_H\ds
    \end{equation*}

    For estimating the final inner product, we consider several cases and rely on the computations in \cite{Ebenbeck21} and it is sufficient to consider $p=q=2$ so the most singular case. In this case, the singular behavior of $\Psi_{1,\delta}'$ is the same as $\Phi_{1,\delta}'$. Before proceeding, we will derive a representation of $\Psi_{1,\delta}'$. Using the definition of $\Psi_{1,\delta}''$ and the continuity and normalization assumption, i.e. $\Psi_{1,\delta}'\in C^1(\R)$ and $\Psi_{1,\delta}'(0)=\Psi_{1}'(0)=0$ we find
    \begin{align*}
      \Psi_{1,\delta}(x)' = \begin{cases}
        \int_{-1+\delta}^x \Psi_1''(-1+\delta)\dy  - \int_{-1+\delta}^0 \Psi_1''(y)\dy, &\qquad  x \in (-\infty,-1+\delta], \\
        \int_{-1+\delta}^x \Psi''_1(y) \dy - \int_{-1+\delta}^0 \Psi''(y) \dy, &\qquad  x\in(-1+\delta,1-\delta)\\
        \int_{1-\delta}^x \Psi_1''(1-\delta)\dy  + \int_{0}^{1-\delta} \Psi_1''(y)\dy, &\qquad  x \in [1-\delta,+\infty)
      \end{cases}  
    \end{align*}

In the following, we will consider upper bounds for $|f(x)\Psi_{1,\delta}'(x)|$ for different cases of $x$.
    \begin{enumerate} \itemsep.3em
        \item In the cases $x\leq -1$ and $x\geq 1$, we find $\Psi_{1,\delta}'(x)f(x)=0$ since $f(x)=0$, see \cref{Ass:Deg:f1}.
        \item In the cases $x\in(-1,-1+\delta)$ and $x\in(1-\delta,-1)$, we exploit symmetry and  only consider $x\in(-1,-1+\delta)$, since $f(s)\geq0$ and $|\Psi_{1,\delta}'(s)|=|\Psi_{1,\delta}'(-s)|$ for $s\leq -1+\delta$. Using the representation formula and Assumption \cref{Ass:Deg:f2} we estimate
    \begin{equation*}
     |f(x)\Psi_{1,\delta}'(x)| \leq \delta\Big| \int_x^{-1+\delta}\frac{1}{\delta(2-\delta)}\dy + \int_{-1+\delta}^0 \frac{1}{1-y^2} \dy \Big|  \leq 1+|x|.
    \end{equation*}
    \item In the cases $x\in[-1+\delta,0]$ and $[0,1-\delta]$, symmetry allows us to consider only the case $x\in[-1+\delta,0]$ and we estimate
    \begin{align*}
     |f(x)\Psi_{1,\delta}'(x)| \leq C f(x)\Big|  \int_{0}^x \frac{1}{1-y^2} \dy\Big|\leq C.   
    \end{align*}
    The last estimate holds, since $\frac{f(x)}{|x+1|}\leq C$ for $x\in[-1+\delta,0]$ and especially for $x=-1+\delta$. \smallskip
    \end{enumerate}

    Hence, in total, we obtain the estimate
    \begin{equation*}
      \int_0^T (\Psi_{1,\delta}'(\phi_\delta),f(\phi_\delta))_H \ds \leq C(1+\norm{\phi_\delta}_H).  
    \end{equation*}

\noindent\textbf{Step 2: $L^2(\Omega)$-estimate.}
Inserting the test function $\xi_1=\phi_\delta$ in \cref{Eq:CHdeltaWeakTimePhi}, we directly obtain
\begin{align*}
    \norm{\phi_\delta(T)}_H^2 + \leq \norm{\phi_\delta(0)}_H^2 + \tfrac{1}{2}\norm{\sqrt{m_\delta(\phi_\delta)}\nabla\mu_\delta}_H^2 + C\norm{\phi_\delta}_V^2 + C\norm{f(\phi_\delta)}_H^2.
\end{align*}

    \noindent\textbf{Step 3: Entropy bounds.} Since it holds $\Phi_\delta''\in L^\infty(\R)$ by the boundedness of $m_\delta$, we have $\Phi_\delta' \in C^{0,1}(\R)$ and $\Phi_\delta'(\phi_\delta) \in L^2(0,T;V)$. Moreover, $\Phi_\delta$ is a convex and non-negative functional due to $\Phi_\delta''(x)>0$ for all $x \in \R$. Thus, we can write $$\Phi_\delta(x)=\int_0^x \int_0^y \frac{1}{m_\delta(z)} \,\text{d}z \text{d}y.$$ Using the chain rule, we can compute
	$$\nabla \Phi'_\delta(\phi_\delta)=\Phi_\delta''(\phi_\delta) \nabla \phi_\delta = \frac{\nabla\phi_\delta}{m_\delta(\phi_\delta)} \in H,$$
	and thus, $\xi=\Phi_\delta'(\phi_\delta)\in V$ is a valid test function in the  variational  formulation
	\begin{equation}  \label{Eq:WeakDeg1} \begin{aligned}   \langle \p_t \phi_\delta, \xi \rangle_V &= ( \mu_\delta , \div(m_\delta(\phi_\delta) \nabla \xi))_H + (f,\xi)_H \\ &=  (-\eps^2 \Delta \phi_\delta + \Psi_\delta'(\phi_\delta) + \kappa \nu_\delta, \div (m_\delta(\phi_\delta) \nabla \xi))_H + (f(\phi_\delta),\xi)_H,
 \end{aligned}
	\end{equation}
	for all $\xi \in V$. Thus, it gives
	\begin{equation*} 
	\begin{aligned}
	\langle \p_t \phi_\delta, \Phi_\delta'(\phi_\delta) \rangle_V &=  \big(-\eps^2 \Delta \phi_\delta + \Psi_\delta'(\phi_\delta) + \kappa \nu_\delta,\div (m_\delta(\phi_\delta) \nabla \Phi_\delta'(\phi_\delta))\big)_H + (f(\phi_\delta),\Phi_\delta'(\phi_\delta))_H\\
	&= -\eps^2  \|\Delta \phi_\delta\|_H^2  - (\nabla \Psi_\delta'(\phi_\delta),  \nabla\phi_\delta)_H + \kappa (\nu_\delta,\Delta \phi_\delta)_H + (f(\phi_\delta),\Phi_\delta'(\phi_\delta))_H \\
	&= -\eps^2  \|\Delta \phi_\delta\|_H^2  - (\Psi_\delta''(\phi_\delta) ,|\nabla\phi_\delta|^2)_H - \kappa (\nabla\nu_\delta,\nabla(\phi_\delta-\bar\phi_\delta))_H + (f(\phi_\delta),\Phi_\delta'(\phi_\delta))_H.
	\end{aligned}
	\end{equation*} 
We integrate both sides, apply the chain rule and use the test function $\xi_3=\phi_\delta-\bar\phi_\delta$ in \cref{Eq:CHdeltaWeakTimenu}, which yields
	\begin{equation*} 
	\begin{aligned}
&\|\Phi_\delta(\phi_\delta(t))\|_{L^1(\Omega)}
	+\eps^2 \|\Delta \phi_\delta\|_{L^2(H)}^2   + \int_0^t (\Psi_\delta''(\phi_\delta) ,|\nabla\phi_\delta|^2)_H\ds + \kappa\|\phi_\delta-\bar\phi_0\|_{L^2(H)}^2 \\ &= \|\Phi_\delta(\phi_0)\|_{L^1(\Omega)} + \int_0^t (f(\phi_\delta),\Phi_\delta'(\phi_\delta))_H \ds
	\end{aligned}
	\end{equation*} 
The remaining term on the right-hand side involving the force $f$ can be treated as in the case of the energy estimate since $\Phi'_\delta$ and $\Psi'_\delta$ have the same singular behavior. The combination of all estimates yields
 \begin{align*}
  &\mathcal{E}(\phi_\delta(T)) + \frac{1}{2}\norm{\phi_\delta(T)}_H^2 + \|\Phi_\delta(\phi_\delta(t))\|_{L^1(\Omega)} + \frac{1}{2}\int_0^T m_\delta(\phi_\delta(s))|\nabla\mu_\delta(s)|^2  + \eps^2|\Delta\phi(s)|^2 \ds \\
  &\leq \mathcal{E}(\phi_\delta(0)) + \|\Phi_\delta(\phi_\delta(0))\|_{L^1(\Omega)} + C\int_0^T 1 + \norm{f(\phi_\delta)}_H^2 + \norm{\phi_\delta}_V^2 \ds. 
 \end{align*}
Using the boundedness of $f$, see \cref{Ass:Deg:f1}, we obtain the a priori estimates by a Gronwall argument. Finally, we note that the property $\Phi_\delta(\phi_0) \leq \Phi(\phi_0)$ a.e. follows due to $m_\delta(\phi_0) \geq m(\phi_0)$ a.e., which gives the desired $\delta$-uniform bound.
	\end{proof}
	
	\begin{lemma} \label{Lem:DegMobEst2} Let \cref{As:Deg} hold. Then it yields	$\|(|\phi_\delta|-1)_+\|_{L^\infty(H)} \leq C \sqrt{\delta}$ where $x_+=\max\{0,x\}$.
	\end{lemma}
		
	\begin{proof}
	Using straightforward computations, we derive for all $x>1$ and $\delta \in (0,1)$ the following lower estimate
	$$\begin{aligned} \Phi_\delta(x) &= \Phi_\delta(1-\delta) + \Phi_\delta'(1-\delta) (x-(1-\delta))+\frac12 \Phi_\delta''(1-\delta) (x-(1-\delta))^2 \\
	&\geq \frac12 \Phi_\delta''(1-\delta) (x-1+\delta)^2 =  \frac{(x-1+\delta)^2}{2m_\delta(1-\delta)}  \geq \frac{(x-1)^2 }{2m_\delta(1-\delta)},
	\end{aligned}$$
    and analogously, it holds $\Phi_\delta(x) \geq  \frac{(x+1)^2}{2m_\delta(\delta-1)}$ for $x<-1$.   Combining these two results gives
    \begin{equation} \label{Eq:LemDeriv} (|x|-1)_+^2 \leq 2 \Phi_\delta(x) \max\{m_\delta(1-\delta),m_\delta(\delta-1)\}, \medskip
    \end{equation}
    for all $x \in \R$. However, we also have $m_\delta(1-\delta)=m(1-\delta)$ and $m(1)=0$, which implies by the mean value theorem
    $$|m_\delta(1-\delta)| = |m(1-\delta)-m(1)| \leq \delta \|m'\|_{L^\infty(-1,1)}, \medskip$$
    and analogously, it holds $|m_\delta(\delta-1)| \leq \delta\|m'\|_{L^\infty(-1,1)}$. 
    Hence, using \cref{Eq:LemDeriv}  we have
	\begin{equation*}\int_\Omega (|\phi_\delta|-1)_+^2 \dx \leq  2 \delta \|m'\|_{L^\infty(-1,1)} \int_\Omega \Phi_\delta(\phi_\delta) \dx,
	\end{equation*}
	a.e. in $(0,T)$. Thus, it yields the desired result after applying the bound of \cref{Lem:DegMobEst}.
	\end{proof}
	
	\noindent 
	Having proved the two lemmata, we are now ready to state and prove the existence theorem in the case of a degenerate mobility.
	\begin{theorem} \label{Thm:DegMob}
	Let \cref{As:Deg} hold. Then there exists a weak solution $(\phi,J)$ with
	$$\begin{aligned}
	\phi &\in H^1(0,T;V') \cap L^\infty(0,T;V) \cap L^2(0,T;H^2(\Omega)) \text{ with } |\phi| \leq 1 \text{ a.e. in } Q_T,\\ 
	J &\in L^2(Q_T)^d, \\ \nu &\in L^\infty(0,T;\mathring{V}),
	\end{aligned}$$
	to the Ohta--Kawasaki equation in the sense that
\begin{subequations}
	\begin{align}\label[equation]{Eq:WeakDegPhi}
	\langle \p_t \phi,\xi_1 \rangle_{L^2(V)}  &=  (J,\nabla \xi_1)_{L^2(H)} + (f(\phi),\xi_1)_{L^2(H)}, \\
	(J,\xi_2)_{L^2(H)}  &= -(\Psi'(\phi) - \eps^2 \Delta \phi + \kappa\nu,\div(m(\phi)  \xi_2))_{L^2(H)}, 
    \label[equation]{Eq:WeakDegJ}\\
    (\nabla\nu,\nabla\xi_3)_{L^2(H)} &= (\phi-\bar\phi,\xi_3)_{L^2(H)} \label[equation]{Eq:WeakDegnu}
	\end{align} \label[equation]{Eq:WeakDeg}
	\end{subequations}
	for all $\xi_1,\xi_3 \in L^2(0,T;V),\xi_2 \in L^2(0,T;V^d) \cap L^\infty(Q_T)^d$ with $\xi \cdot n_\Omega = 0$ on $\p \Omega \times (0,T)$.
	\end{theorem}

		We note that \cref{Thm:DegMob} does not state the existence of a tuple $(\phi,\mu,\nu)$ but instead $(\phi,J,\nu)$. This is due to the low regularity of $\mu$ in the degenerate case. In the weak form with the mass flux $J$, all the terms are well-defined.	
	\begin{proof}
	We consider a weak solution $(\phi_\delta,\mu_\delta,\nu_\delta)$ to \cref{Eq:CHdelta} which exists by \cref{Thm:PosMob} and fulfills the $\delta$-uniform bounds \cref{Eq:EnergyH2}.
	 Hence, there are functions $(\phi,J,\tilde J)$ such that
	$$\begin{aligned}	
	\phi_\delta &\longweak \phi &&\text{\rlap{weakly-$*$ }\phantom{strongly} ~in } L^\infty(0,T;V), \\
	\p_t \phi_\delta &\longweak \p_t \phi &&\text{\rlap{weakly}\phantom{strongly} ~in } L^2(0,T;V'), \\
	\phi_\delta &\longrightarrow \phi &&\text{strongly ~in }  L^2(Q_T), \\
		\tilde J_\delta = -\sqrt{m_\delta(\phi_\delta)} \nabla  \mu_\delta &\longweak \tilde J &&\text{\rlap{weakly}\phantom{strongly} ~in }
 L^2(Q_T)^d, \\
J_\delta = -m_\delta(\phi_\delta) \nabla  \mu_\delta &\longweak J &&\text{\rlap{weakly}\phantom{strongly} ~in }
 L^2(Q_T)^d, \\
 \nu_\delta &\longweak \nu &&\text{\rlap{weakly}\phantom{strongly} ~in }
 L^2(V),
	\end{aligned}$$
	as $\delta \to 0$. Here, we used the estimate $$\|J_\delta\|_{L^2(H)}^2 = \|\sqrt{m_\delta(\phi_\delta)} \tilde J_\delta\|_{L^2(H)}^2 \leq \|m_\delta\|_{L^\infty(\R)} \|\tilde J_\delta\|_{L^2(H)}^2 \leq C(T,\phi_0). $$ 
	Due to the higher spatial regularity, see \cref{Thm:Reg} and the improved energy inequality \cref{Eq:EnergyH2}, we have
	\begin{equation} \label{Eq:Weak_Deg} \begin{aligned}\phi_\delta &\longweak \phi &&\text{\rlap{weakly}\phantom{strongly} in } L^2(0,T;H^2(\Omega)), \\
	\phi_\delta &\longrightarrow \phi &&\text{strongly in } L^2(0,T;V),
	\end{aligned}
	\end{equation}
	where we employed the compact embedding $$H^1(0,T;V') \cap  L^2(0,T;H^2(\Omega)) \com L^2(0,T;V).$$ for the Gelfand triple $H^2(\Omega) \com V  \con V'.$ Moreover, using lower semicontinuity and passing to the limit $\delta \to 0$ in $\int_\Omega (|\phi_\delta|-1)_+^2 \dx \leq C\delta$, see \cref{Lem:DegMobEst2}, gives $|\phi(t,x)|\leq 1$ for a.e. $(t,x) \in Q_T$.
	
	We take the limit $\delta \to 0$ in the weak form \cref{Eq:CHdeltaWeakTime} of the solution $(\phi_\delta,\mu_\delta,\nu_\delta)$ and use the weak and strong convergences resulting in
		$$\begin{aligned} \int_0^T \langle \p_t \phi, \xi_1 \rangle_V \dt &= \int_0^T (J,\nabla \xi_1)_H + (f(\phi),\xi_1)\dt, \\
		\int_0^T (\mu,\xi_2)_H \dt &= \int_0^T (\Psi'(\phi) - \eps^2 \Delta \phi + \kappa \nu, \xi_2)_H \dt, \\
        \int_0^T (\nabla\nu,\xi_3)_H \dt &= \int_0^T (\phi-\bar\phi,\xi_3)_H \dt, 
		\end{aligned}$$
		for all $\xi_1,\xi_3\in L^2(0,T;V)$ and $\xi_2 \in  L^2(Q_T)$. 
  
  It remains to show that it holds
		$$J=-m(\phi) \nabla( \Psi'(\phi) -\eps^2 \Delta \phi+\kappa\nu) ,$$
		in the sense of the weak form \cref{Eq:WeakDegJ}.
We take the test function $\zeta = \div(m_\delta(\phi_\delta) \xi)$ in \cref{Eq:CHdeltaWeakTimeMu} for any $\xi \in L^2(0,T;V^d) \cap L^\infty(Q_T)^d$ with $\xi \cdot n_\Omega =0$ on $\partial \Omega \times (0,T)$. Indeed, the test function is well defined due to	$$\begin{aligned}\|\!\div(m_\delta(\phi_\delta) \xi)\|_{L^2(H)} &\leq \| m_\delta'(\phi_\delta) \nabla \phi_\delta \cdot \xi\|_{L^2(H)} + \|m_\delta(\phi_\delta) \div \xi\|_{L^2(H)} \\
			&\leq \|m_\delta'\|_{L^\infty(\R)} \|\nabla \phi_\delta\|_{L^2(H)} \|\xi\|_{L^\infty(Q_T)} + \|m_\delta\|_{L^\infty(\R)} \|\xi\|_{L^2(V)}. \end{aligned} $$
			Then we have after integration by parts
	\begin{equation} \label{Eq:TestMuDeg}\begin{aligned}
	     -\int_0^T \! (\nabla \mu_\delta,m_\delta(\phi_\delta) \xi)_H \dt  &= \int_0^T\! (\mu_\delta, \div(m_\delta(\phi_\delta) \xi))_H \dt = \int_0^T \!  (\Psi_\delta'(\phi_\delta)-\eps^2\Delta \phi_\delta+\kappa\nu, \div(m_\delta(\phi_\delta) \xi))_H \dt. 
      \end{aligned}
	\end{equation}
	The left-hand side of this equation is equal to $\int_0^T (J_\delta,\xi)_H \dt$ and converges to $\int_0^T (J,\xi)_H \dt$ for all $\xi$ as $\delta \to 0$ due to the weak convergence of $J_\delta$. Hence, we also take the limit $\delta \to 0$ in the right-hand side to match the weak form of $J$, i.e., we have to show
	$$\begin{aligned} &\int_0^T  (\Psi_\delta'(\phi_\delta)-\eps^2\Delta \phi_\delta+\kappa\nu_\delta, \div(m_\delta(\phi_\delta) \xi))_H \dt \longrightarrow \int_0^T (\Psi'(\phi) - \eps^2 \Delta \phi + \kappa\nu,\div(m(\phi) \xi))_H \dt,
 \end{aligned}$$
	as $\delta \to 0$.
	To do so, we rewrite the term on the left-hand side as
		\begin{equation}\begin{aligned}   &\int_{Q_T} \Psi_{1,\delta}'(\phi_\delta) \div(m_\delta(\phi_\delta) \xi) + \Psi_2'(\phi_\delta)\div(m_\delta(\phi_\delta) \xi)  -\eps^2 \Delta \phi_\delta \div(m_\delta(\phi_\delta) \xi) +\kappa\nu_\delta\div(m_\delta(\phi_\delta) \xi)\dd(t,x),	\end{aligned} \label{Eq:Deg3Terms}\end{equation}
		and take the limit $\delta \to 0$ in each of the four terms. \medskip
	\begin{itemize} \itemsep0.3em
	    \item 
	We begin with the second, third and fourth term. We split them once more by employing the weak product rule on $\div(m_\delta(\phi_\delta)\xi)$. We have proven $\phi_\delta \to \phi$ a.e. in $Q_T$ and $m_\delta \to \bar m$ uniformly as $\delta \to 0$ since
	$$|m_\delta(x)-\bar m(x)| \leq \max\{m(1-\delta),m(\delta-1)\} \longrightarrow 0,$$
	for all $x \in \R$ as $\delta \to 0$; note that $\bar m(\phi)=m(\phi)$ due to $|\phi|\leq 1$.
	Moreover, $\Psi_2'(\phi_\delta) \to \Psi_2'(\phi)$ a.e. by the continuity of $\Psi_2'$, $\Delta \phi_\delta \rightharpoonup \Delta \phi$ weakly in $L^2(Q_T)$ by \cref{Eq:Weak_Deg} and thus, we conclude by the Lebesgue dominated convergence theorem
	$$\begin{aligned}
	\int_{Q_T} \Psi_2'(\phi_\delta) m_\delta(\phi_\delta) \div \xi \dd(t,x) &\longrightarrow \int_{Q_T} \Psi_2'(\phi) m(\phi) \div \xi \dd(t,x),
	\\
	\int_{Q_T} \Delta \phi_\delta m_\delta(\phi_\delta) \div\xi \dd(t,x) &\longrightarrow \int_{Q_T}  \Delta \phi m(\phi) \div\xi \dd(t,x), \\
	\int_{Q_T} \nu_\delta m_\delta(\phi_\delta) \div\xi \dd(t,x) &\longrightarrow \int_{Q_T}  \nu m(\phi) \div\xi \dd(t,x).
	\end{aligned}$$
	
	\item Next, we treat the other parts of the product formula, that is, we have to pass to the limit in the terms involving $m_\delta'(\phi_\delta) \nabla \phi_\delta$. 
	Since $m_\delta' \to m_\delta$ uniformly as $\delta \to 0$, we have by the dominated convergence theorem $m_\delta'(\phi_\delta) \nabla \phi_\delta \to m'(\phi) \nabla \phi$ in $L^2(Q_T)^d$ due to the strong convergence of $\nabla \phi_\delta$. Thus, we have as $\delta \to 0$
		$$\begin{aligned}
	\int_{Q_T} \Psi_2'(\phi_\delta) m_\delta'(\phi_\delta) \nabla \phi_\delta \cdot \xi \dd(t,x) &\longrightarrow \int_{Q_T} \Psi_2'(\phi) m'(\phi) \nabla \phi \cdot \xi \dd(t,x),
	\\
	\int_{Q_T}  \Delta \phi_\delta m_\delta'(\phi_\delta) \nabla \phi_\delta \cdot \xi \dd(t,x) &\longrightarrow \int_{Q_T}  \Delta \phi m'(\phi) \nabla \phi \cdot \xi \dd(t,x),
 \\
	\int_{Q_T} \nu_\delta m_\delta'(\phi_\delta) \nabla \phi_\delta \cdot \xi \dd(t,x) &\longrightarrow \int_{Q_T}  \nu m'(\phi) \nabla \phi \cdot \xi \dd(t,x).
	\end{aligned}$$
	
	\item At this point, we only miss the first term of \cref{Eq:Deg3Terms}. We have after integration by parts
	$$\int_{Q_T} \Psi_{1,\delta}''(\phi_\delta) m_\delta(\phi_\delta) \nabla \phi_\delta \cdot \xi \dd(t,x).$$
	The term $m_\delta \Psi''_{1,\delta}$ is uniformly bounded, and it holds $\nabla \phi_\delta \to \nabla \phi$ a.e. in $Q_T$ according to \cref{Eq:Weak_Deg}. Therefore, we have to show
	$$m_\delta(\phi_\delta) \Psi_{1,\delta}''(\phi_\delta) \longrightarrow m(\phi) \Psi_1''(\phi) \text{ a.e. in } Q_T,$$
	and we proceed as in \cite[p.416f]{elliott1996cahn}. If it holds $|\phi|<1$ a.e. in $Q_T$, then the result follows from $m_\delta(\phi)=m(\phi)$ and $\Psi_{1,\delta}(\phi)=\Psi_1(\phi)$. Hence, we consider the case $\phi_\delta \to \phi=1$ a.e. in $Q_T$. If it holds $\phi_\delta \geq 1-\delta$, then it gives
	$$m_\delta(\phi_\delta) \Psi_{1,\delta}''(\phi_\delta) = m(1-\delta) \Psi_{1}''(1-\delta) \longrightarrow m(1) \Psi_1''(1)=m(\phi) \Psi_1''(\phi),$$
	and finally, in the other case of $\phi_\delta \leq 1-\delta$, it yields
	$$m_\delta(\phi_\delta) \Psi_{1,\delta}''(\phi_\delta) = m(\phi_\delta) \Psi_{1}''(\phi_\delta) \longrightarrow m(\phi) \Psi_1''(\phi).$$
 \end{itemize}
	This completes the proof of \cref{Thm:DegMob}.
	\end{proof}

\section{Numerical discretization} \label{Sec:Num}

In this section, we present a fully discrete scheme for the Ohta--Kawasaki system. We establish the existence and uniqueness of a discrete solution that preserves key structural quantities of the model, such as mass balance and energy dissipation. This approach aligns with the principles of structure-preserving discretizations, which are numerical schemes specifically designed to maintain essential properties of the continuous model, including energy dissipation, phase conservation, and thermodynamic consistency. These properties are vital for achieving physically realistic and stable numerical solutions, particularly in long-term simulations of complex systems. Numerous studies have significantly advanced structure-preserving discretizations for various dissipative problems \cite{egger2019structure} and coupled Cahn--Hilliard equations \cite{brunk2022second,brunk2023second,brunk2023variational,brunk2024structure,shimura2020error,brunk2024structure,elbar2024analysis}.  Furthermore, we mention the works \cite{bubba2022nonnegativity,bretin2023multiphase,brunk2023stability,chen2021fully,cherfils2021convergent} on analysis of numerical schemes for Cahn--Hilliard-type equations.

To implement this approach, we utilize the classical energy splitting method for the potential \(\Psi = \Psi_1 + \Psi_2\), which ensures unconditional energy stability for the classical Cahn--Hilliard equation; see \cite{elliott1993global}. Specifically, we treat the contractive (convex) component \(\Psi_1\) implicitly, while the expansive (concave) component \(\Psi_2\) is treated explicitly.

\begin{assumption} \label{As:Num} ~\\[-0.3cm]
\begin{enumerate}[label=(C\arabic*), ref=C\arabic*, leftmargin=.9cm] \itemsep.1em
    \item $\Psi=\Psi_1+\Psi_2$ with $\Psi_i\in C^2(\mathbb{R};\mathbb{R})$ such that $\Psi_1''\geq 0$ and $\Psi_2''\leq 0$ with growth condition $$|\Psi_i'(x)|\leq C_\Psi(1+|x|^r) \text{ where } 0\leq r\leq \begin{cases}
        2d/(d-2),\quad d\geq 3, \\
        p \in [0,\infty),\quad d\in\{1,2\}.
    \end{cases}$$  \label{Ass:Num:Pot}
    \item $m \in C^0(\R)$ such that $0< m_0 \leq m(x) \leq m_\infty$ for all $x \in \R$ for some constants $m_0,m_\infty<\infty$. \label{Ass:Num:Mob}
    \item $f \in C^{0}(\R;\R_{\geq 0})\cap L^\infty(\R)$. \label{Ass:Num:F}
\end{enumerate}
\end{assumption}

As a preparatory step, we introduce the relevant notation and assumptions regarding the discretization strategy for the Ohta--Kawasaki equation. \medskip 

\noindent\textbf{Spatial discretization.}
Regarding the spatial discretization, we require that $\Th$ is a geometrically conforming partition of $\Omega$ into simplices. We denote the space of continuous, piecewise linear functions over $\Th$ and the mean-free subspace by
\begin{align*}
    \Vh &:= \{v \in H^1(\Omega)\cap C^0(\bar\Omega) : v|_K \in P_1(K) \quad \forall K \in \Th\}, \\ \Vhm &:= \{v \in \Vh : \bar v = 0\}.
\end{align*}

In the following, we introduce the discrete inverse Laplacian with a mobility function in the spirit of \cite{Barrett1999}. For a given $q\in\Vh$ we introduce the discrete, weighted inverse Laplacian $\Delta_{h,m}^{-1}:\mathring{H}^{-1}(\Omega)\to\Vhm$ by
\begin{align*}
  ( m(q)\nabla(-\Delta_{h,m})^{-1}v,\nabla w )_{\!H} = (v,w)_{\!H}, \quad \forall w\in\Vhm.
\end{align*} 
The special case of $m\equiv 1$ is denoted by $(-\Delta_h)^{-1}$. This operator defines a discrete $H^{-1}(\Omega)$-norm via
\begin{align*}
  \norm{v}_{H^{-1}_{h,m}(\Omega)}^2 &:= ( \nabla(-\Delta_{h,m})^{-1}v,\nabla (-\Delta_{h})^{-1}v)_{\!H} = (m(\phi_1^n)\nabla(-\Delta_{h,m})^{-1}v,\nabla(-\Delta_{h,m})^{-1}v)_{\!H},
\end{align*}
and an associated interpolation inequality is given by
\begin{equation*}
  \norm{v}_H^2 = (m(q)\nabla(-\Delta_{h,m})^{-1}v, \nabla v)_{\!H} \leq \norm{\sqrt{m(\phi_1^n)}}_{L^\infty(\Omega)}\norm{v}_{H^{-1}_{h,m}(\Omega)}\norm{\nabla v}_{H}.  
\end{equation*}
In the case $m=1$ we denote $\norm{\cdot}_{H^{-1}_{h,m}(\Omega)}$ by $\norm{\cdot}_{H^{-1}_{h}(\Omega)}$. \medskip

\noindent\textbf{Time discretization.}
We divide the time interval $[0,T]$ into uniform steps with step size $\tau>0$ and introduce $$\Itau:=\{0=t^0,t^1=\tau,\ldots, t^{n_T}=T\},$$ where $n_T=\tfrac{T}{\tau}$ is the absolute number of time steps. We denote by $\Pi^1_c(\Itau;X)$ and $\Pi^0(\Itau;X)$ the spaces of continuous piecewise linear and piecewise constant functions on $\Itau$ with values in the space $X$, respectively. By $g^{n+1}$ and $g^n$ we denote the evaluation/approximation of a function $g$ in $\Pi^1_c(\Itau)$ or $\Pi^0(\Itau)$ at $t=\{t^{n+1},t^n\}$, respectively. We introduce the following operators and abbreviations: \medskip
\begin{itemize} \itemsep-.5em
    \item The time difference and the discrete time derivative, respectively, are denoted by
\begin{equation*}
	d^{n+1}g := g^{n+1} - g^n, \qquad d^{n+1}_\tau g := \frac{g^{n+1}-g^n}{\tau}.
\end{equation*}
\item The convex-concave splitting of the Cahn--Hilliard potential $\Psi$ is denoted by
\begin{equation*}
\Psi'(\phi_h^{n+1},\phi_h^n):=\Psi'_{1}(\phi_h^{n+1}) + \Psi'_{2}(\phi_h^{n}).
\end{equation*}
\item The discrete energy is defined as
\begin{equation*}
    \mathcal{E}(\phi_h^k) := \int_\Omega \frac{\eps^2}{2} |\nabla \phi_h^k|^2 + \Psi(\phi_h^k) +\frac{\kappa}{2} (\phi_h^k-\bar\phi_h^k)(-\Delta_h)^{-1} (\phi_h^k-\bar\phi_h^k) \, \dd x.
\end{equation*}
\end{itemize}


Next, we state the fully discrete scheme of the Ohta--Kawasaki equation.

\begin{problem}\label{prob:okscheme}
	Let the initial data $\phi_{h,0}\in \Vh$ be given. Find $\phi_h\in \Pi^1_c(\Itau;\Vh)$ and $(\mu_{h},\nu_h)\in \Pi^0(\Itau;\Vh\times\Vhm)$ that satisfy the variational system
	\begin{align}
        (\dtau\phi_h,\xi_{1,h})_{\!H} & + ( m(\phi_h^n)\nabla\mu_h^{n+1},\nabla\xi_{1,h} )_{\!H} = ( f(\phi_h^{n}),\xi_{1,h} )_{\!H}, \label{eq:pg1} \\      \eps^2(\nabla\phi_h^{n+1},\nabla\xi_{2,h} )_{\!H} &+  ( \Psi'(\phi_h^{n+1},\phi_h^n),\xi_{2,h} )_{\!H} + \kappa( \nu_h^{n+1},\xi_{2,h})_{\!H} = ( \mu_h^{n+1}, \xi_{2,h})_{\!H}, \label{eq:pg2}\\
        ( \nabla \nu_h^{n+1} ,\nabla \xi_{3,h}  )_{\!H} & = ( \phi^{n+1}_h-\bar\phi^{n+1}_h,\xi_{3,h} )_{\!H}, \label{eq:pg3}
	\end{align}
	for every $(\xi_{1,h},\xi_{2,h},\xi_{3,h})\in\Vh\times\Vh\times\Vhm$ and every $0\leq n\leq n_T-1$.
\end{problem}

\begin{theorem} \label{thm:scheme}
Let \cref{As:Pos} and \cref{As:Num} hold. Furthermore, let $h > 0$ and $\tau\leq C_0$ with $C_0>0$  solely depending on the parameters and external forces. Then, for any $\phi_{h,0} \in \Vh$, \cref{prob:okscheme} admits a unique solution $(\phi_h,\mu_h,\nu_h)$. Moreover, any such solution conserves the mass balance and energy-dissipation balance, that is, it holds for every $0 \leq m < n \leq n_T$
\begin{align}
  (\phi_h^{n},1)_{\!H} &= (\phi_h^m,1)_{\!H} + \tau\sum_{k=m}^{n-1}(f(\phi_h^k),1)_{\!H}, \label{thm1:mass:phi}  \\
   \mathcal{E}(\phi_h^{n}) &+ \tau\sum_{k=m}^{n-1}\norm{\sqrt{m(\phi_h^k)}\nabla\mu_h^{k+1}}^2_{H} \leq \mathcal{E}(\phi_h^m) + \tau\sum_{k=m}^{n-1} (f(\phi_h^k),\mu_h^{k+1})_{\!H} . \label{thm1:mass:energy}
 \end{align}
\end{theorem}

To enforce the mean-value condition of $\Vhm$, we extend the system by a Lagrange multiplier that imposes the integral mean freedom of $\nu_h^{n+1}$ and $\xi_{3,h}$.

\begin{proof} We prove first the mass and energy-dissipation balances. Afterward, we derive a priori estimates that allow us to show the existence of a solution. Lastly, we show the uniqueness of said solution. \medskip

\noindent\textbf{Step 1: Mass balance.} The balance of mass for $\phi_h$, see \eqref{thm1:mass:phi}, follows directly by inserting the test function $\xi_{1,h}=1\in\Vh$ in  \eqref{eq:pg1}. Thus, we obtain
\begin{align*}
(\phi_h^{n}-\phi_h^{m}, 1)_{\!H} = \sum_{k=m}^{n-1}(\phi_h^{k+1}-\phi_h^{k}, 1)_{\!H} = \tau\sum_{k=m}^{n-1}( d_{\tau}^{k}\phi_h, 1)_{\!H}   
&= \tau\sum_{k=m}^{n-1} (f(\phi_h^{k}),1)_{\!H}.
\end{align*}

\noindent\textbf{Step 2: Energy-dissipation balance.} Regarding the desired energy-dissipation balance, see \cref{thm1:mass:energy}, we consider first a single time-step
\begin{align*}
 \mathcal{E}(\phi_h^{n+1})-\mathcal{E}(\phi_h^n) &= \eps^2(\nabla\phi_h^{n+1},\nabla d^{n+1}\phi_h)_{\!H} + (\Psi'(\phi_h^{n+1},\phi_h^n),d^{n+1}\phi_h)_{\!H} 
 \\ &\quad + \kappa(\nabla (-\Delta_h)^{-1}(\phi_h^{n+1}-\bar\phi_h^{n+1}),\nabla (-\Delta_h)^{-1}d^{n+1}(\phi_h-\bar\phi_h))_{\!H}- \frac{\eps^2}{2}\norm{\nabla d^{n+1}\phi_h}_{\!H}^2 \\
 &\quad  - \frac{\kappa}{2}\norm{d^{n+1}(\phi_h-\bar\phi_h)}_{H^{-1}_h(\Omega)}^2 + (\Psi(\phi_h^{n+1}) - \Psi(\phi_h^{n}) - \Psi'(\phi_h^{n+1},\phi_h^n)d^{n+1}\phi_h,1)_{\!H} \\
 & = (a) + \ldots + (f).
\end{align*}
Let us start with the third term on the right-hand side, namely $(c)$. We insert $\xi_{3,h}=\kappa(-\Delta_h)^{-1}d^{n+1}(\phi_h^{n+1}-\bar\phi_h^{n+1})\in \Vhm$ into \cref{eq:pg3} and find
\begin{align*}
(c) =  \kappa(\phi_h^{n+1}-\bar\phi_h^{n+1},(-\Delta_h)^{-1}d^{n+1}(\phi_h-\bar\phi_h))_{\!H} 
&=  \kappa(\nabla\nu_h^{n+1},\nabla (-\Delta_h)^{-1}d^{n+1}(\phi_h^{n+1}-\bar\phi_h^{n+1}))_{\!H}  \\
& = \kappa(\nu_h^{n+1},d^{n+1}(\phi_h^{n+1}-\bar\phi_h^{n+1}))_{\!H} \\
& = \kappa(\nu_h^{n+1},d^{n+1}\phi_h^{n+1})_{\!H}, 
\end{align*}
where the last equality is due to the mean-freedom of $\nu_h^{n+1}$, that is $\nu_h^{n+1}\in\Vhm$.
Using $(a)$, $(b)$ and $(c)$ by inserting the test function $\xi_{2,h}=d^{n+1}\phi_h\in\Vh$ into \cref{eq:pg2}, we find
\begin{align*}
 (a) + (b) + (c) = (\mu_h^{n+1},d^{n+1}\phi_h)_{\!H}.   
\end{align*}
Finally, inserting $\xi_{1,h}=\tau\mu_h^{n+1}\in\Vh$ into into \cref{eq:pg1} yields
\begin{align*}
 (a) + (b) + (c) = -\tau(m(\phi_h^n)\nabla\mu_h^{n+1},\nabla\mu_h^{n+1})_{\!H} + \tau(f(\phi_h^{n}),\mu_h^{n+1})_{\!H}.  
\end{align*}
We note that the terms $(d)$, $(e)$, $(f)$, that is, the numerical dissipation terms, are non-negative by construction, i.e., it holds that
\begin{align*}
 (d) + (e) + (f) &\leq (\Psi(\phi_h^{n+1}) - \Psi(\phi_h^{n}) - \Psi'(\phi_h^{n+1},\phi_h^n)d^{n+1}\phi_h,1)_{\!H}     = \frac{1}{2}((\Psi_2''(\zeta_2)-\Psi_1''(\zeta_1))d^{n+1}\phi_h,d^{n+1}\phi_h)_{\!H} \leq 0,
\end{align*}
where we used convexity of $\Psi_2$ and concavity of $\Psi_1$.
Hence, in total, we obtain 
\begin{align}
 \mathcal{E}(\phi_h^{n+1}) &+ \tau\norm{\sqrt{m(\phi_h^n)}\nabla\mu_h^{n+1}}^2_{\!H} \leq \mathcal{E}(\phi_h^n) +  \tau(f(\phi_h^n),\mu_h^{n+1})_{\!H},   \label{eq:onestepE} 
\end{align}
which yields the result after summation over the relevant time steps,  \eqref{thm1:mass:energy} \medskip

\noindent\textbf{Step 3: A priori estimates.}
We will derive a priori estimates of discrete solutions by using the energy-dissipation balance \eqref{eq:onestepE}.
To obtain uniform bounds from \eqref{eq:onestepE}, we will have to estimate $ (f^n,\mu_h^{n+1})_{\!H}$ suitably. Indeed, we have by the H\"older and Young inequalities
\begin{align*}
 (f(\phi_h^n),\mu_h^{n+1})_{\!H} &= (f(\phi_h^n),\mu_h^{n+1}-\bar \mu_h^{n+1})_{\!H} + (f(\phi_h^n),\bar\mu_h^{n+1})_{\!H}  \leq \delta\norm{\nabla\mu_h^{n+1}}_{H}^2 + \delta|\bar\mu_h^{n+1}|^2 + C\norm{f(\phi_h^n)}_{H}^2.
\end{align*}
The mean-value of $\mu_h^{n+1}$ can be found using $\xi_{2,h}=1\in\Vh$ in \cref{eq:pg2}, which yields
\begin{align*}
\bar\mu_h^{n+1} = (\mu_h^{n+1},1)_{\!H} &= \eps^2(\nabla\phi_h^{n+1},\nabla 1)_{\!H} + (\Psi'(\phi_h^{n+1},\phi_h^n),1)_{\!H} + \kappa(\nu_h^{n+1},1)_{\!H}   = (\Psi'(\phi_h^{n+1},\phi_h^n),1)_{\!H},
\end{align*}
where we used $\nabla 1=0$ and $\nu_h^{n+1}\in\Vhm$. Further, using the growth condition of $\Psi'$, see \cref{Ass:Num:Pot}, yields
\begin{align}
 \mathcal{E}(\phi_h^{n+1}) &+ \tau (m_0-\delta)\norm{\nabla\mu_h^{n+1}}_{H}^2 \leq \mathcal{E}(\phi_h^{n}) + C\tau(\norm{\phi_h^{n+1}}_{V}^2 + \norm{\phi_h^{n}}_{V}^2). \label{eq:H1phi} 
\end{align}
In the next step, we insert $\xi_{1,h}=\beta\tau\phi_h^{n+1}\in\Vh$ into \cref{eq:pg1} and by the usual estimates we find
\begin{align}
 \frac{\beta}{2}\norm{\phi_h^{n+1}}_{H}^2 \leq  \frac{\beta}{2}\norm{\phi_h^{n}}_{H}^2 + \tau\delta\norm{\nabla\mu_h^{n+1}}_{H}^2 + C\tau\norm{\phi_h^{n+1}}_{V}^2 + C\tau\norm{f^n}^2_{H}. \label{eq:l2phi}
\end{align}
Combining \eqref{eq:H1phi} and \eqref{eq:l2phi} and choosing $\delta= \frac{m_0}{4}$ yields
\begin{align}
 \mathcal{E}(\phi_h^{n+1}) &+ \frac{\beta}{2}\norm{\phi_h^{n+1}}_{H}^2 + \frac{\tau m_0}{2}\norm{\nabla\mu_h^{n+1}}_{H}^2 \leq \mathcal{E}(\phi_h^{n}) + C\tau(\norm{\phi_h^{n+1}}_{V}^2 +  \norm{\phi_h^{n}}_{V}^2 + \norm{f^n}_{H}^2). \label{eq:rawapriori}
\end{align}
The left-hand side of \eqref{eq:rawapriori} can be estimated from below via
\begin{align*}
 (\text{LHS})_{\eqref{eq:rawapriori}} \geq C_1\norm{\phi_h^{n+1}}_V^2 + \norm{\Psi(\phi_h^{n+1})}_{L^1(\Omega)}  + \frac{\kappa}{2}\norm{\phi-\bar\phi}_{H^{-1}_h(\Omega)}^2 =: y^{n+1}. 
\end{align*}
Furthermore, the right-hand side of \eqref{eq:rawapriori} can be estimated from above by
\begin{align*}
 (\text{RHS})_{\eqref{eq:rawapriori}}  \leq  (C+C_2\tau)y^{n} + C_2\tau y^{n+1} + C\tau\norm{f^n}_{H}^2.
\end{align*}
Finally, we set $d^{n+1}=\tfrac{m_0}{2}\norm{\nabla\mu_h^{n+1}}_{H}^2$ and obtain
\begin{align*}
    y^{n+1} + \tau d^{n+1} \leq (1+C\tau)y^{n} + C\tau y^{n+1} + C\tau\norm{f^n}_{H}^2.
\end{align*}
An application of the discrete Gronwall lemma with $\tau\leq \frac{C}{2}$ implies the desired uniform bounds
\begin{align}
 \norm{\phi_h}_{L^\infty(V)}^2 + \norm{\phi_h^{n+1}-\bar\phi_h^{n+1}}_{L^\infty(H^{-1}_h(\Omega))}^2 + \norm{\mu_h}_{L^2(V)}^2 \leq C. \label{eq:apriori1}   
\end{align}
By definition of the discrete inverse Laplace $(-\Delta_h)^{-1}$ and \eqref{eq:pg3}, one finds $\nu_h^{k}=(-\Delta_h)^{-1}(\phi_h^{k}-\bar \phi_h^{k})$. Then, using $\nu_h^{k}\in\Vhm$, i.e. mean free, we directly find 
\begin{equation}
    \norm{\nu_h}_{L^\infty(V)}^2 \leq C. \medskip \label{eq:apriori2} 
\end{equation}

\noindent\textbf{Step 4: Existence of solutions.}
We consider the $n$-th time-step and assume that $\phi_h^{n-1}$ is already known. 
After choosing a basis for $\Vh\times\Vh\times\Vhm$, we rewrite the variational system \eqref{eq:pg1}--\eqref{eq:pg3} in \cref{prob:okscheme} as a nonlinear system in the form of 
$$J(x)=0 \text{ in } \mathbb{R}^{Z},$$ 
with $Z=\mathrm{dim(\Vh)}^2\times\mathrm{dim}(\Vhm)$. In fact, writing $( J(x),x)$ is equivalent to testing the corresponding variational identities with $$x=(\tau\mu_h^{n+1}+\tau\beta\phi_h^{n+1},d^{n+1}\phi_h,\kappa(-\Delta_h)^{-1}d^{n+1}(\phi_h^{n+1}-\bar\phi_h^{n+1})).$$
As a consequence of the latter, we thus obtain 
\begin{align*}
( J(x),x) &\geq y^{n+1} - Cy^n + \tau D^{n+1} - \tau C.
\end{align*}
Using the same arguments as we have used to derive the a priori bounds \eqref{eq:apriori1} and \eqref{eq:apriori2}, we directly obtain $$(J(x),x) \to \infty \text{ for } |x| \to \infty.$$ 
Thus, the existence of a solution follows from a corollary to Brouwer's fixed-point theorem, see \cite[Proposition~2.8]{Zeidler1}. \medskip

\noindent\textbf{Step 5: Uniqueness.}
For the sake of presentation, we will neglect the space discretization index $h$ in the remaining part of the proof. Again, we consider only one time step.
We assume that two different solutions $(\phi_1^{n+1},\mu_1^{n+1},\nu_1^{n+1})$ and $(\phi_2^{n+1},\mu_2^{n+1},\nu_2^{n+1})$ exist. We denote their differences by $\phi^{n+1},\mu^{n+1}$ and $\nu^{n+1}$. Taking the difference in the variational forms yields the following discrete formulations
\begin{align}
(\phi^{n+1},\xi_{1})_{\!H} & + \tau( m(\phi_1^n)\nabla\mu^{n+1},\nabla\xi_{1} )_{\!H} = 0, \label{eq:diffphi}\\
 \gamma( \nabla\phi^{n+1},\nabla\xi_{2} )_{\!H} &+  ( \Psi_{1}''(\zeta)\phi^{n+1},\xi_{2} )_{\!H} + \kappa(\nu^{n+1},\xi_{2})_{\!H} = ( \mu^{n+1}, \xi_{2})_{\!H},\label{eq:diffmu}\\
(\nabla\nu^{n+1},\nabla \xi_{3})_{\!H} & = (\phi^{n+1}-\bar\phi^{n+1},\xi_{3})_{\!H}, \label{eq:diffnu}
\end{align}
which holds for all $(\xi_{1},\xi_{2},\xi_{3})\in\Vh\times\Vh\times\Vhm$.
We note that $\zeta$ is a convex combination of $\phi_1^{n+1}$ and $\phi_2^{n+1}$.
Although $\phi_1^{n+1}$ and $\phi_2^{n+1}$ are not mean free, the difference $\phi^{n+1}$ is mean free, that is, inserting $\xi_{1}=1$ in \eqref{eq:diffphi}. Hence, we are allowed to apply the weighted inverse Laplacian on $\phi^{n+1}.$
Using $\xi_{1}=(-\Delta_{h,m}^{-1})\phi^{n+1}\in\Vh$ with $q_h=\phi_1^{n-1}$, $\xi_{2}=\phi^{n+1}\in\Vh$ and $\xi_{3}=\nu^{n+1}\in\Vhm$ as test functions in \eqref{eq:diffphi}, \eqref{eq:diffmu}, \eqref{eq:diffnu}, respectively, which yields
\begin{align*}
 \norm{\phi^{n+1}}_{H^{-1}_{h,m}(\Omega)}^2 &+ \tau\big(\norm{\nabla\phi^{n+1}}_{H}^2 + ( \Psi_{1}''(\zeta)\phi^{n+1},\phi^{n+1} )_{\!H} + \kappa\norm{\nabla\nu^{n+1}}_{H}^2\big) = 0. 
\end{align*}
Since the terms contained in the left-hand side are non-negative it follows that $\phi^{n+1}=0$. This directly shows that $\nu^{n+1}=\mu^{n+1}=0$, which yields a contradiction and implies the uniqueness result, as stated in \cref{thm:scheme}. This finishes the proof.
\end{proof}

\section{Numerical illustrations} \label{Sec:NumIll}

In this section, we apply the discretization of time and space as described in \cref{prob:okscheme} to some examples. We consider either  the two-dimensional square domain $\Omega=(0,1)^2$ or  the three-dimensional cube domain $\Omega=(0,1)^3$. We consider $h=\tau=10^{-2}$ for all two-dimensional simulations and $h=\tau=10^{-3}$ for all three-dimensional simulations.
At each time step, we solve this system with the Newton method. The procedure has been implemented in the high-performance multiphysics finite element software NGSolve \cite{schoberl2014c++} to obtain the numerical results which we present in this section.

In all simulations, we fix the interfacial width $\eps^2=10^{-3}$, the mobility function $m(\phi)=10^{-14} + \frac{1}{16}(1-\phi^2)^2$ and the potential function
 $\Psi(\phi) = \frac{1}{4}(\phi^2-1)^2$ with $\Psi_1(\phi) = \frac{1}{4}(\phi^4+1)$ and $\Psi_2(\phi)= - \frac{1}{2}\phi^2$.

\subsection{2D experiments}
In this subsection, we consider two exemplary settings and compare the effects of the repulsion strength $\kappa>0$ in the Ohta--Kawasaki system to the case $\kappa=0$, that is, the usual Cahn--Hilliard equation. First, we consider the case of no external force $f=0$, but variable mobility $m$ as chosen above.

\begin{experiment}\label{exp:1} $\phi_0(x)=\dfrac{\cos(2\pi x_1) \cos(2\pi x_2)}{100}$, \quad $f(\phi)=0$, \quad $\kappa\in\{0,10,100\}$.
\end{experiment}

The evolution of the phase-field \(\phi\) for varying values of the repulsion parameter \(\kappa \in \{0,10,100\}\) is illustrated in \cref{pic:exp1}. When \(\kappa = 10\), the system exhibits behavior closely aligned with the Cahn--Hilliard evolution (\(\kappa = 0\)), indicating that moderate repulsion has minimal impact on the general phase-separation dynamics. Notably, the primary distinction lies in the relaxation time of smaller bubbles, which is slightly extended compared to \(\kappa = 0\). In contrast, when \(\kappa = 100\), we observe a markedly different evolution, with phase separation occurring at a significantly slower rate. This high repulsion parameter delays the coalescence of phases, leading to incomplete separation at the chosen final time.
The plot on the left in \cref{pic:exp1struc} confirms that mass conservation is preserved to machine precision (approximately \(10^{-15}\)) across all cases, underscoring the robustness of the numerical scheme. Meanwhile, the energy evolution depicted on the right in \cref{pic:exp1struc} reveals that larger \(\kappa\) values noticeably dampen energy dissipation. This effect of \(\kappa\) suggests an inhibitory influence on phase separation, as greater repulsion requires more energy to overcome stabilizing effects. These observations align with the physical interpretation of \(\kappa\) as a regulator of repulsive interactions, where an increase in \(\kappa\) retards phase progression. \medskip

\begin{figure}[htbp!]
\begin{center}
\begin{tabular}{cM{.24\textwidth}M{.24\textwidth}M{.24\textwidth}}
			&$\kappa=0$&$\kappa=10$&$\kappa=100$ \\
$t=0$ &\includegraphics[trim={39.0cm 10.0cm 22.0cm 10.0cm},clip,scale=0.055]{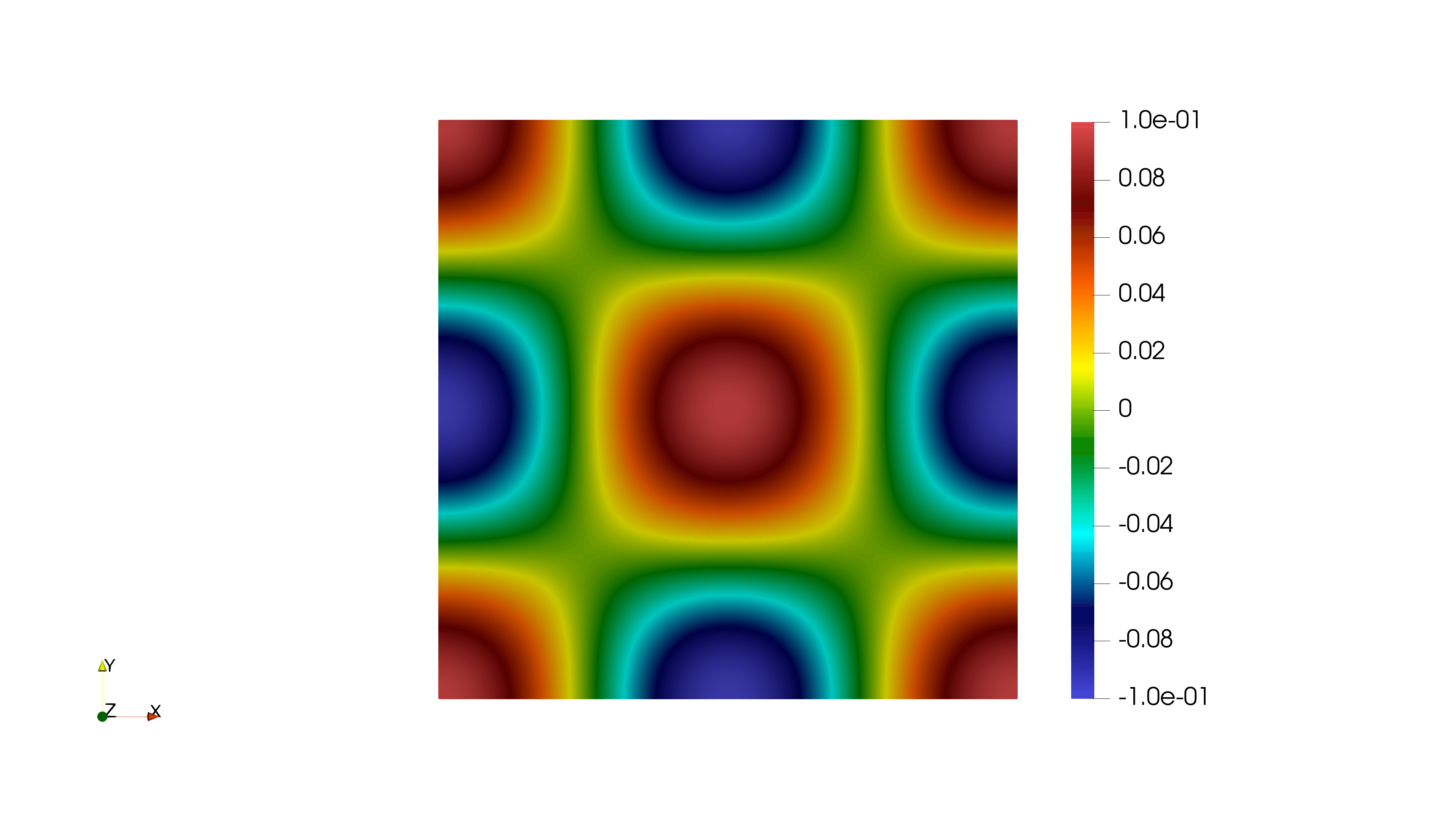}
&\includegraphics[trim={39.0cm 10.0cm 22.0cm 10.0cm},clip,scale=0.055]{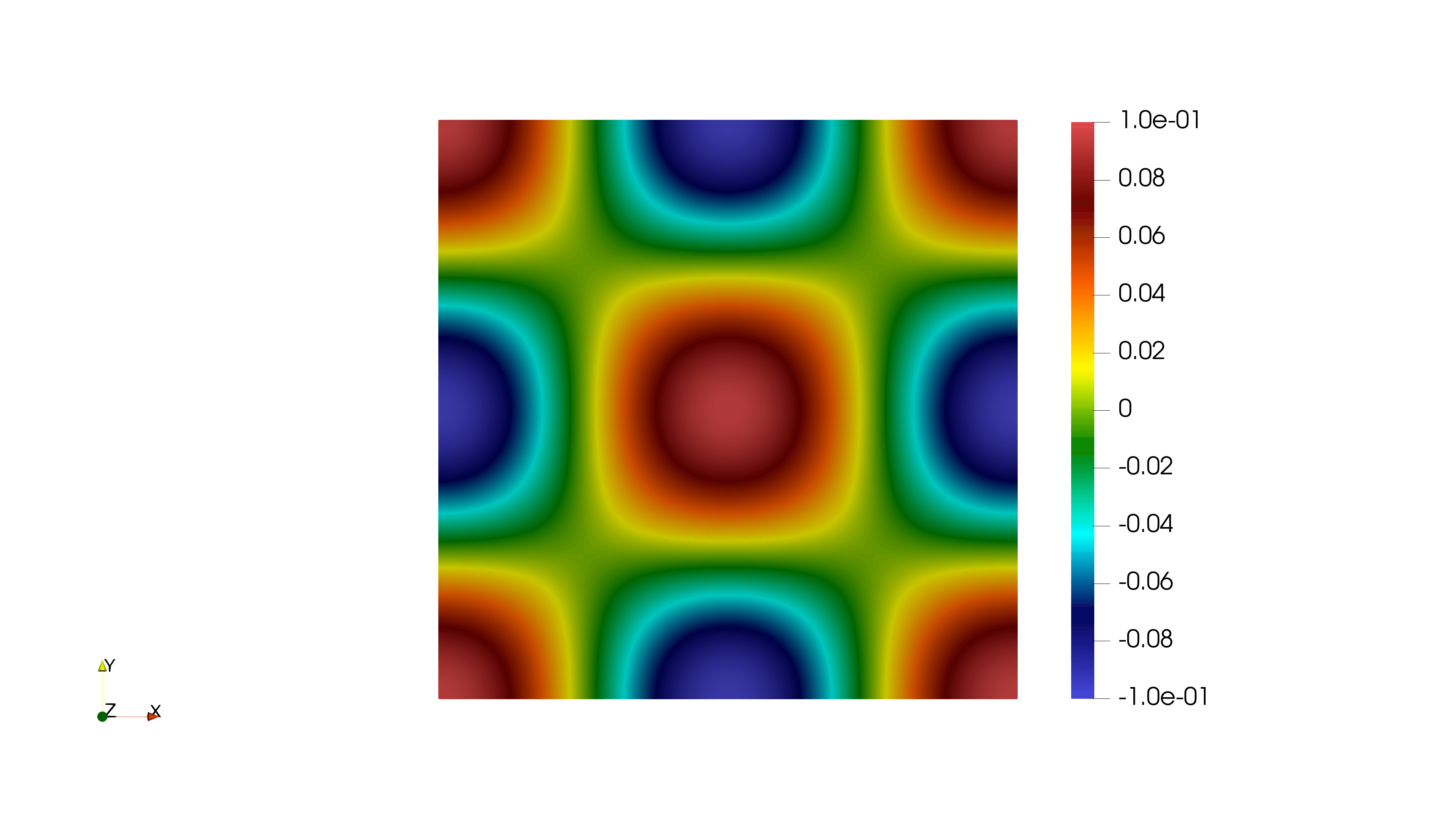}
&\includegraphics[trim={39.0cm 10.0cm 22.0cm 10.0cm},clip,scale=0.055]{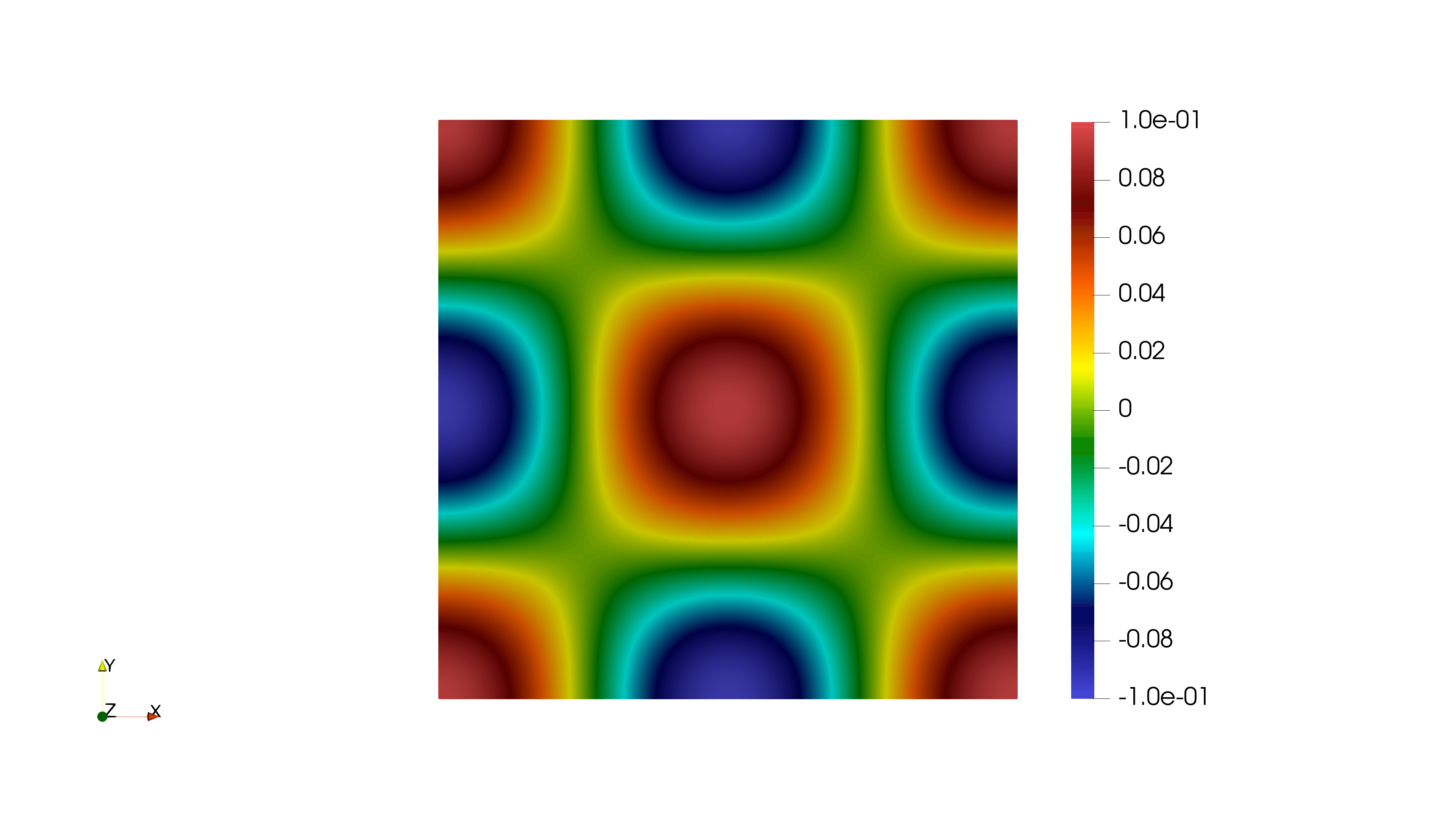}\\[-.1cm]
$t=1$ &\includegraphics[trim={39.0cm 10.0cm 22.0cm 10.0cm},clip,scale=0.055]{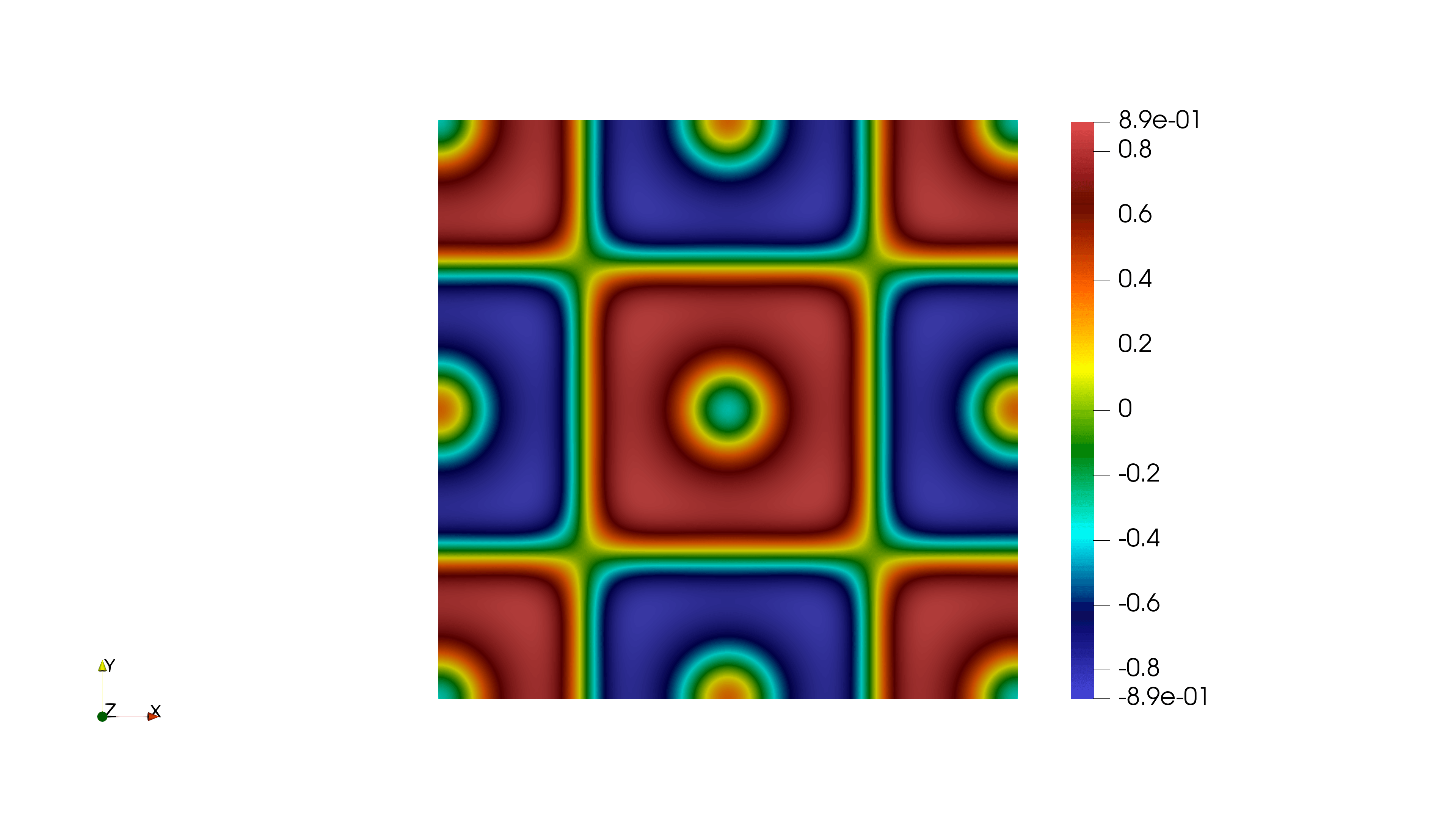}
&\includegraphics[trim={39.0cm 10.0cm 22.0cm 10.0cm},clip,scale=0.055]{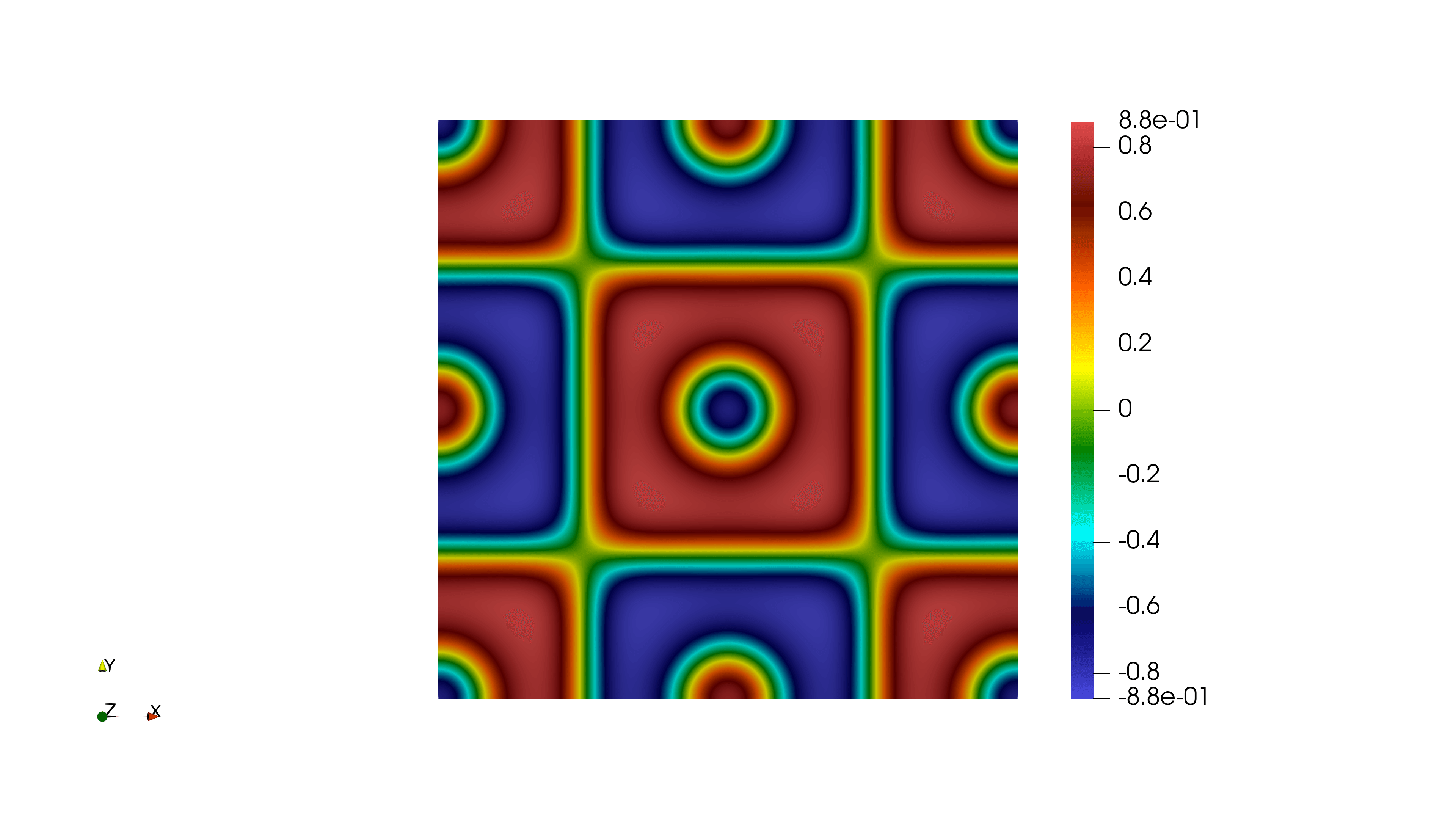}
&\includegraphics[trim={39.0cm 10.0cm 22.0cm 10.0cm},clip,scale=0.055]{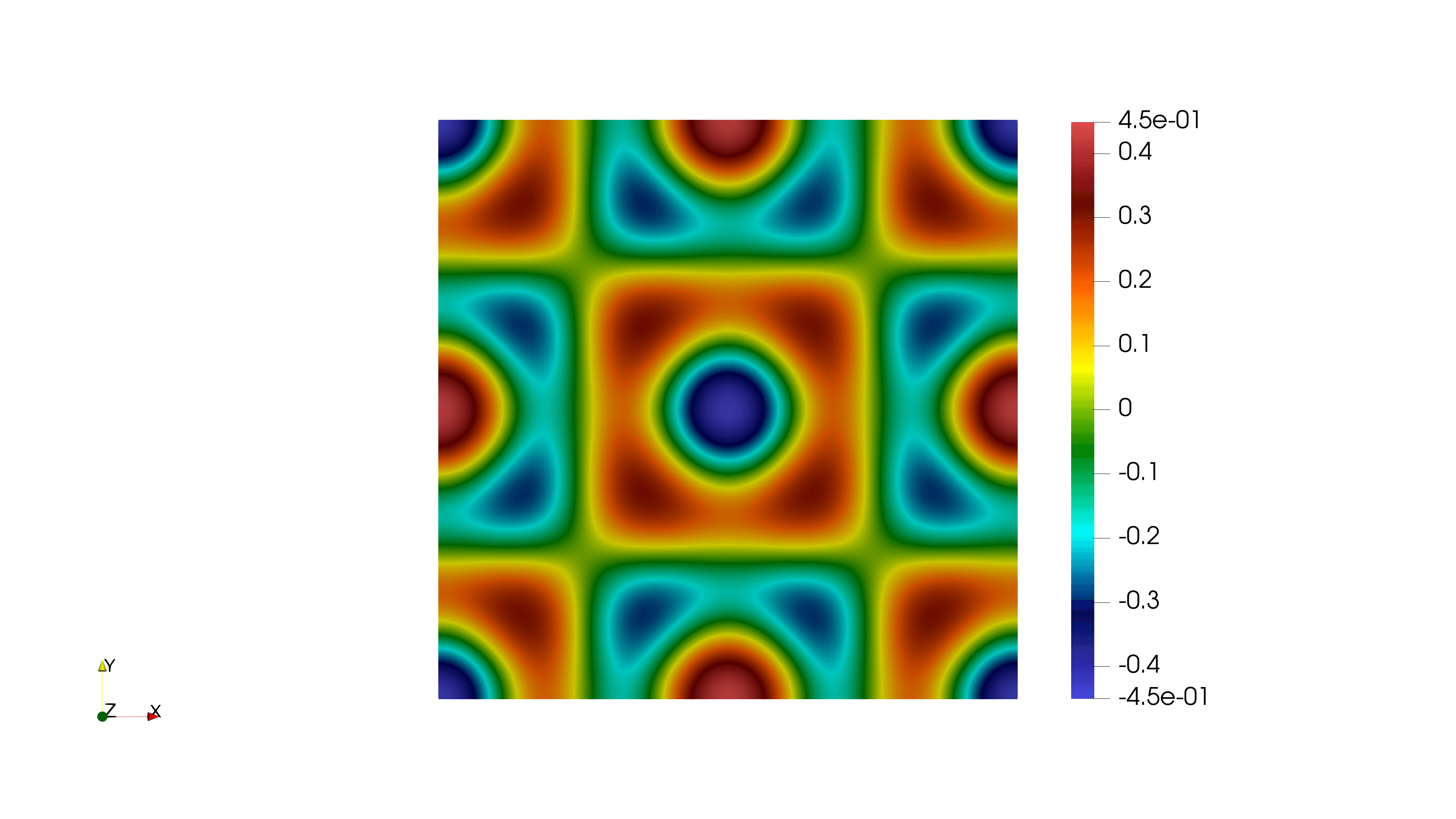}\\[-.1cm]
$t=2$ &\includegraphics[trim={39.0cm 10.0cm 22.0cm 10.0cm},clip,scale=0.055]{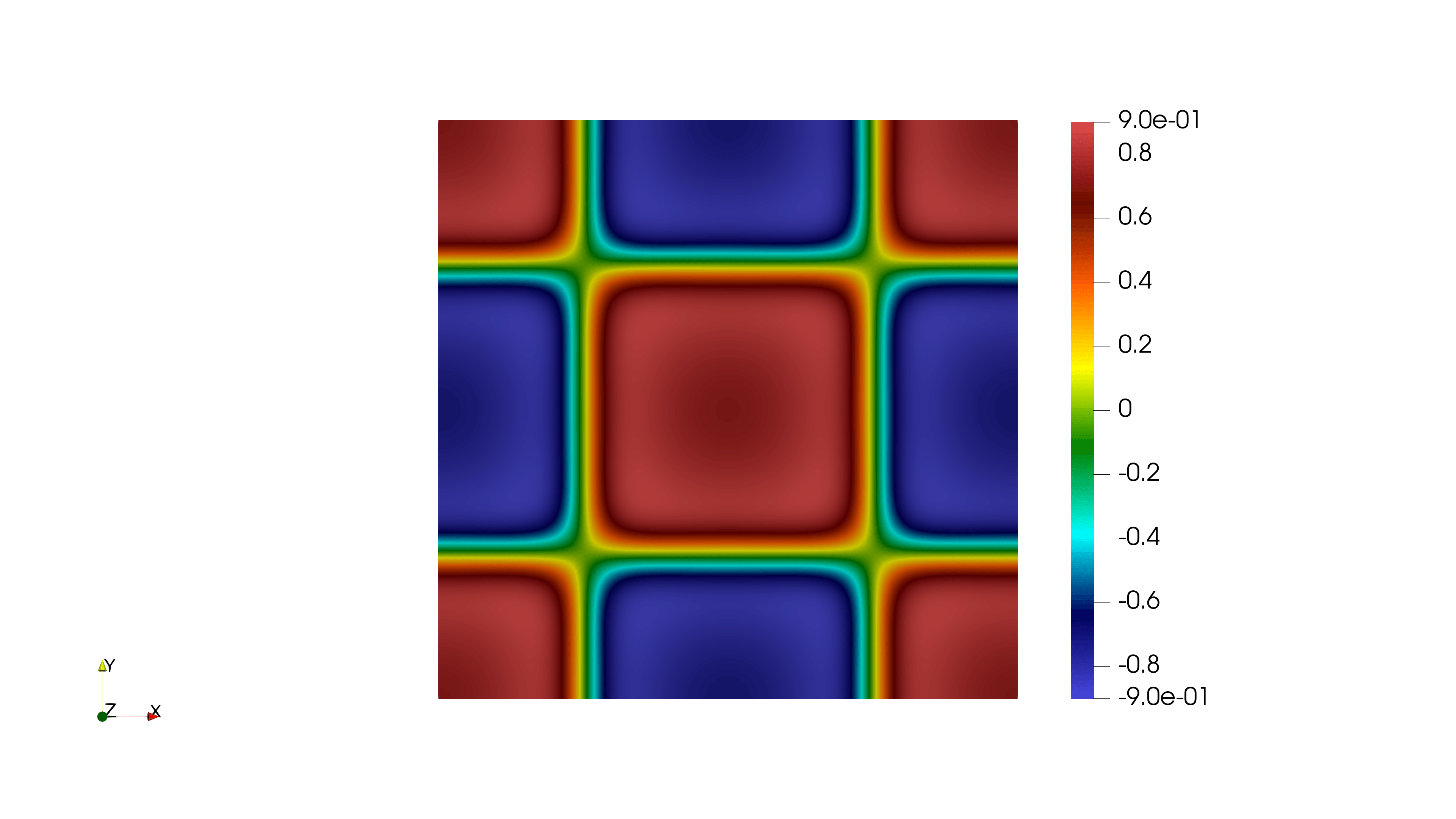}
&\includegraphics[trim={39.0cm 10.0cm 22.0cm 10.0cm},clip,scale=0.055]{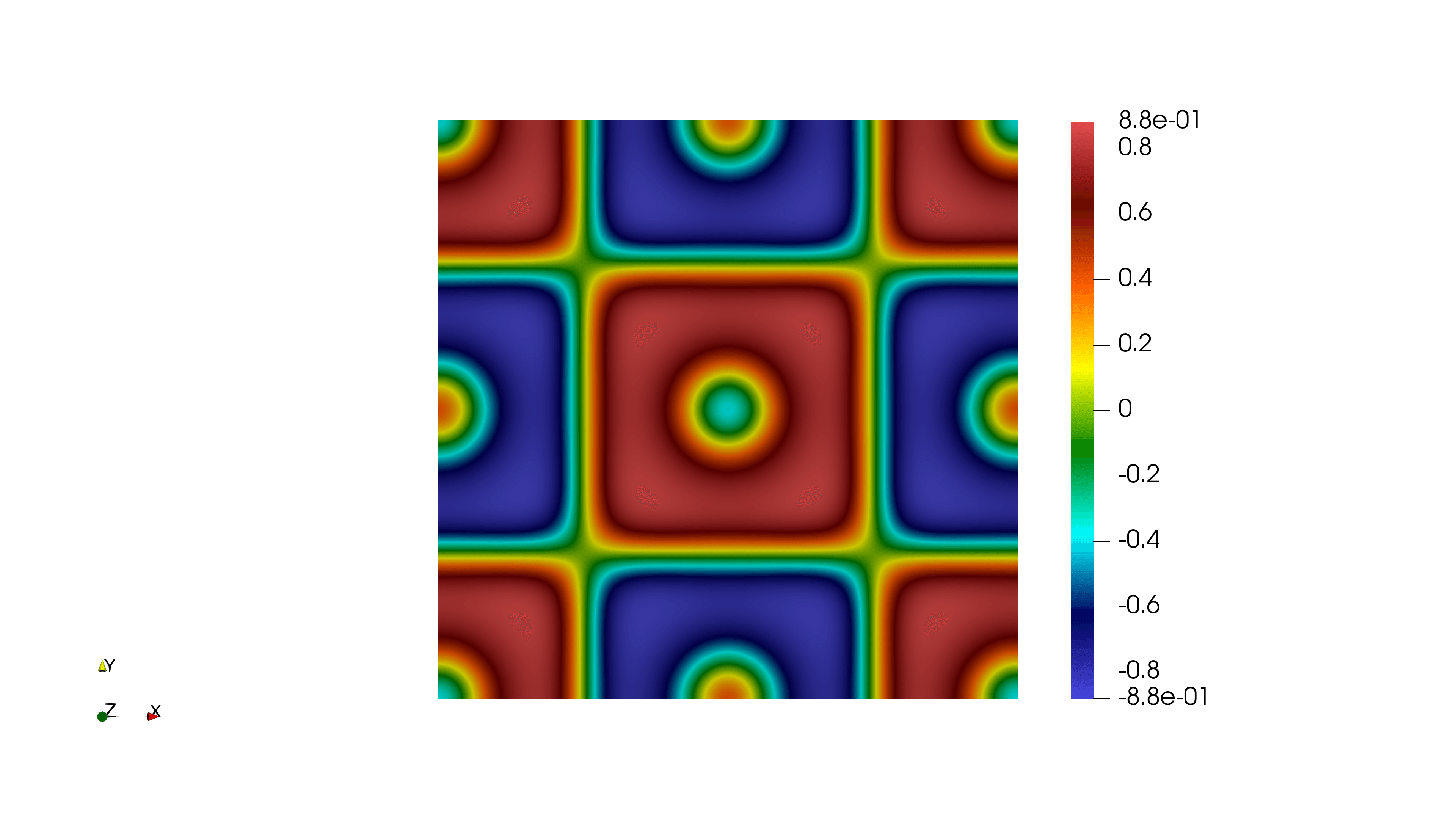}
&\includegraphics[trim={39.0cm 10.0cm 22.0cm 10.0cm},clip,scale=0.055]{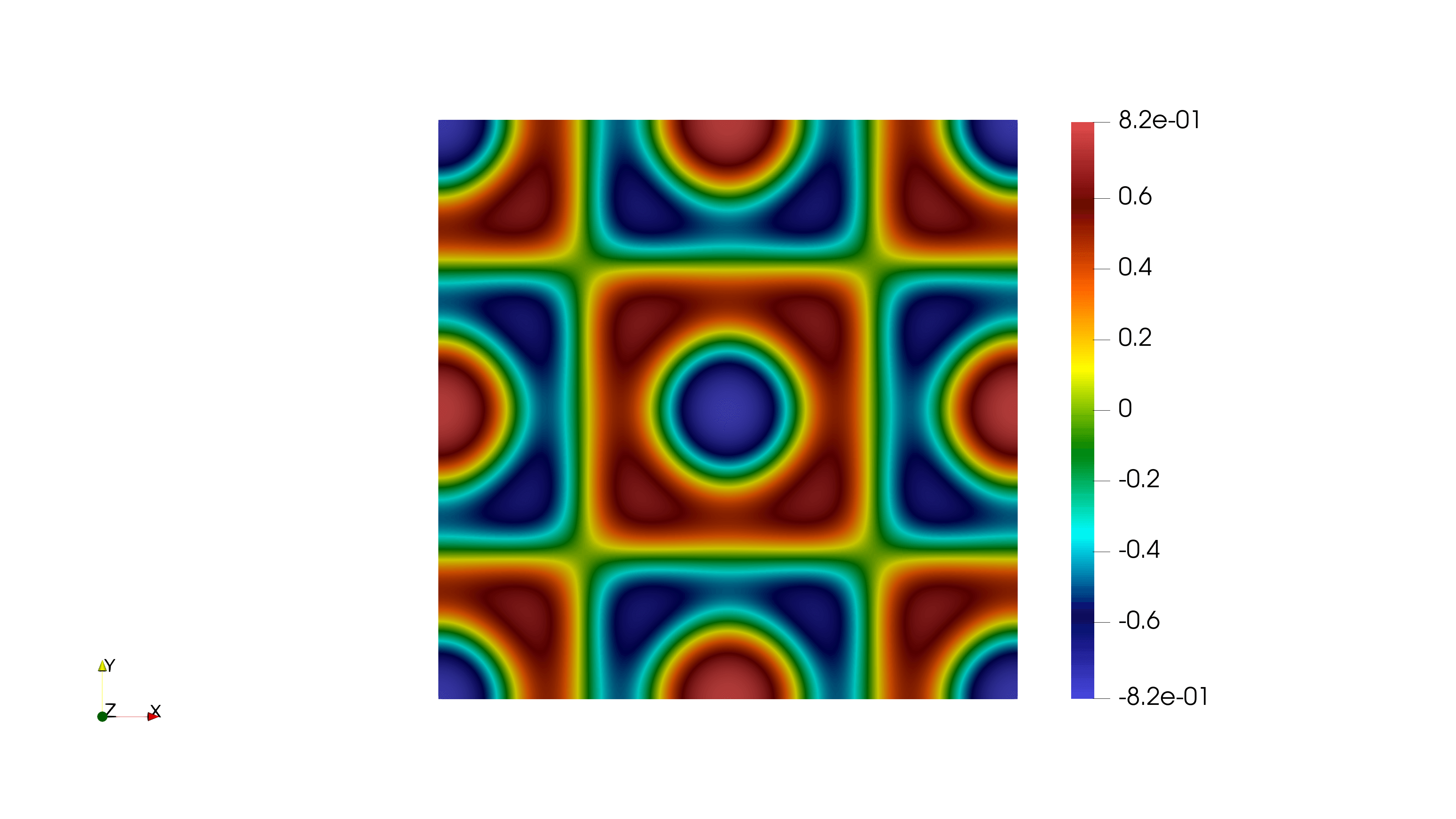}\\[-.1cm]
$t=5$ &\includegraphics[trim={39.0cm 10.0cm 22.0cm 10.0cm},clip,scale=0.055]{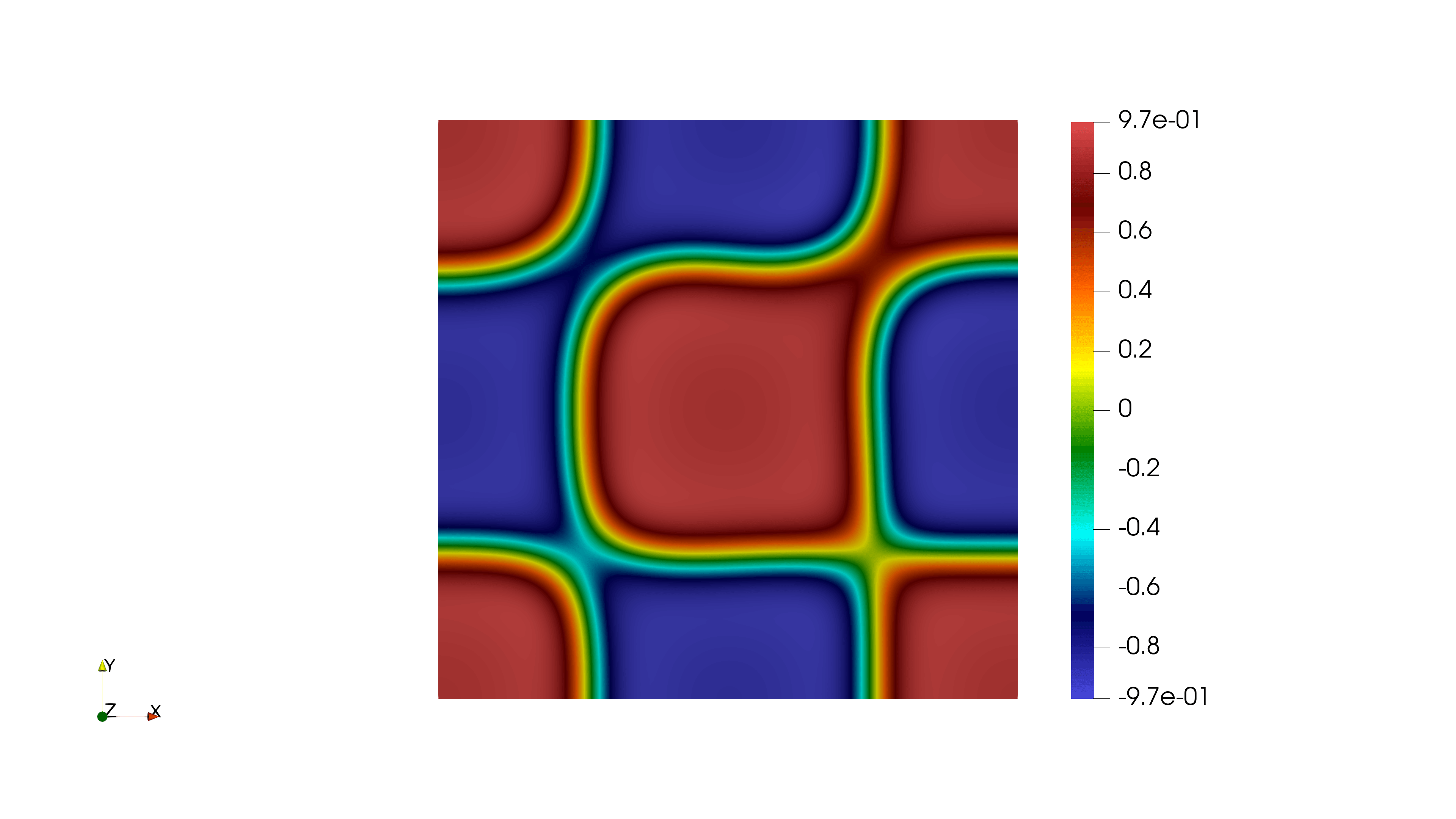}
&\includegraphics[trim={39.0cm 10.0cm 22.0cm 10.0cm},clip,scale=0.055]{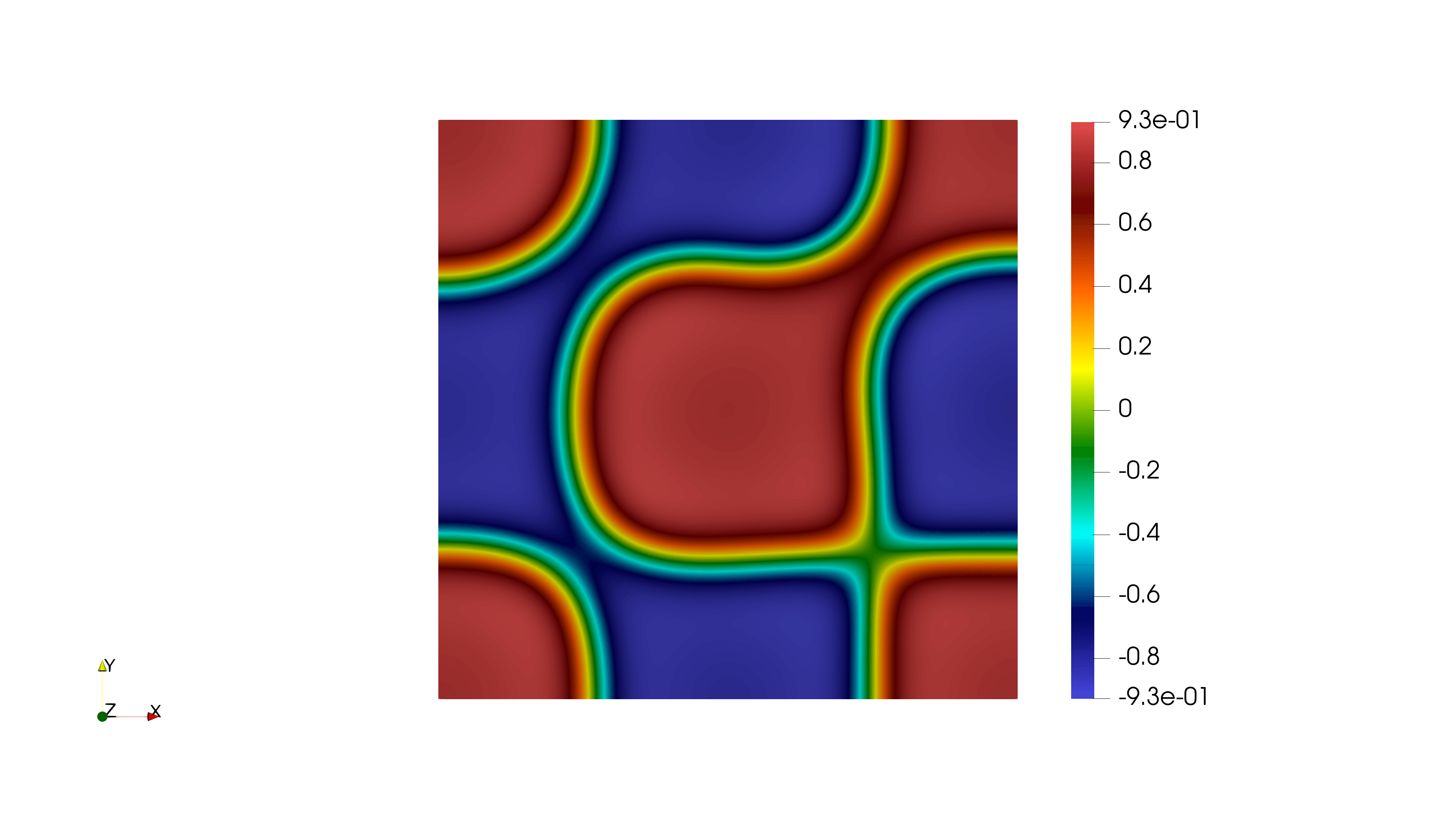}
&\includegraphics[trim={39.0cm 10.0cm 22.0cm 10.0cm},clip,scale=0.055]{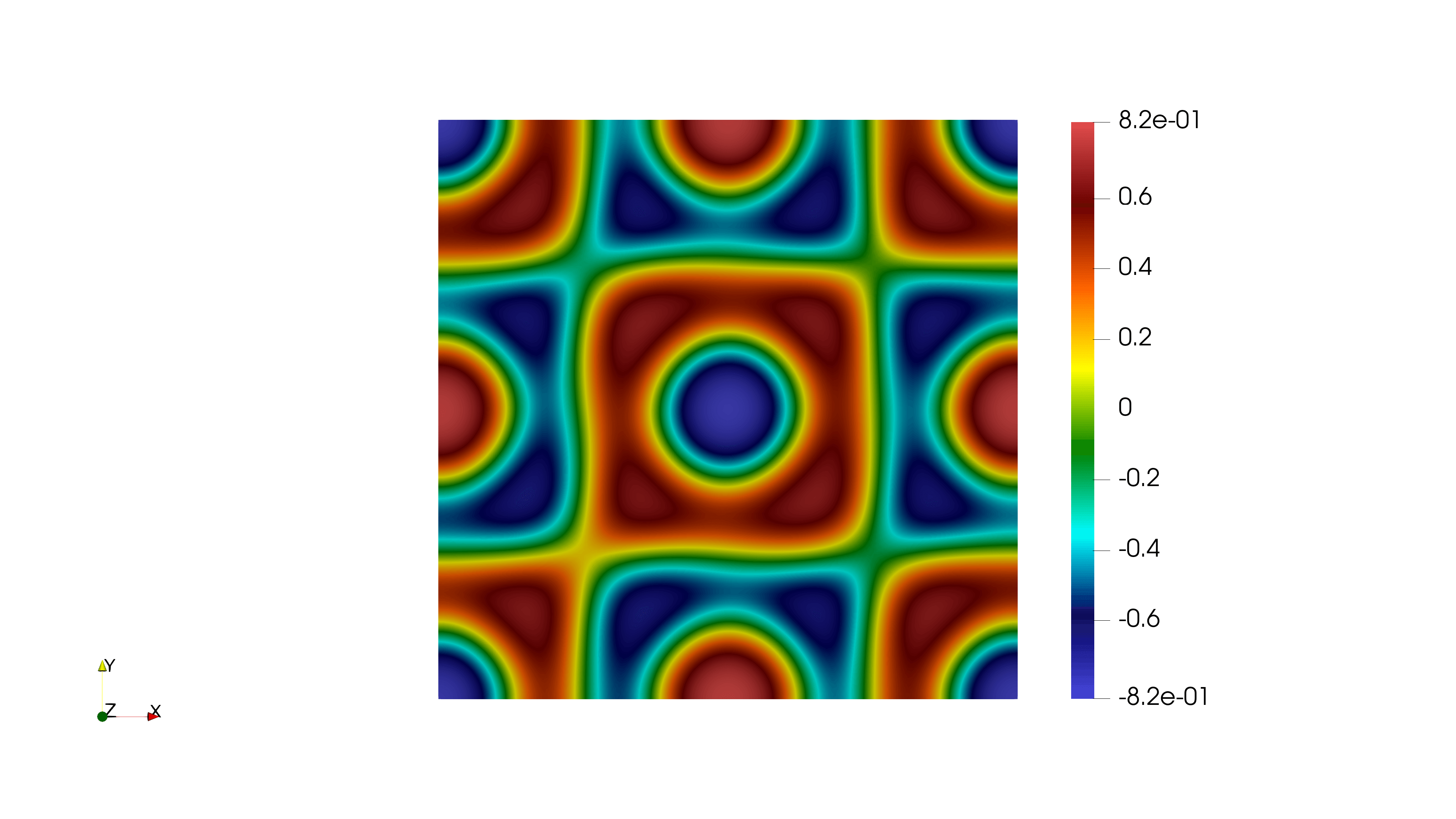}
\end{tabular}
\end{center}
\caption{\cref{exp:1} for the Ohta--Kawasaki model: Snapshots of the phase-field \(\phi\) for \(\kappa \in \{0, 10, 100\}\) (left to right) at times \(t \in \{0, 1, 2, 5\}\) (top to bottom).  \label{pic:exp1}}
\end{figure}

\begin{figure}[htbp!]
\centering
\includegraphics[trim={0.0cm 0.0cm 0.0cm 0.0cm},clip,scale=0.45]{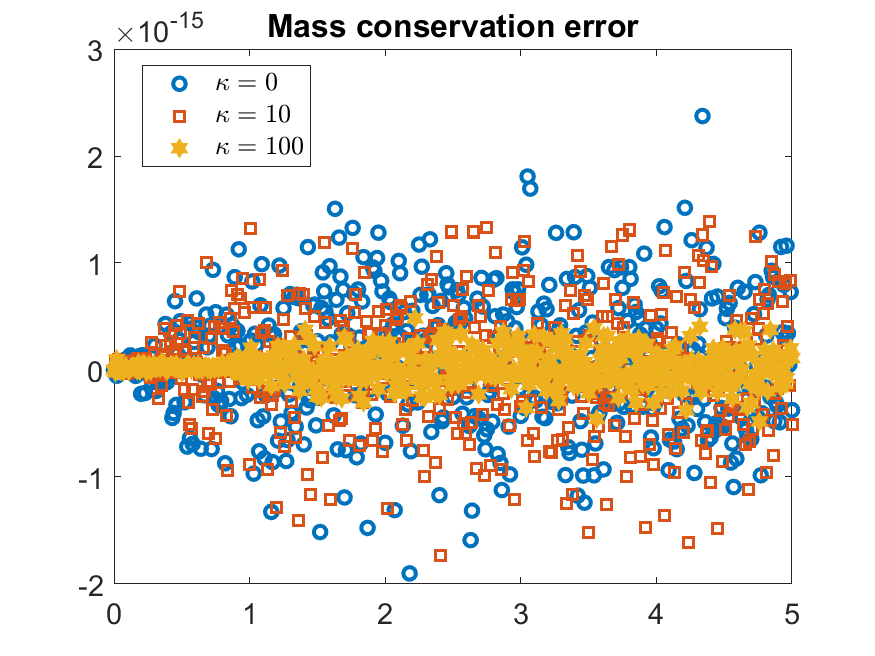}
\includegraphics[trim={0.0cm 0.0cm 0.0cm 0.0cm},clip,scale=0.45]{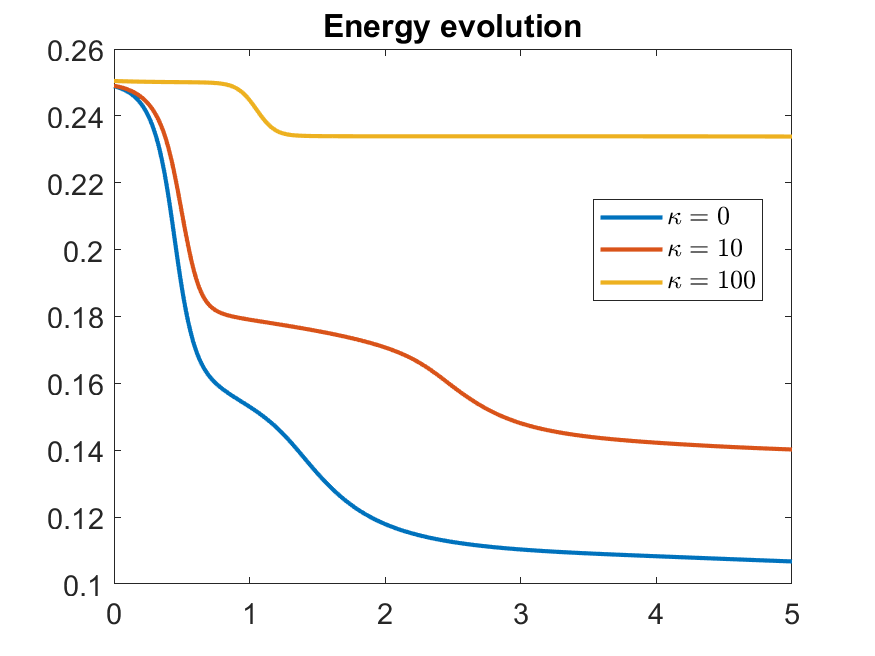}
\caption{\cref{exp:1} for the Ohta--Kawasaki model: Mass conservation error (left) and evolution of the energy $\mathcal{E}$ (right).  \label{pic:exp1struc}}
\end{figure}

In the second experiment, we consider the influence of a logistic growth term $f$ as follows:

\begin{experiment}\label{exp:2} 
$\phi_0(x)=\dfrac{\cos(2\pi x_1) \cos(2\pi x_2)}{100}$, \quad $f(\phi)=\dfrac{1}{10}\max\{0,1-\phi^2\}$, \quad $\kappa\in\{0,10,100\}$.
\end{experiment}

The evolution of the phase-field \(\phi\) is depicted in \cref{pic:exp2}. As in the previous experiment, the case \(\kappa=10\) behaves similarly to the Cahn--Hilliard evolution (\(\kappa=0\)), showing only slight differences in the relaxation dynamics. For \(\kappa=100\), however, the evolution deviates significantly, indicating that the repulsion parameter \(\kappa\) influences the number of energetically favorable bubbles. Similar findings have been reported in \cite{Xu2019} for an Allen--Cahn variant of the system, supporting the connection between \(\kappa\) and the stable bubble count.
In \cref{pic:exp2struc}, we display the evolution of mass and energy over time. By construction, the mass of the phase-field increases and we observe a correlation between \(\kappa\) and the growth rate. This outcome aligns with expectations, since slower phase separation processes yield values in regimes where logistic growth is prominent. Interestingly, despite the presence of an external force, the energy decreases over time, suggesting that the external influence does not preclude energy dissipation but may instead alter the phase separation dynamics. This trend emphasizes the role of \(\kappa\) in modulating the stability of the system and the phase distribution over time. \medskip

\begin{figure}[htbp!]
\begin{center}
\begin{tabular}{cM{.24\textwidth}M{.24\textwidth}M{.24\textwidth}}
			&$\kappa=0$&$\kappa=10$&$\kappa=100$ \\
$t=0$ &\includegraphics[trim={39.0cm 10.0cm 22.0cm 10.0cm},clip,scale=0.055]{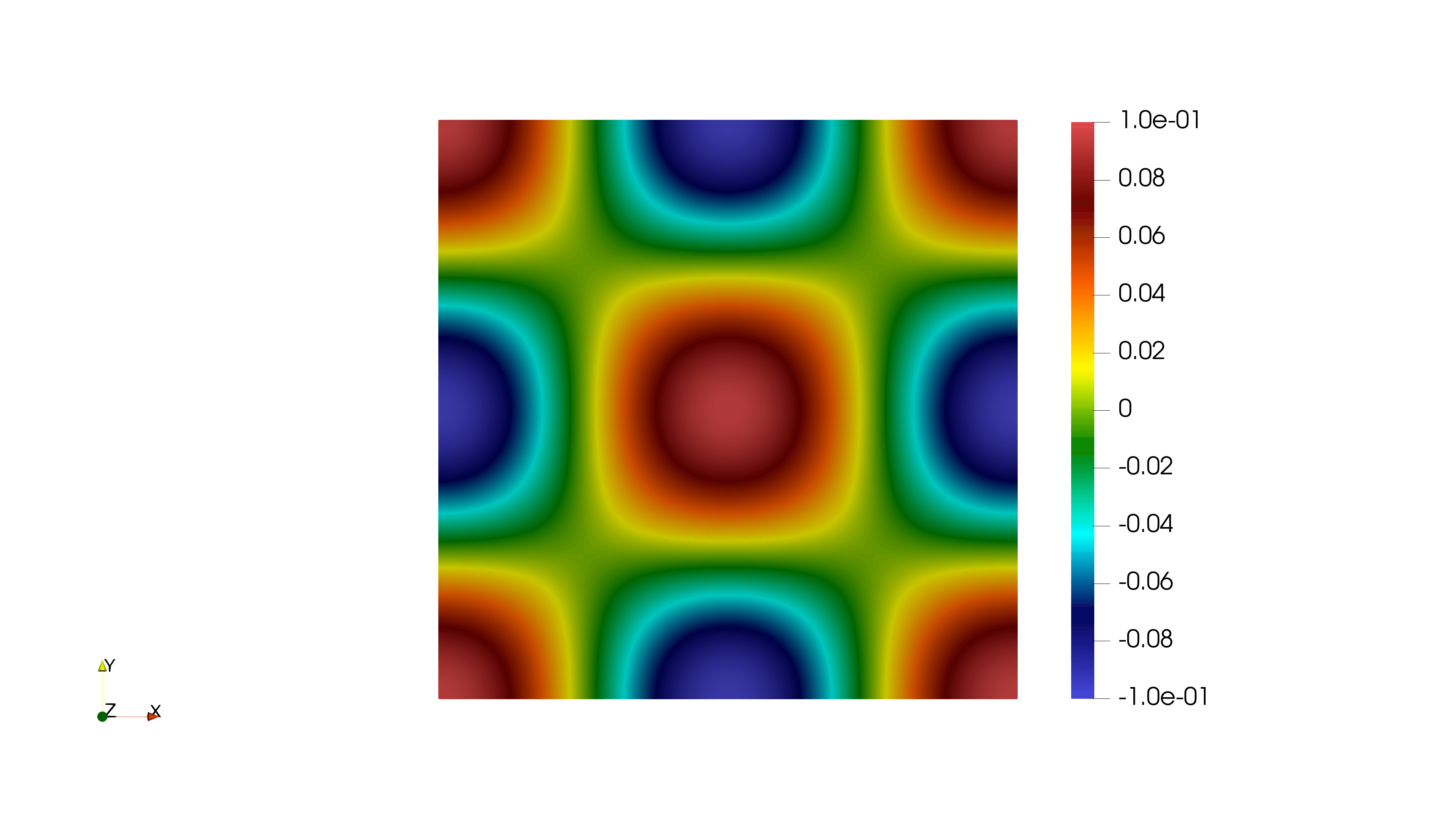}
&\includegraphics[trim={39.0cm 10.0cm 22.0cm 10.0cm},clip,scale=0.055]{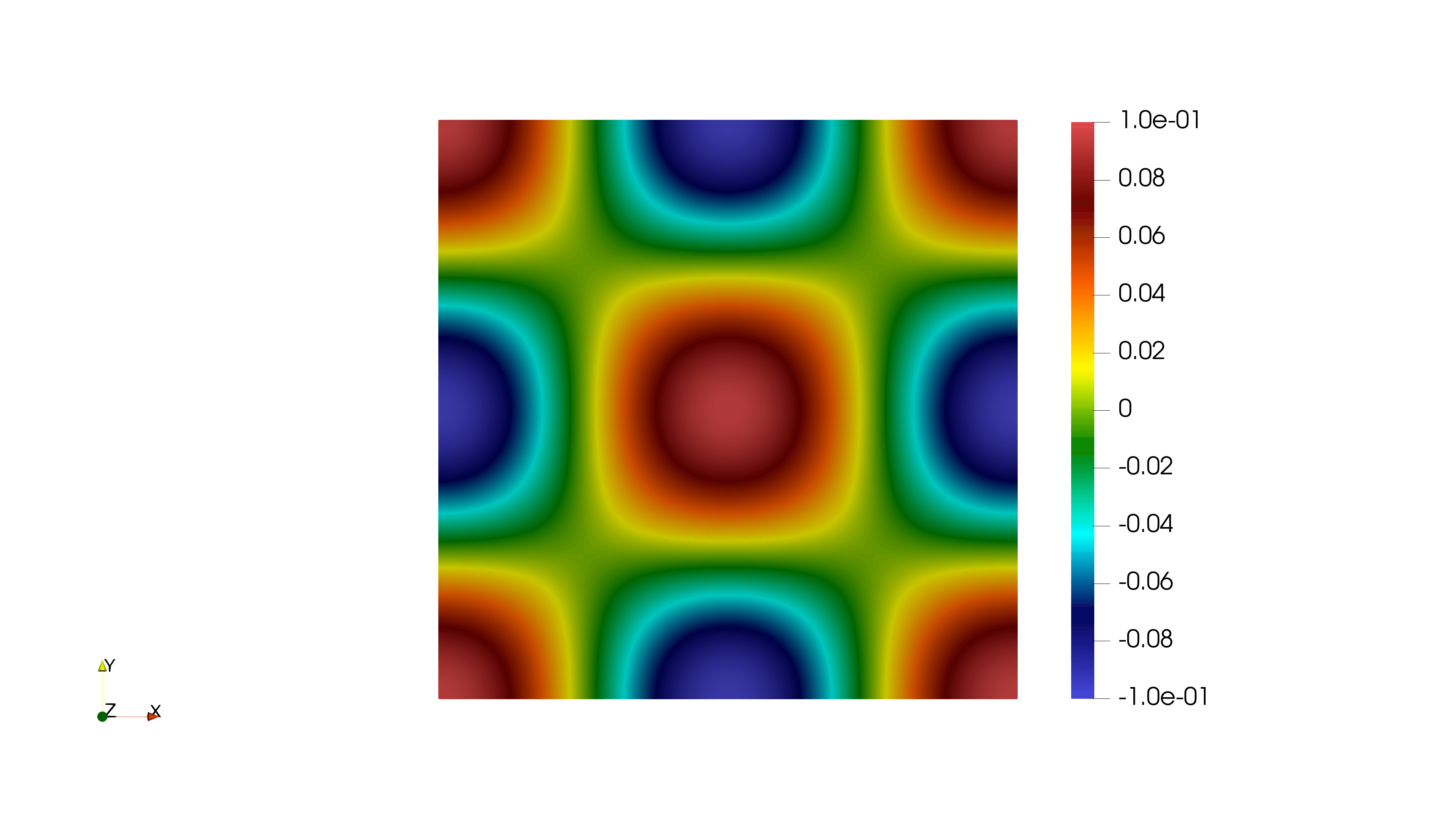}
&\includegraphics[trim={39.0cm 10.0cm 22.0cm 10.0cm},clip,scale=0.055]{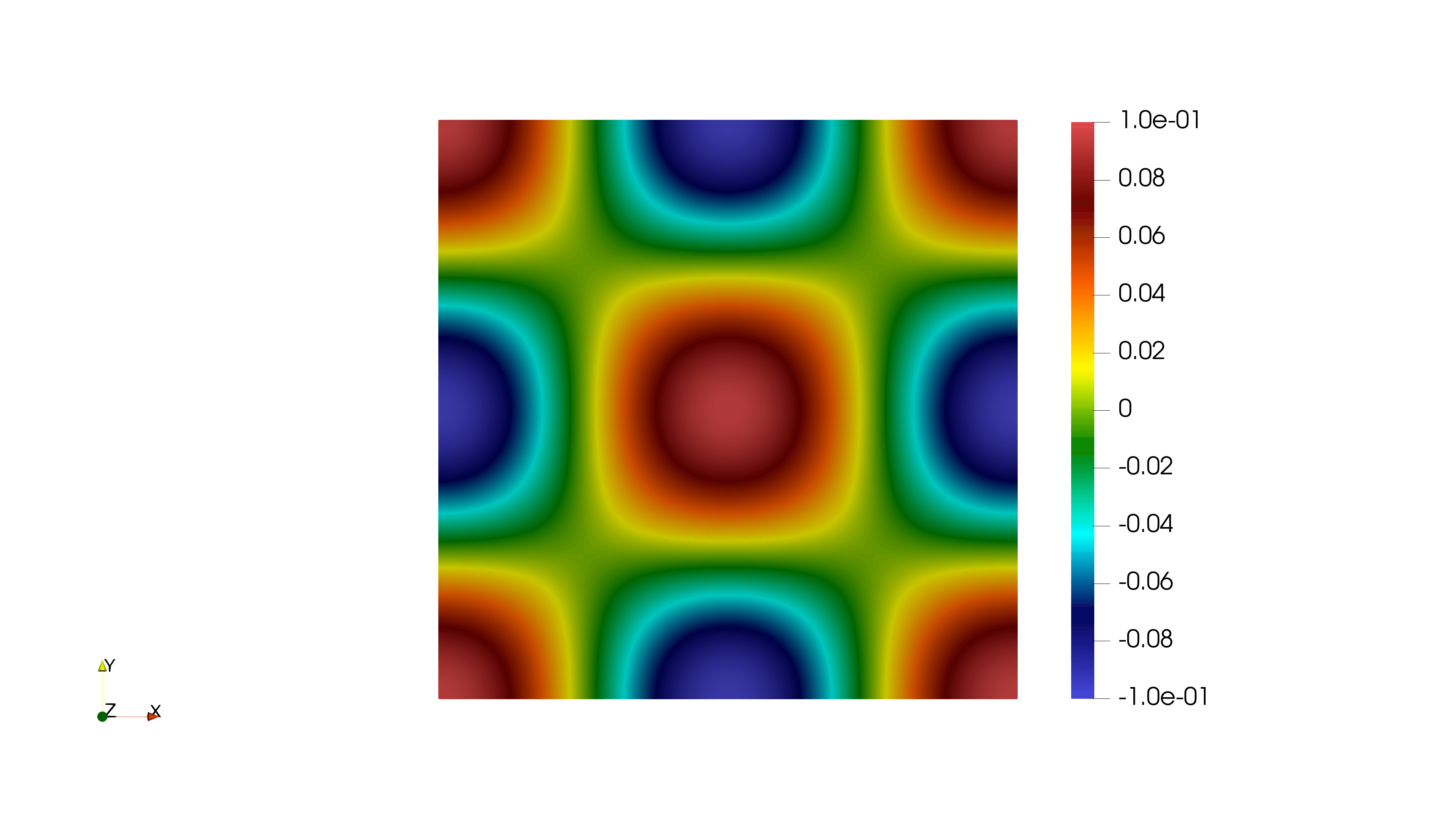}\\[-.1cm]
$t=1$ &\includegraphics[trim={39.0cm 10.0cm 22.0cm 10.0cm},clip,scale=0.055]{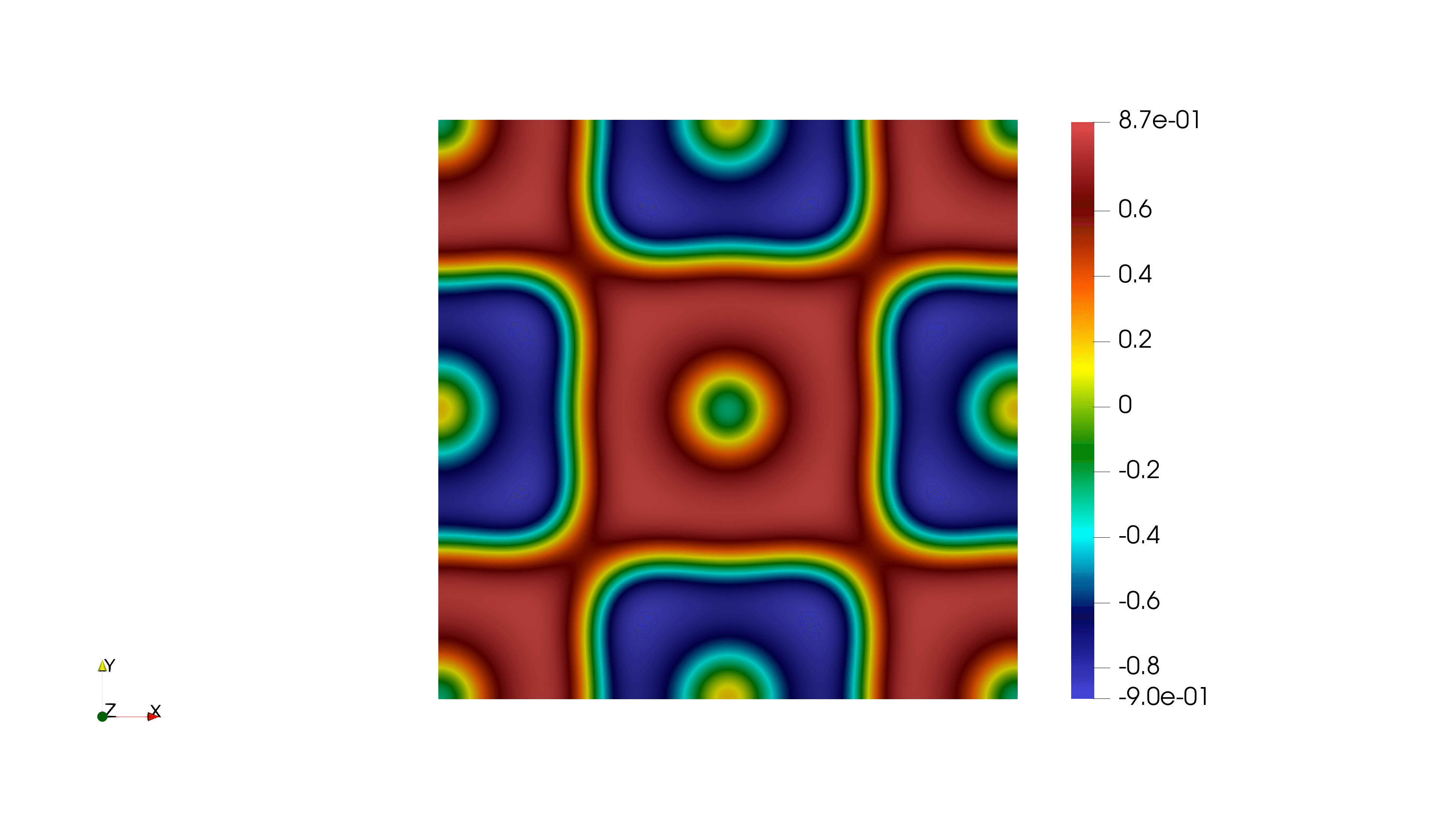}
&\includegraphics[trim={39.0cm 10.0cm 22.0cm 10.0cm},clip,scale=0.055]{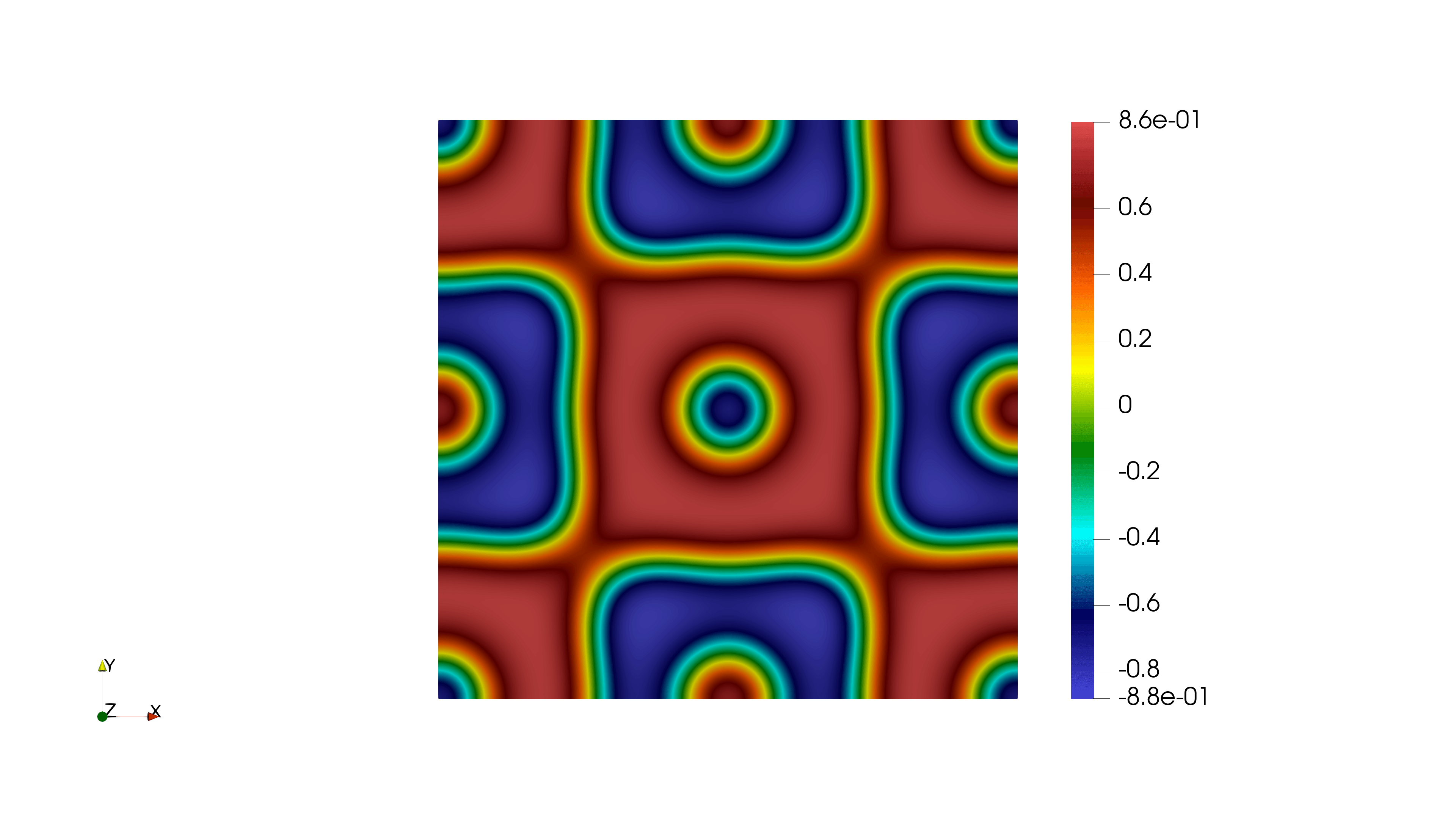}
&\includegraphics[trim={39.0cm 10.0cm 22.0cm 10.0cm},clip,scale=0.055]{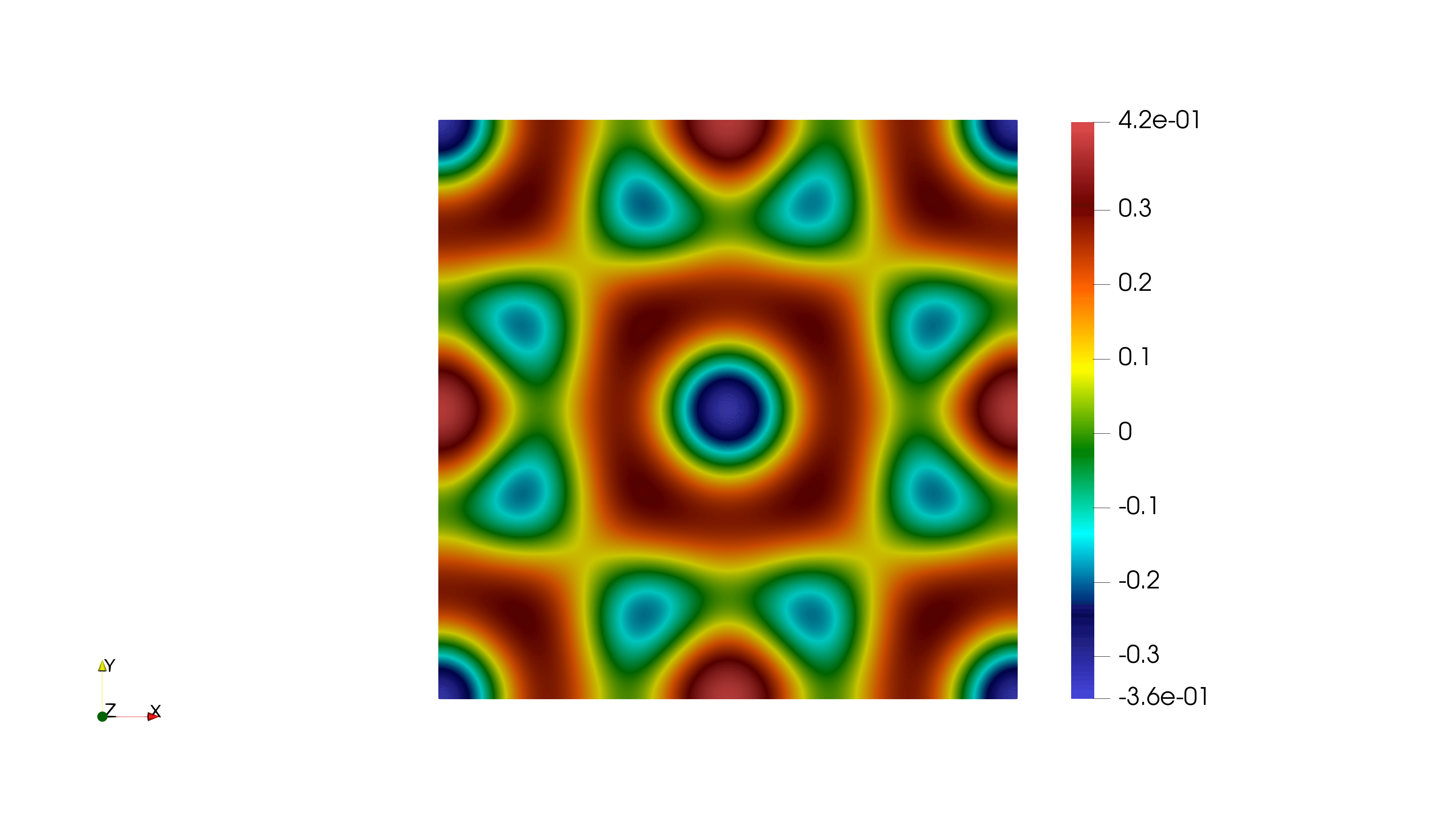}\\[-.1cm]
$t=2$ &\includegraphics[trim={39.0cm 10.0cm 22.0cm 10.0cm},clip,scale=0.055]{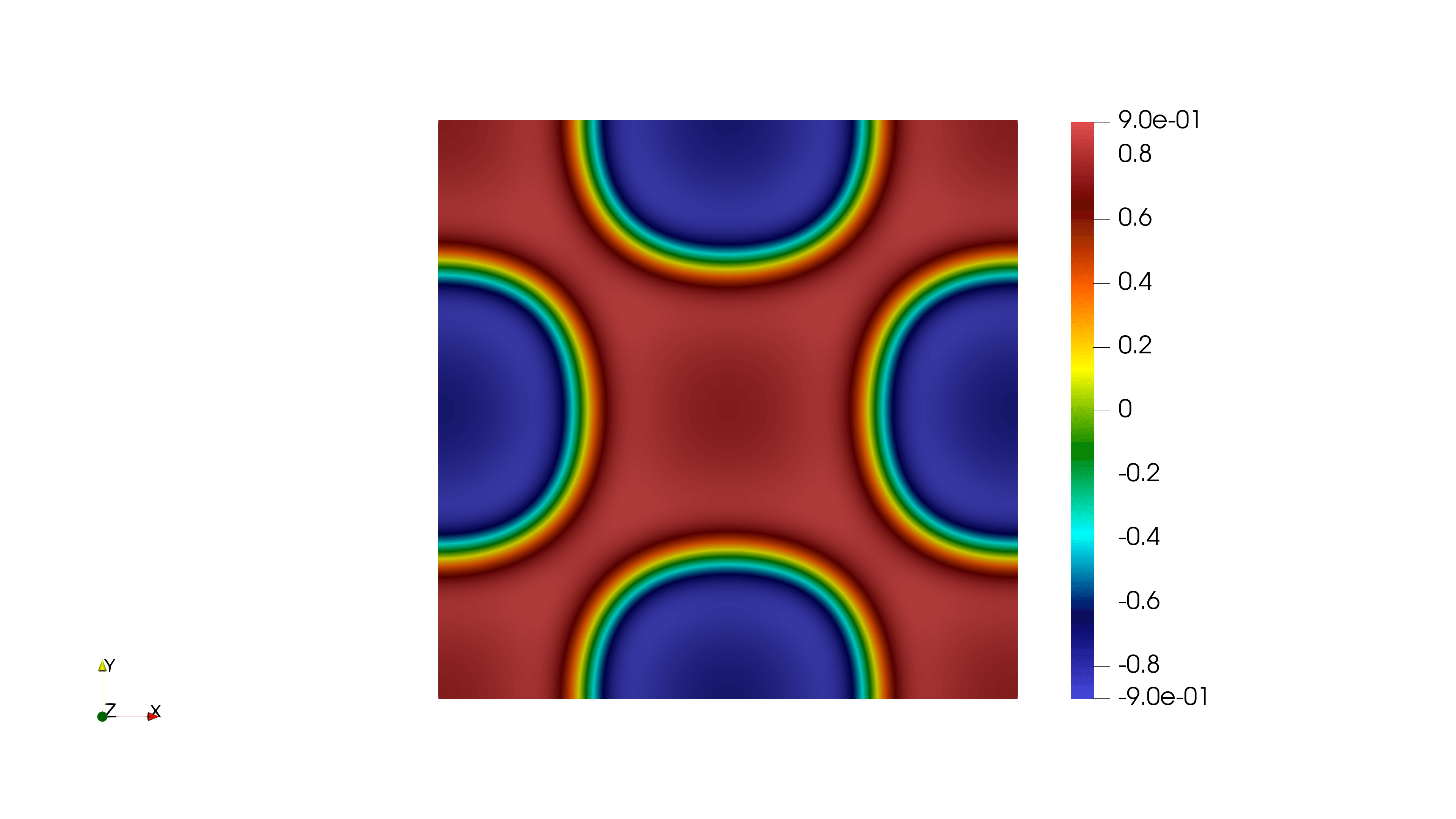}
&\includegraphics[trim={39.0cm 10.0cm 22.0cm 10.0cm},clip,scale=0.055]{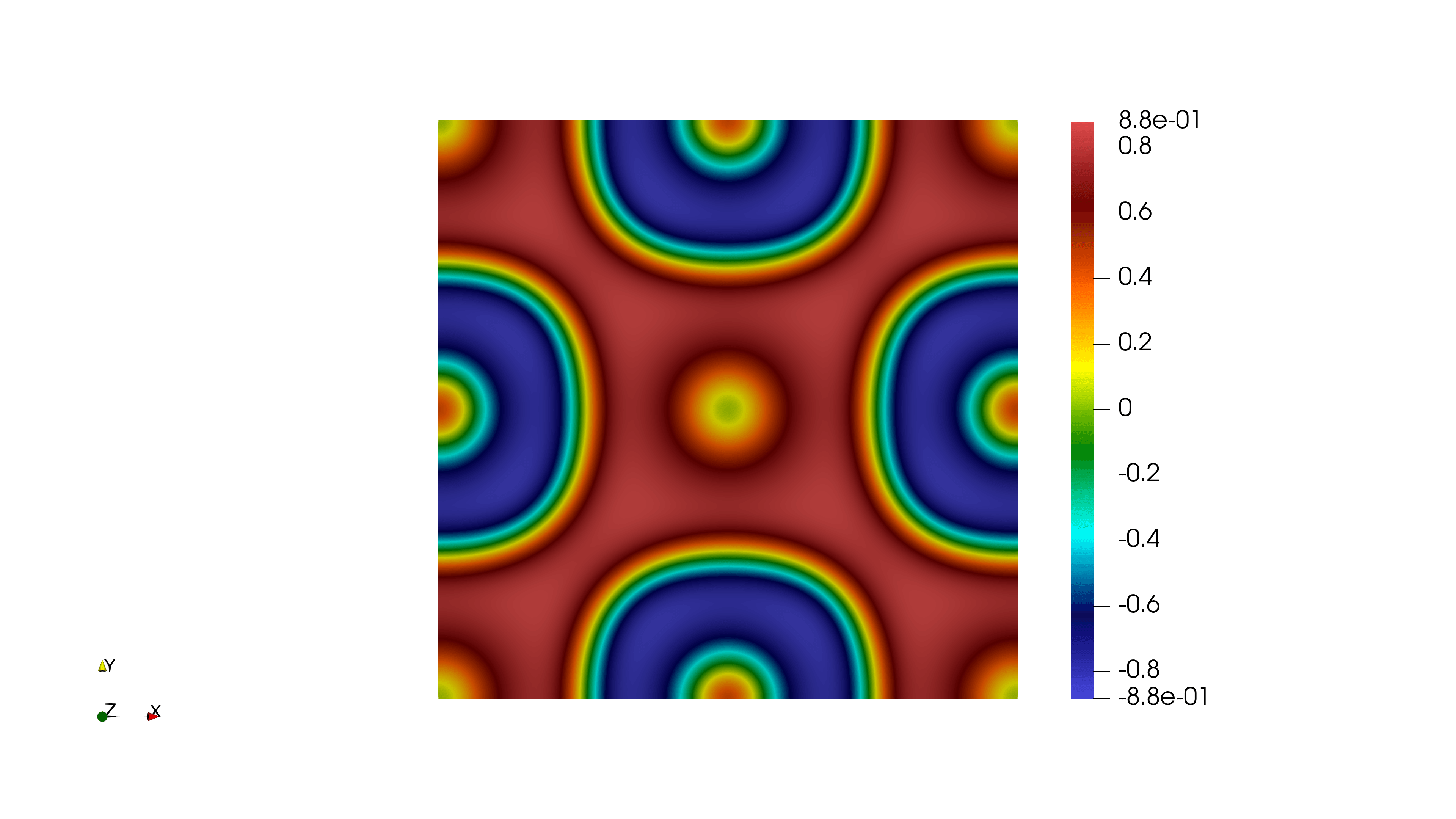}
&\includegraphics[trim={39.0cm 10.0cm 22.0cm 10.0cm},clip,scale=0.055]{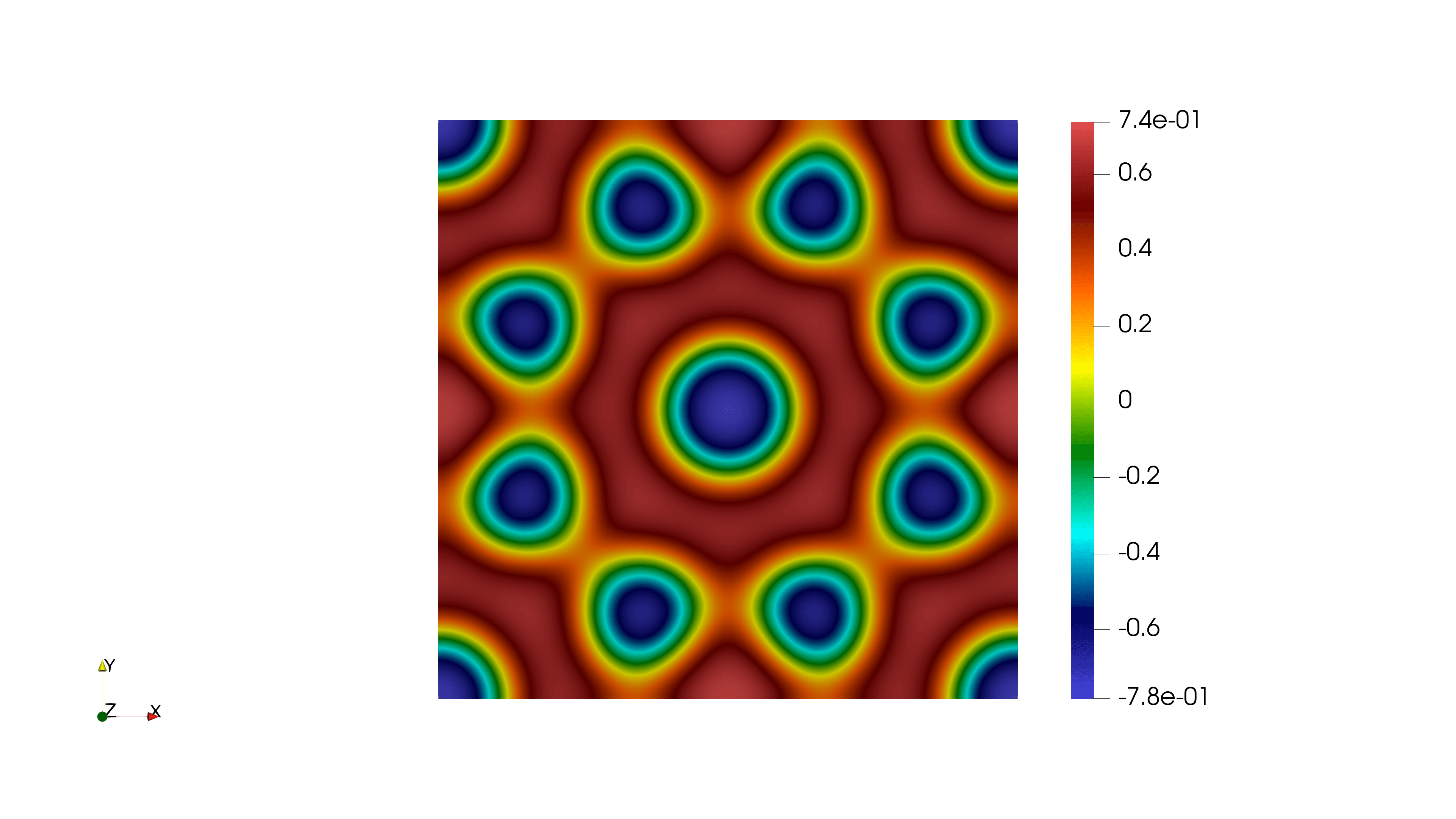}\\[-.1cm]
$t=5$ &\includegraphics[trim={39.0cm 10.0cm 22.0cm 10.0cm},clip,scale=0.055]{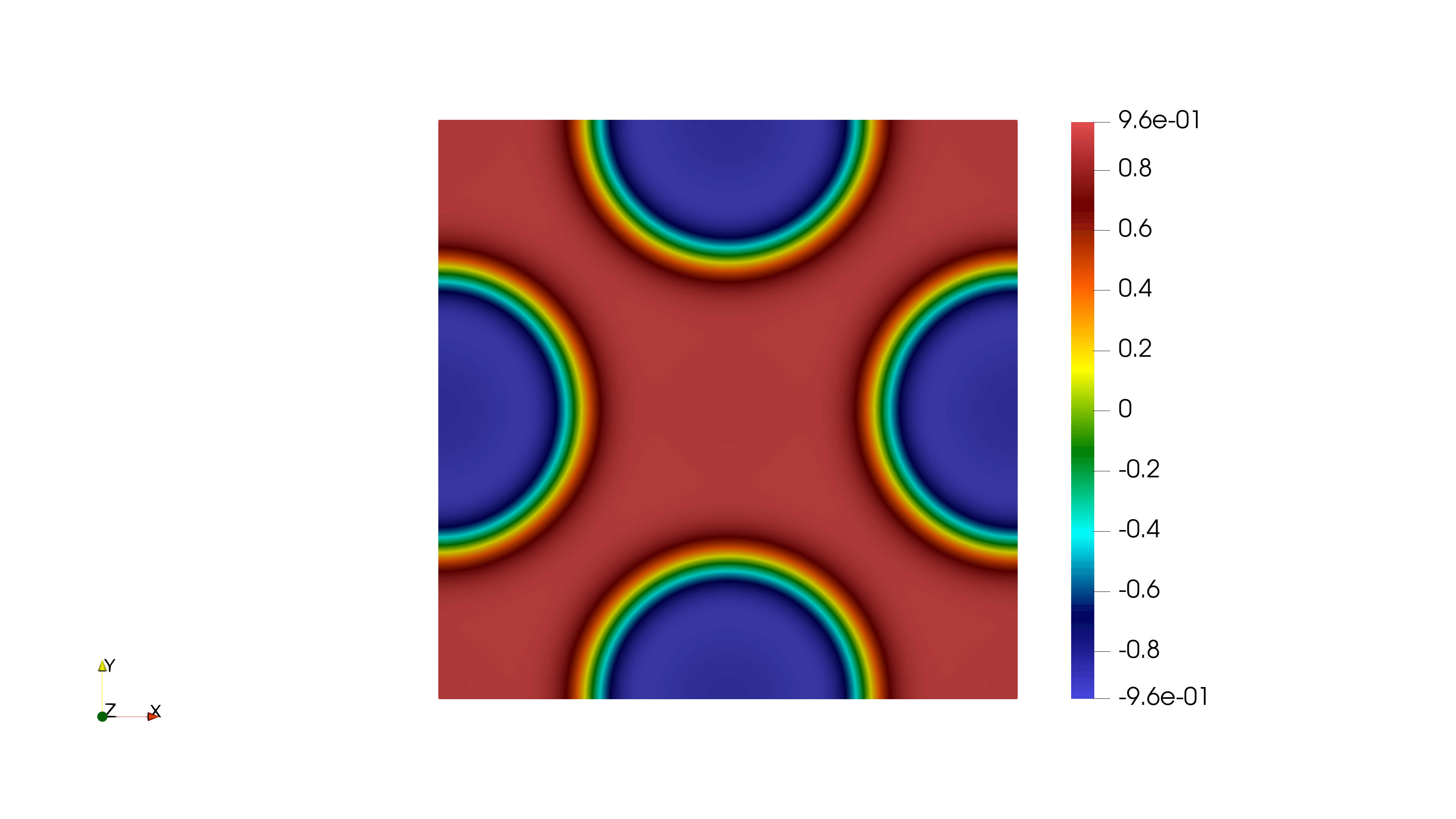}
&\includegraphics[trim={39.0cm 10.0cm 22.0cm 10.0cm},clip,scale=0.055]{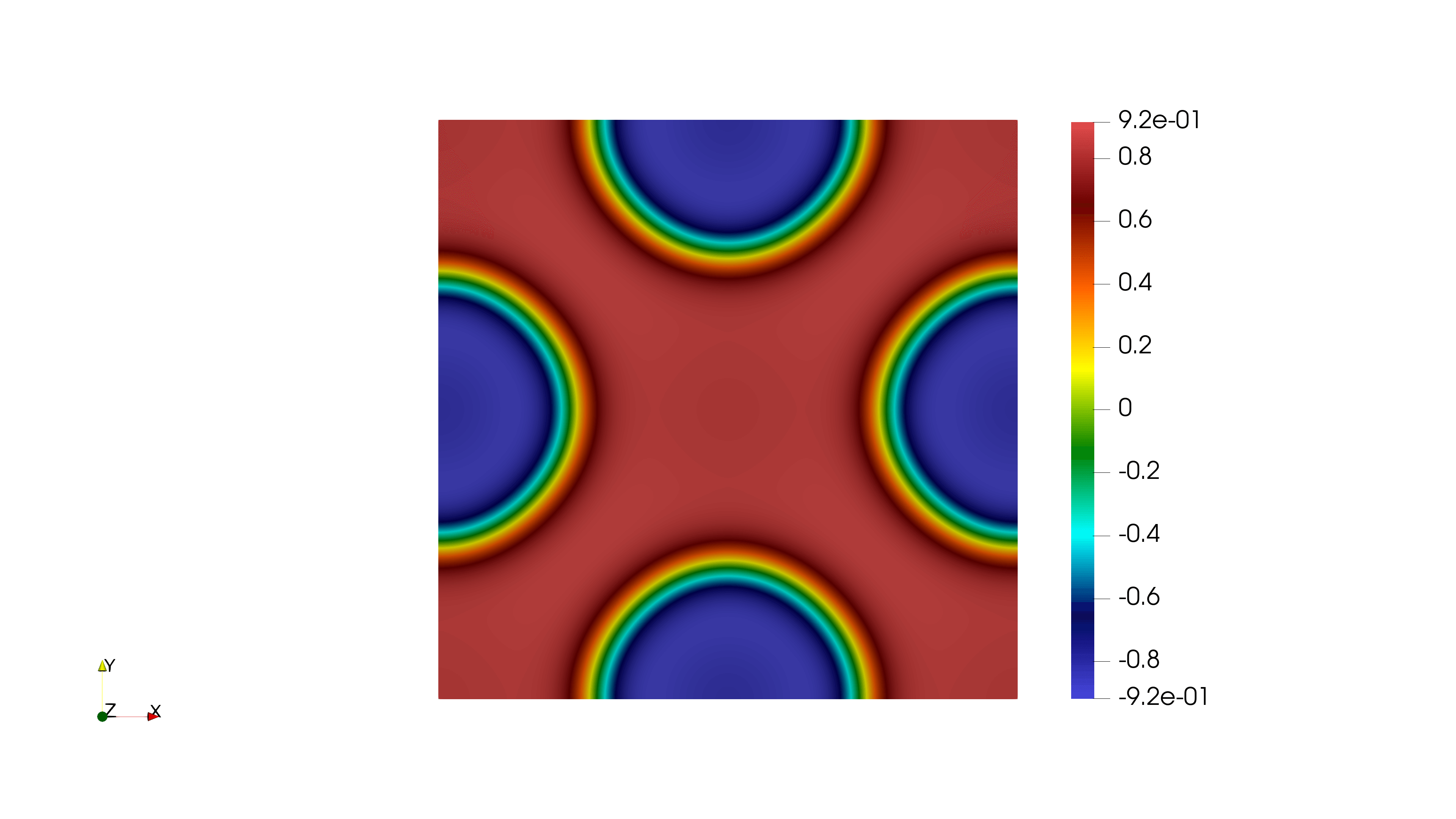}
&\includegraphics[trim={39.0cm 10.0cm 22.0cm 10.0cm},clip,scale=0.055]{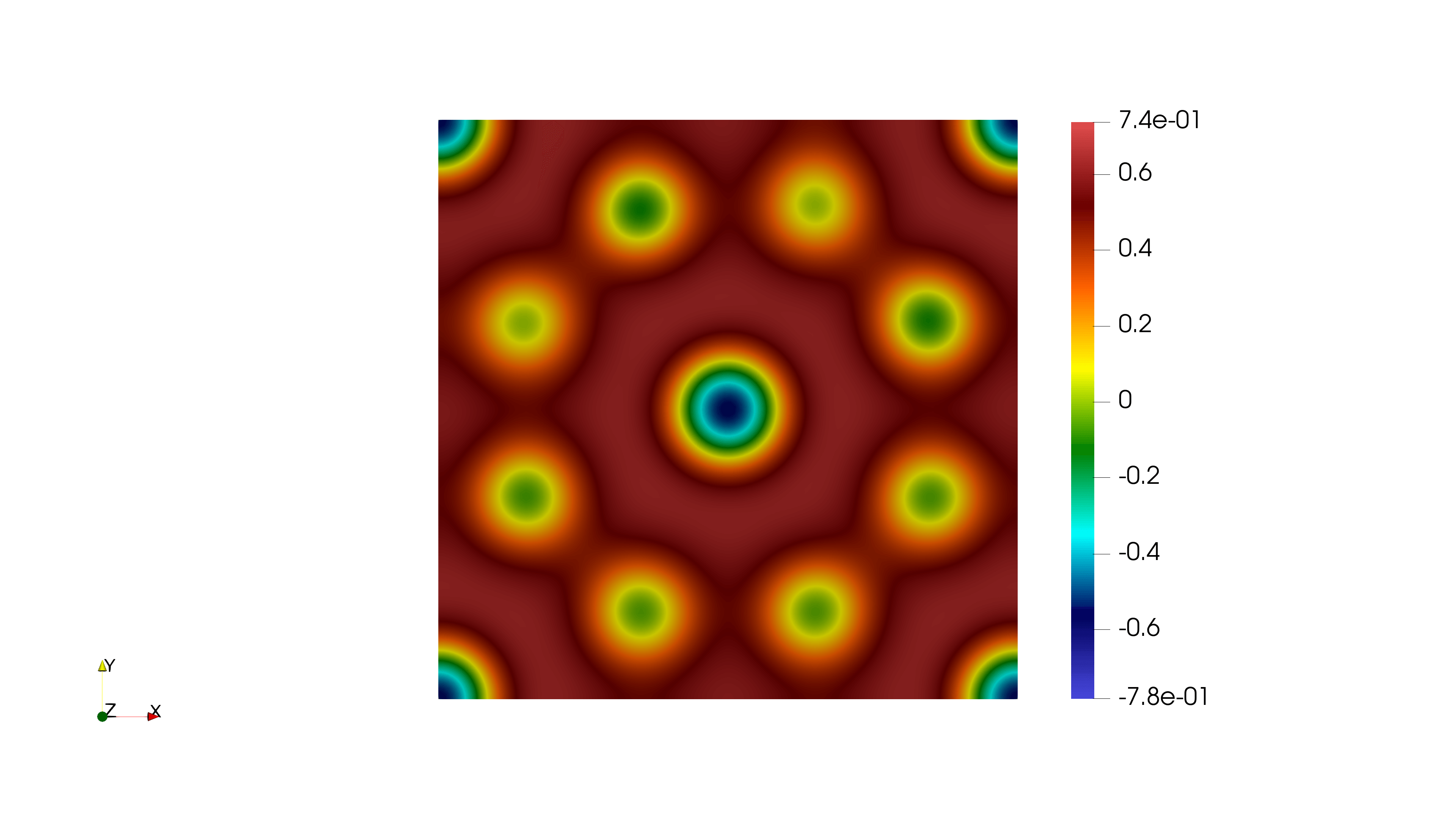}
\end{tabular}
\end{center}
\caption{\cref{exp:2} for the Ohta--Kawasaki model: Snapshots of the phase-field $\phi$ for $\kappa\in\{0,10,100\}$ (left to right) at times $t\in\{0,1,2,5\}$ (top to bottom).  \label{pic:exp2}}
\end{figure}

\begin{figure}[htbp!]
\centering
\includegraphics[trim={0.0cm 0.0cm 0.0cm 0.0cm},clip,scale=0.45]{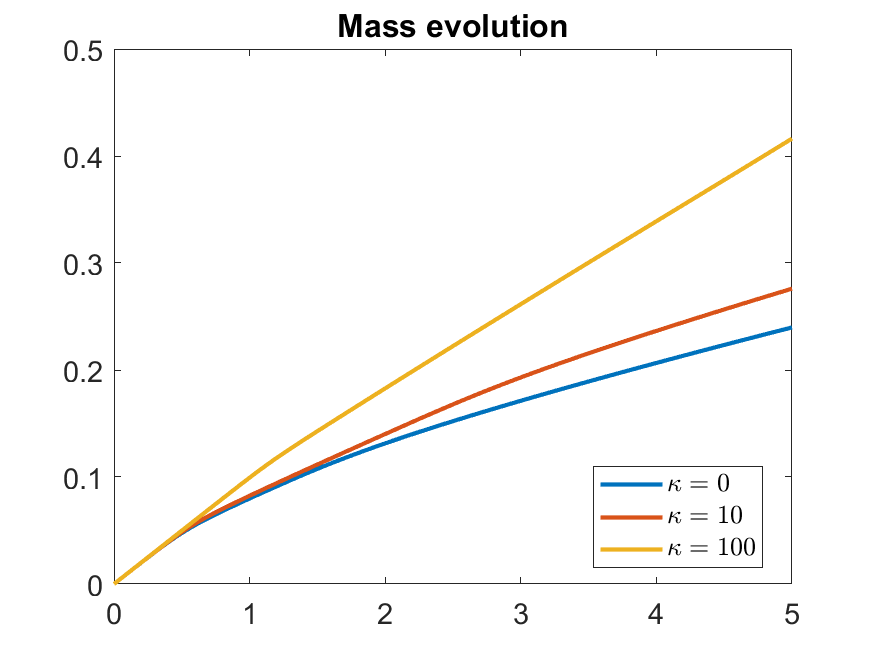}
\includegraphics[trim={0.0cm 0.0cm 0.0cm 0.0cm},clip,scale=0.45]{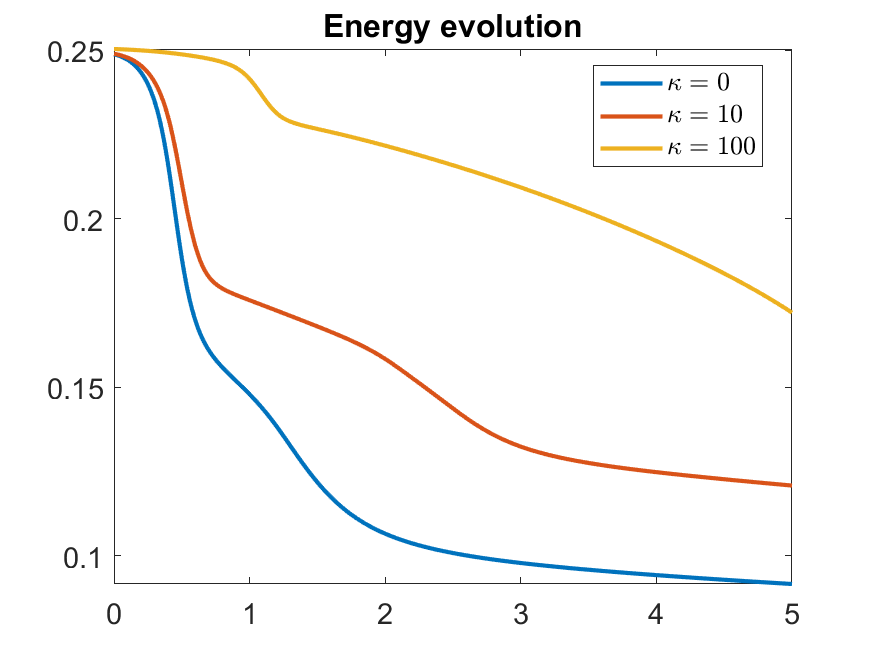}
\caption{Experiment \ref{exp:2} for the Ohta--Kawasaki model: Evolution of total mass (left) and the energy $\mathcal{E}$ (right).  \label{pic:exp2struc}}
\end{figure}

\noindent\textbf{Convergence test.}
In the case $\kappa=100$ in \cref{exp:2}, we performed a convergence test up to $T=1$. Since no exact solution is available, we compute the error using the finest space resolution as a reference solution. The error quantities we consider are given as:
\begin{align*}
   \text{err}(\phi,h_k)&:= \max_{n\in\Itau}\norm{\phi^n_{h_k}-\phi^n_{\text{ref}}}_V^2 , \\
   \text{err}(\mu,h_k)&:=\tau\sum_{n=1}^{n_T}\norm{\mu^n_{h_k}-\mu^n_{\text{ref}}}_0^2, \\
   \text{err}(\nu,h_k)&:=\max_{n\in\Itau}\norm{\nu^n_{h_k}-\nu^n_{\text{ref}}}_V^2.
\end{align*}
To this end, we consider mesh refinements with the mesh sizes $h_k \approx 2^{-(k+1)}$ for $k=0,\ldots,7$ and fix $h_{\text{ref}}\approx2^{-8}$. Regarding the size of the time step, we choose $\tau=10^{-3}$. The errors and approximate convergence rates are presented in \cref{tab:rates1}, demonstrating optimal first-order convergence in the standard energy norms. These results validate the accuracy and consistency of the numerical scheme with the theoretical expectations of this model.

\begin{table}[htbp!]
\centering
\small
\caption{Error and rates for \cref{exp:2} in the case $\kappa=100$.\label{tab:rates1}} 
\begin{tabular}{c||c|c|c|c|c|c}
$ k $ & $\text{err}(\phi,h_k)$ &  eoc & $\text{err}(\mu,h_k)$ & eoc & $\text{err}(\eta,h_k)$ &  eoc   \\
\hline\hline
1 & $5.68\cdot 10^{+1}$ & --       &  $2.54\cdot 10^{-1}$ & --       & $3.89\cdot 10^{-4}$ & --  \\
2 & $1.06\cdot 10^{+2}$ & -0.91    &  $3.84\cdot 10^{-1}$ & -0.60    & $5.81\cdot 10^{-4}$ & -0.58 \\
3 & $5.56\cdot 10^{+1}$ & \phm0.94 &  $2.29\cdot 10^{-1}$ & \phm0.74 & $2.91\cdot 10^{-4}$ & \phm1.00\\
4 & $1.02\cdot 10^{+1}$ & \phm2.44 &  $4.19\cdot 10^{-2}$ & \phm2.45 & $4.83\cdot 10^{-5}$ & \phm2.59  \\
5 & $1.69\cdot 10^{+0}$ & \phm2.59 &  $7.37\cdot 10^{-3}$ & \phm2.51 & $8.38\cdot 10^{-6}$ & \phm2.52  \\
6 & $3.97\cdot 10^{-1}$ & \phm2.09 &  $1.68\cdot 10^{-3}$ & \phm2.13 & $2.04\cdot 10^{-6}$ & \phm2.04  \\
7 & $9.87\cdot 10^{-2}$ & \phm2.01 &  $4.14\cdot 10^{-4}$ & \phm2.02 & $5.13\cdot 10^{-7}$ & \phm1.99  \\
\end{tabular}
\end{table}

\subsection{3D experiment}

In this subsection, we consider a three-dimensional example and compare the well-established Cahn--Hilliard equation to the Ohta--Kawasaki model.

\begin{experiment}\label{exp:3} 
$\phi_0(x)=-0.1 + \mathcal{U}(x,y,-10^{-3},10^{-3})$, \quad $f(\phi)=\dfrac{1}{10}\max\{0,1-\phi^2\}$, \quad $\kappa\in\{0,100\}$.
\end{experiment}

\begin{figure}[htbp!]
\begin{center}
\begin{tabular}{cM{.3\textwidth}M{.3\textwidth}c}
			&$\kappa=0$&$\kappa=100$& \\
$t=0.5$ &\includegraphics[trim={41.0cm 7.5cm 41.0cm 12.5cm},clip,scale=0.08]{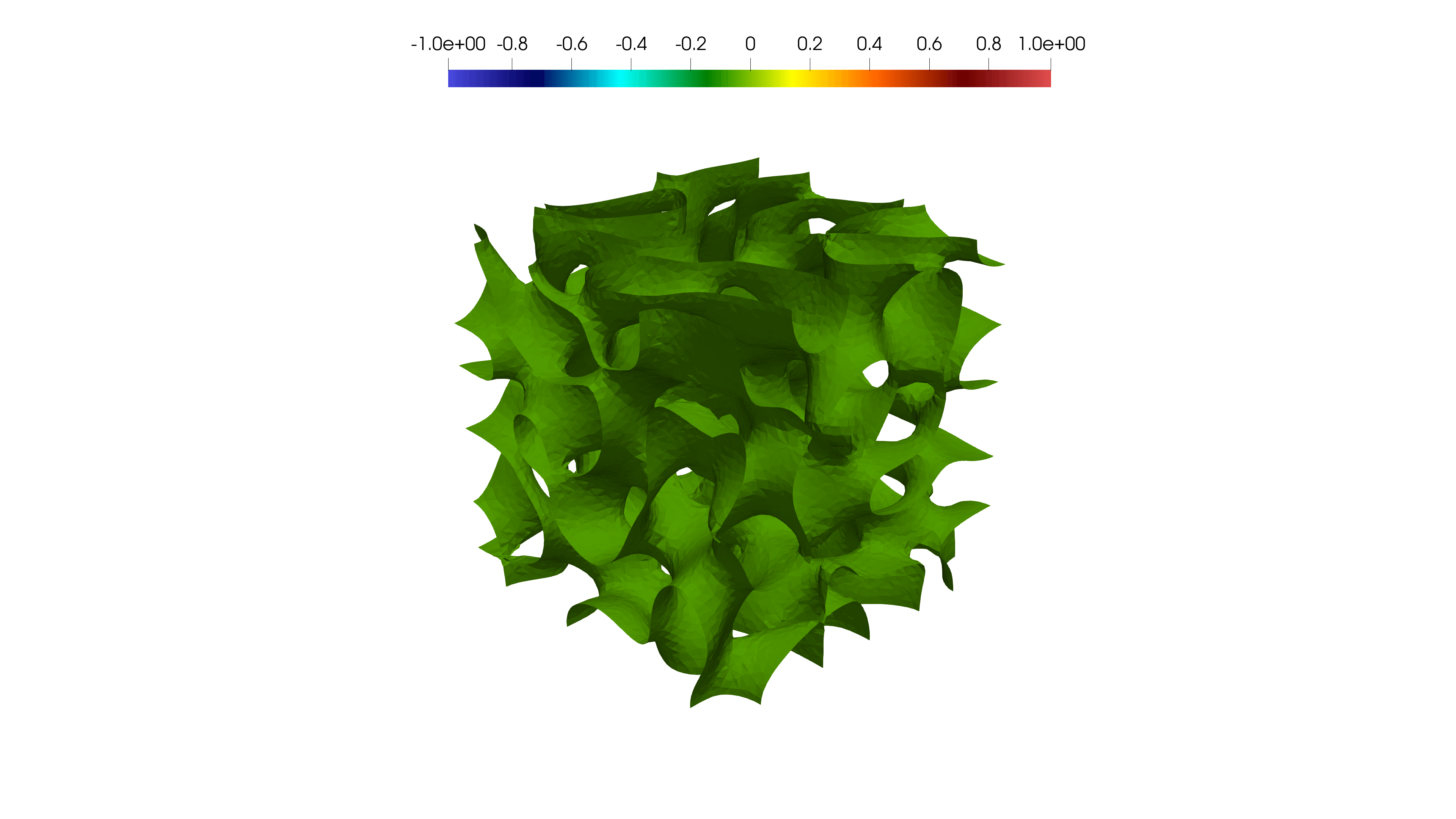}
&\includegraphics[trim={35.0cm 2.0cm 35.0cm 6.1cm},clip,scale=0.065]{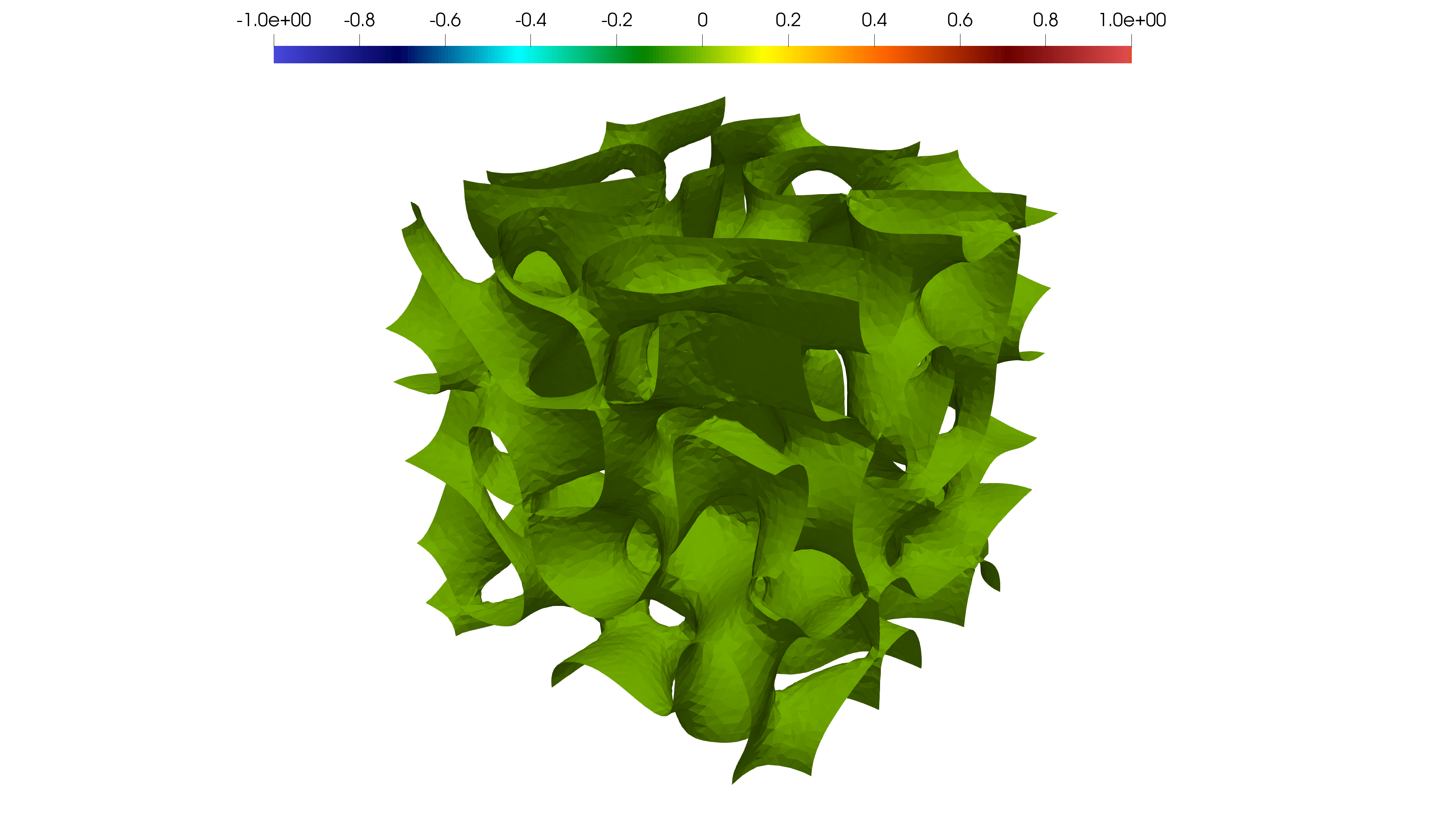} &$t=1$\\[.1cm]
$t=1$ &\includegraphics[trim={41.0cm 7.5cm 41.0cm 12.5cm},clip,scale=0.08]{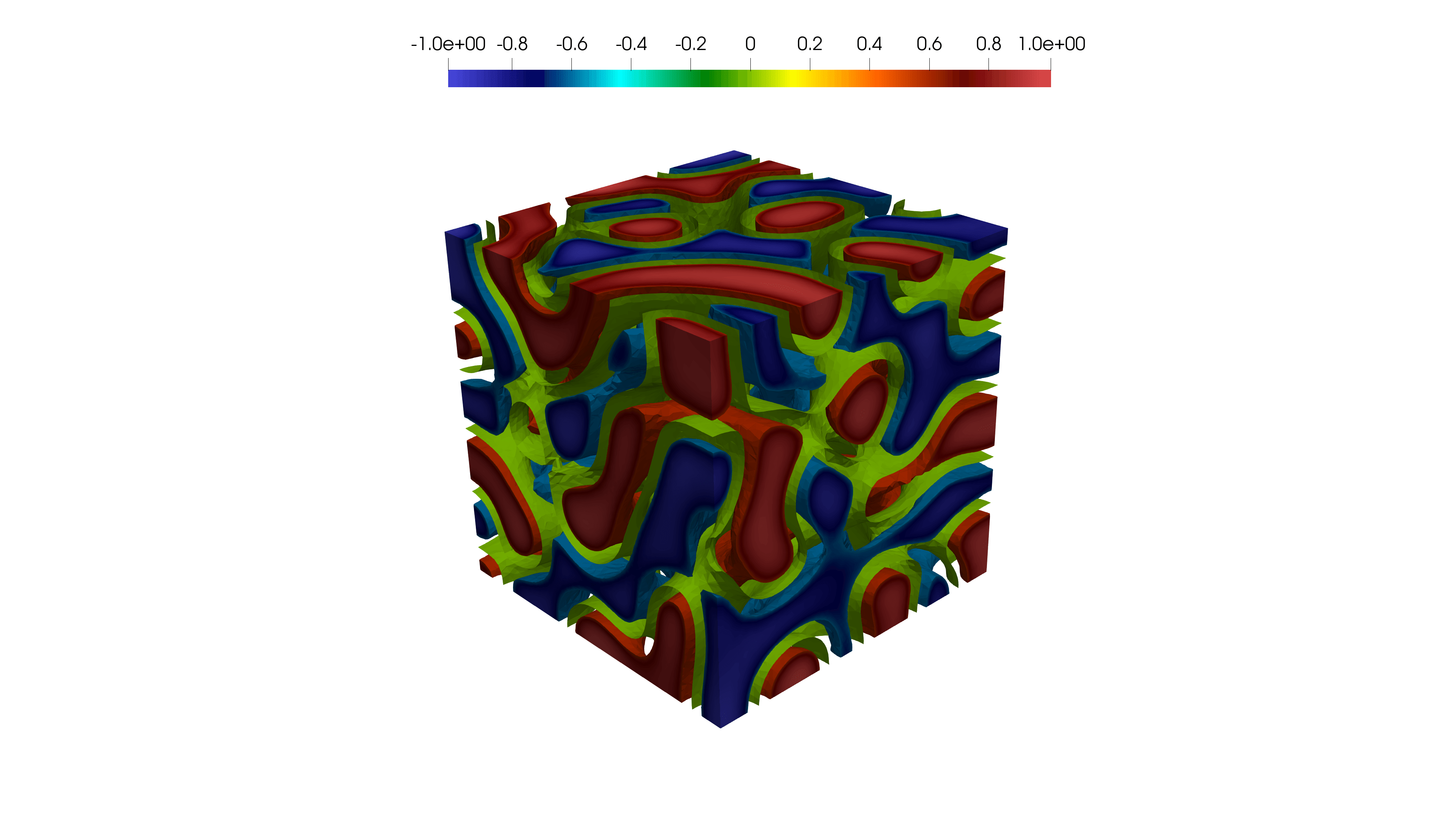}
&\includegraphics[trim={35.0cm 2.0cm 35.0cm 6.1cm},clip,scale=0.065]{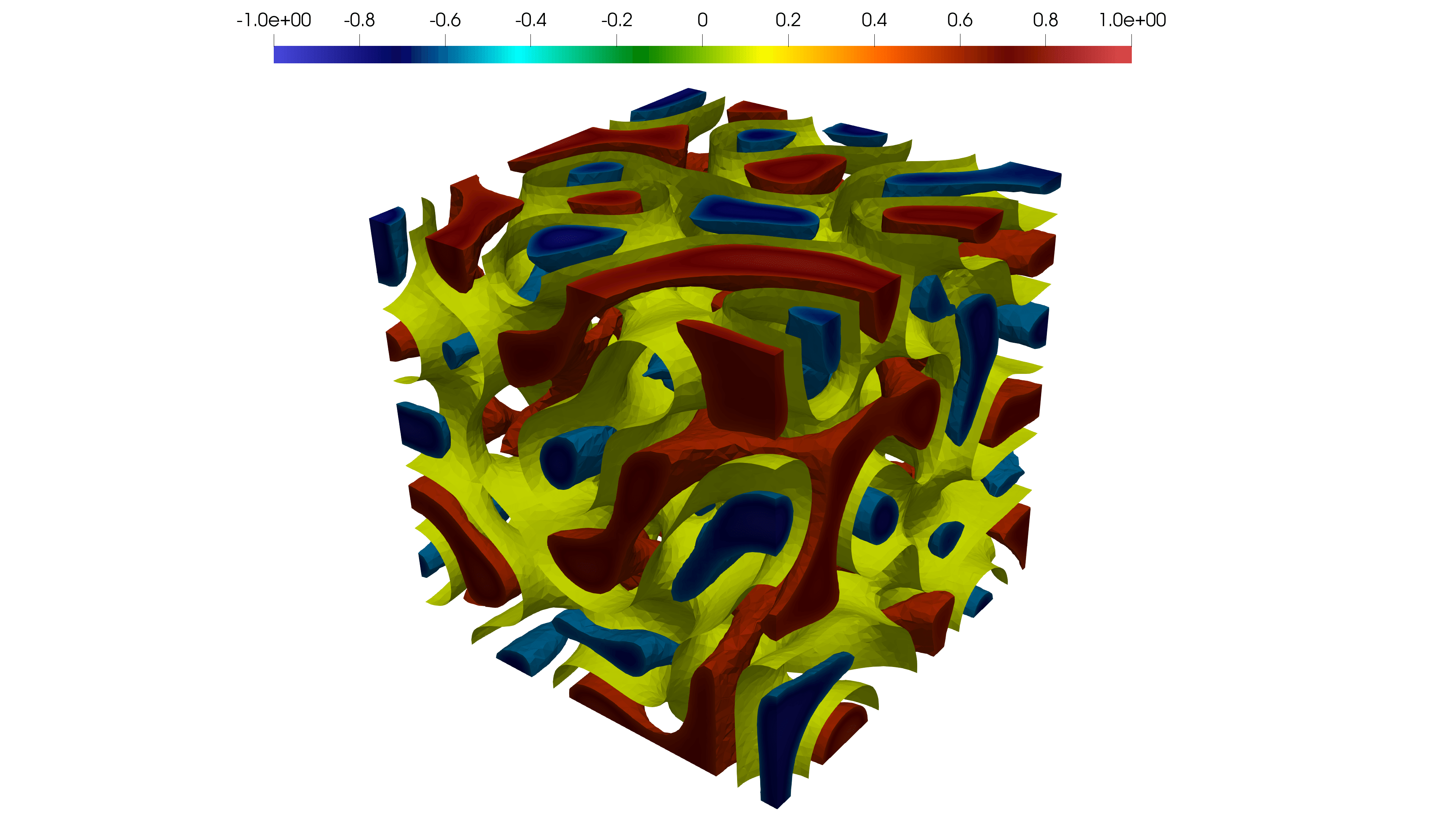} &$t=2$\\[.1cm]
$t=3$ &\includegraphics[trim={41.0cm 7.5cm 41.0cm 12.5cm},clip,scale=0.08]{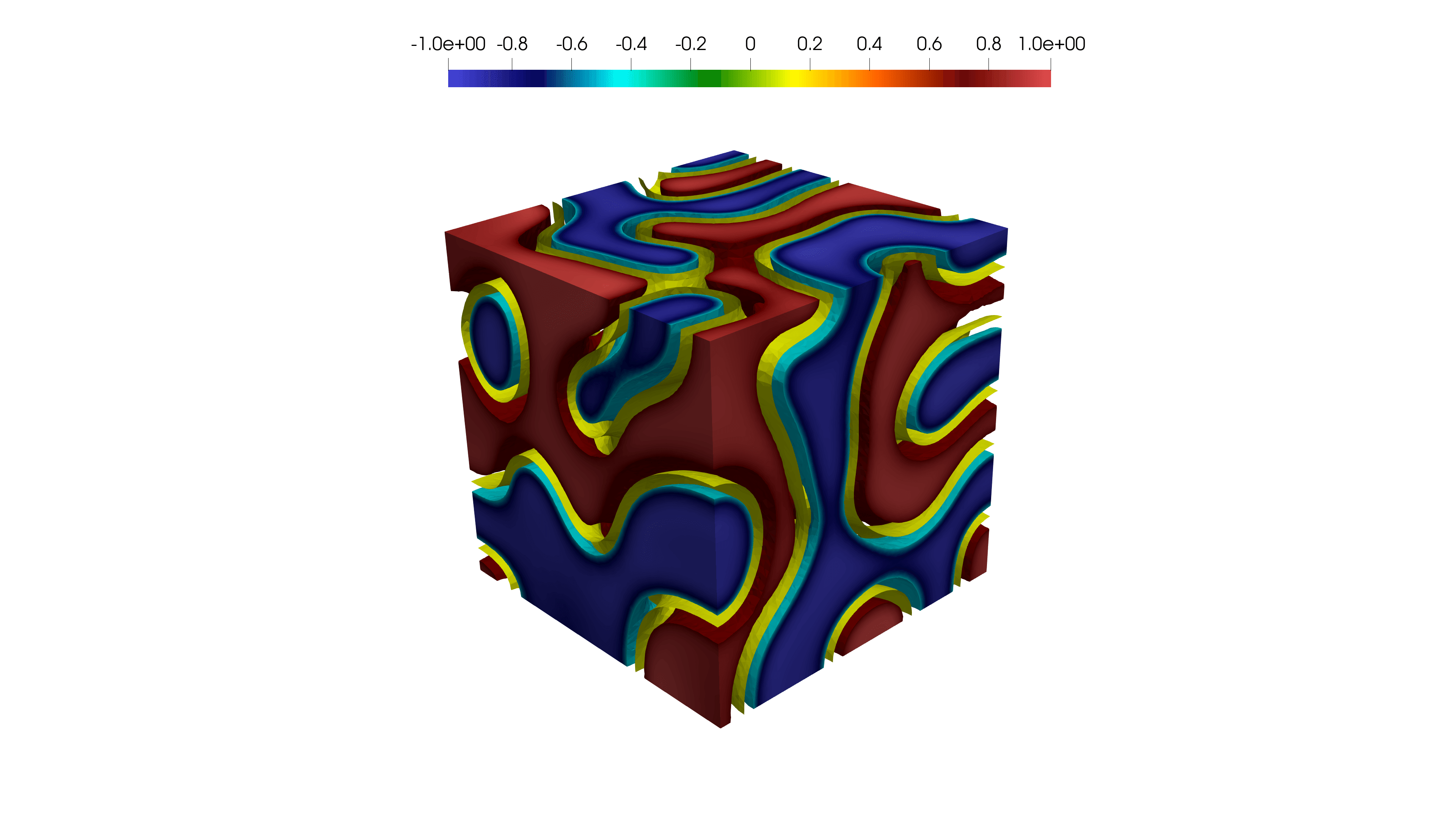}
&\includegraphics[trim={35.0cm 2.0cm 35.0cm 6.1cm},clip,scale=0.065]{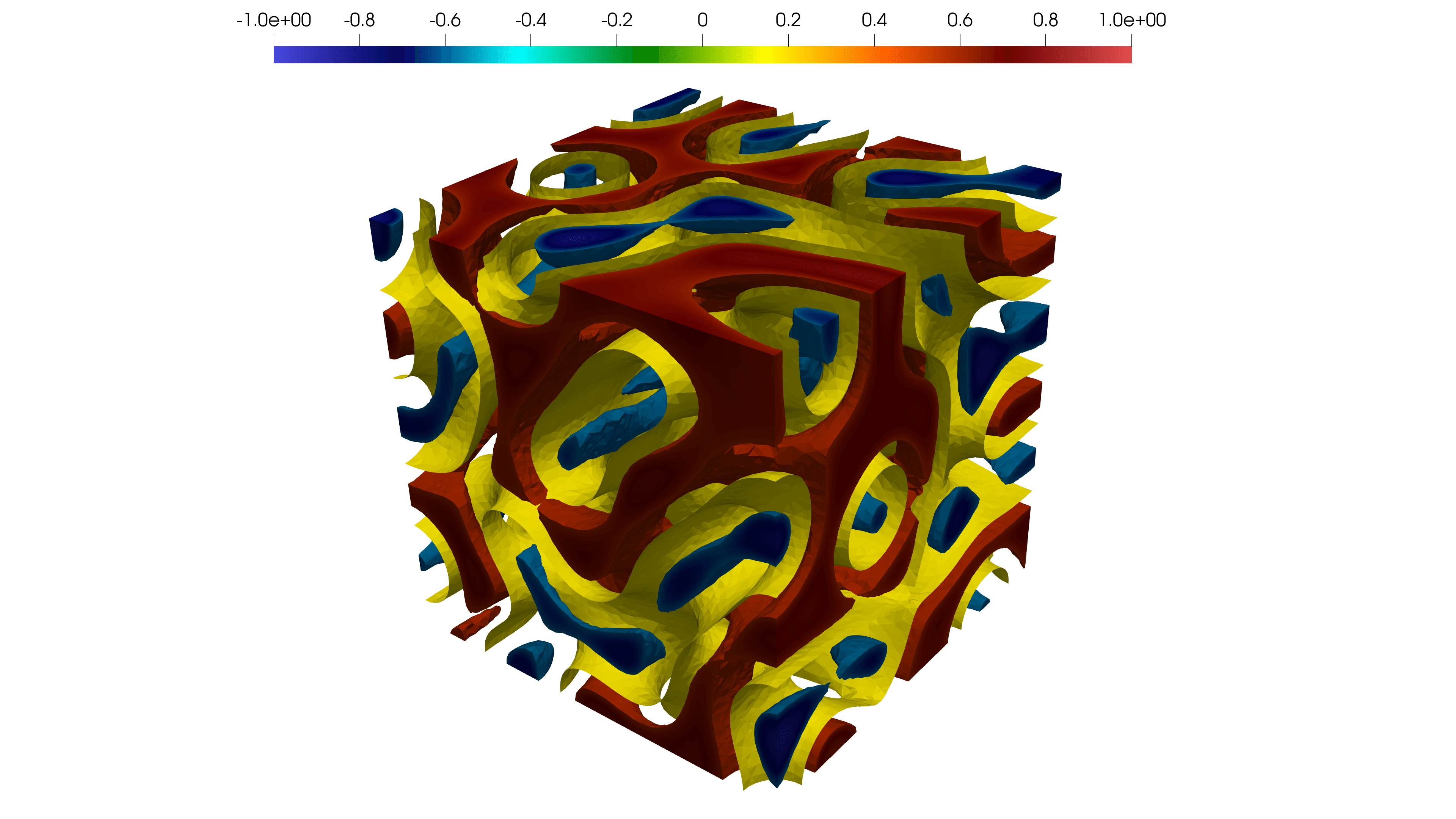} &$t=3$\\[.1cm]
$t=5$ &\includegraphics[trim={39.0cm 5.5cm 39.0cm 10.5cm},clip,scale=0.075]{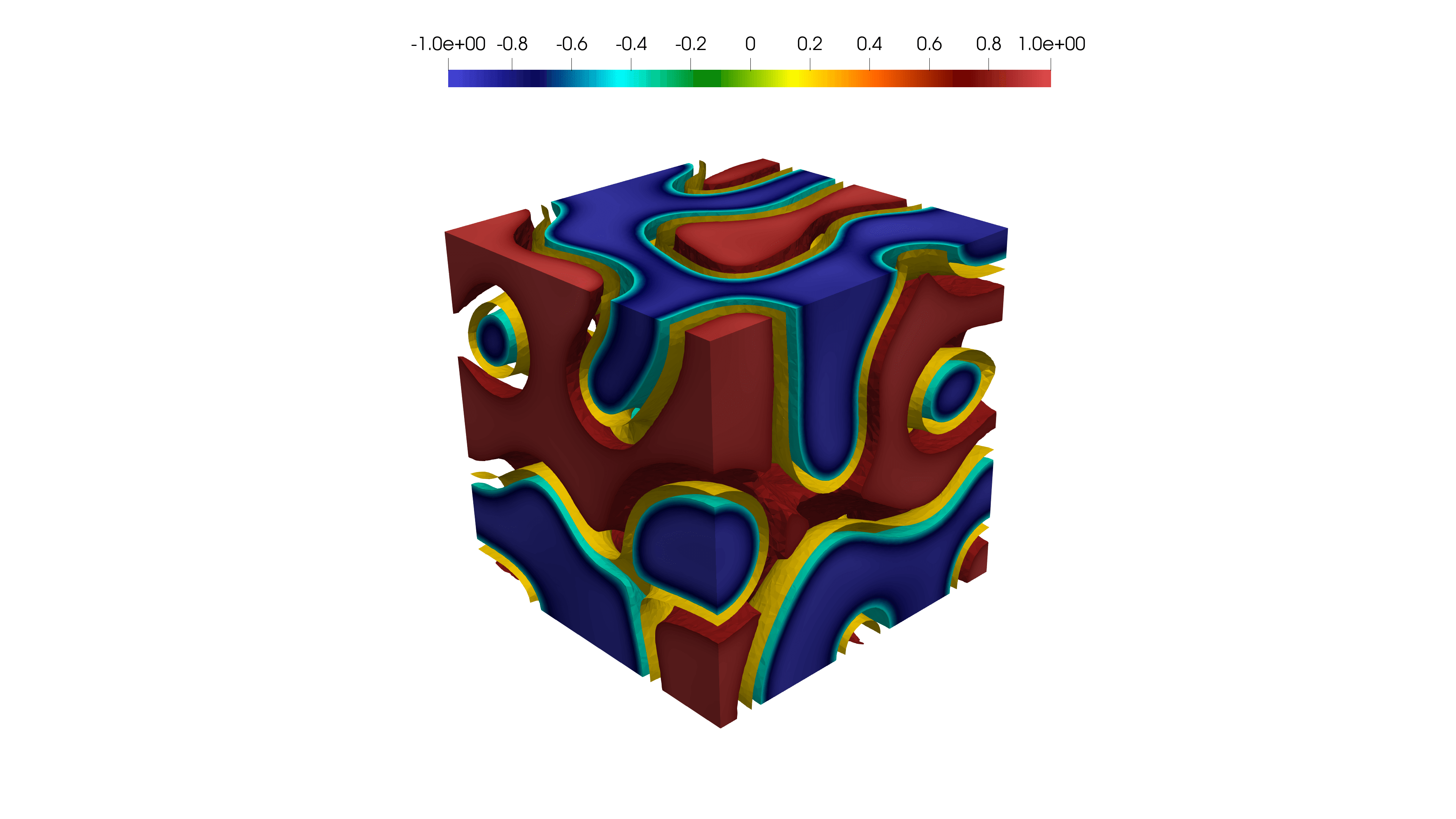}
&\includegraphics[trim={35.0cm 2.0cm 35.0cm 6.1cm},clip,scale=0.065]{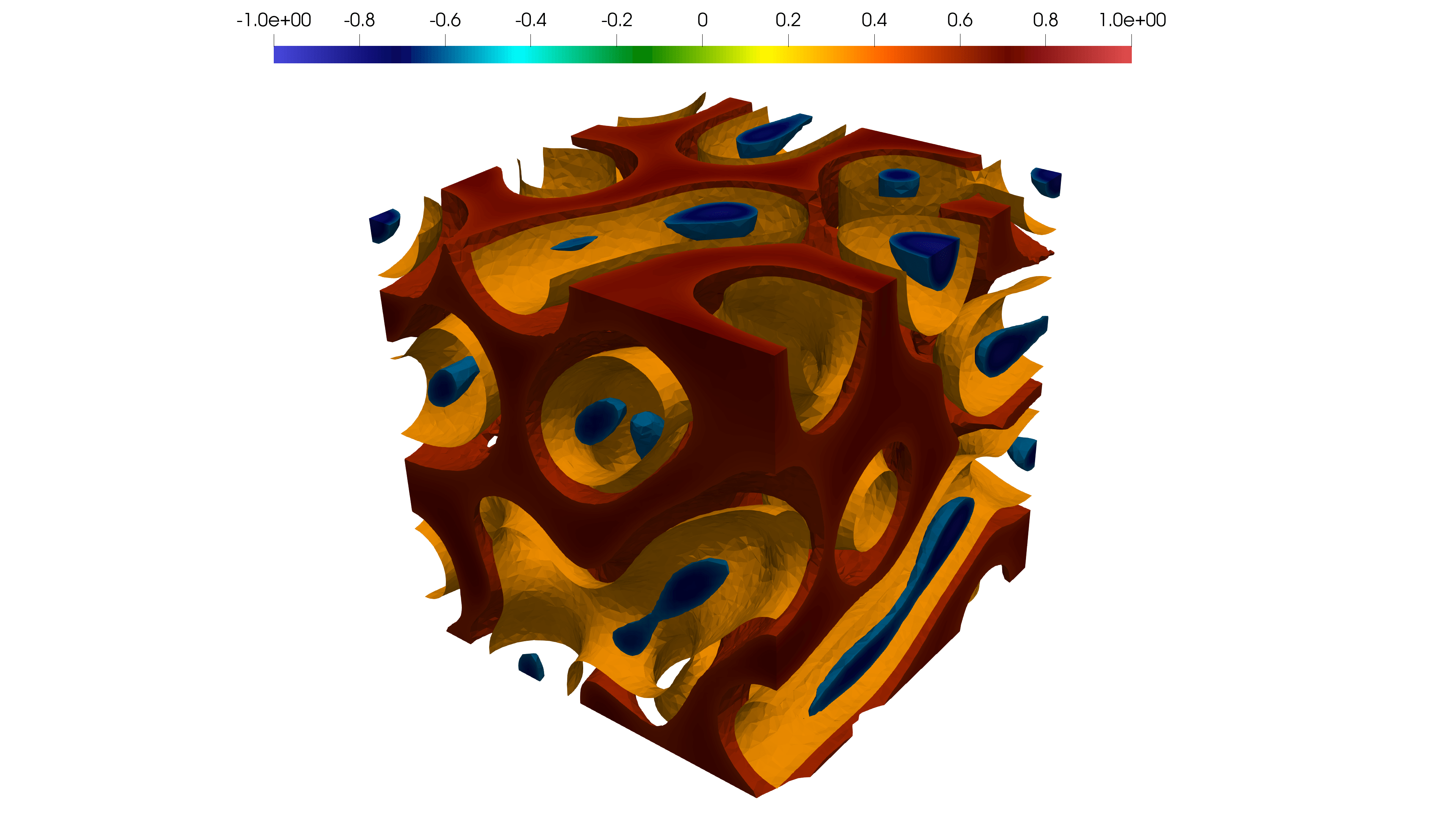} &$t=5$
\end{tabular}
\includegraphics[trim={22.0cm 70.0cm 27.0cm 0cm},clip,scale=0.15]{pics/3D_100.0100.png} 
\end{center}
\caption{\cref{exp:3} for the Ohta--Kawasaki model: Snapshots of the phase-field $\phi$ for $\kappa\in\{0,100\}$. Left: $\kappa=0$ at times $t\in\{0.5,1,3,5\}$. Right: $\kappa=100$ at the times $t\in\{1,2,3,5\}$.  \label{pic:exp3}}
\end{figure}

The temporal evolution of the phase-field \(\phi\) is depicted in \cref{pic:exp3} using isosurfaces located at \(\phi \in \{-0.6, \bar{\phi}, 0.6\}\). For the case \(\kappa=0\), we observe typical phase separation driven by the Cahn--Hilliard dynamics. The mean value \(\bar{\phi}\) increases from \(-0.1\) to approximately \(0.2\), as shown in the plot on the left in \cref{pic:exp3struc}. In the Ohta--Kawasaki system with a repulsion strength of \(\kappa=100\), the mean phase-field value \(\bar{\phi}\) increases further, from \(\bar{\phi} \approx -0.1\) to around \(0.3\) over time. The phase separation process remains notably slow and even after \(t=5\), the separation has not progressed to the same extent as observed for \(\kappa=0\) at \(t=3\).
As illustrated in the right plot of \cref{pic:exp3struc}, the energy in both cases initially increases, reaching a peak at \(t \approx 0.5\) for \(\kappa=0\) and at \(t \approx 1.3\) for \(\kappa=100\). Thereafter, the energy begins to decay over time. In the Cahn--Hilliard case, this decay is rapid and almost instantaneous following the initial increase driven by interface minimization. In contrast, for \(\kappa = 100\), the energy decay is nearly linear, and even by \(t=5\), the energy level has not yet matched that of the Cahn--Hilliard case. 

\begin{figure}[htbp!]
\centering
\includegraphics[trim={0.0cm 0.0cm 0.0cm 0.0cm},clip,scale=0.45]{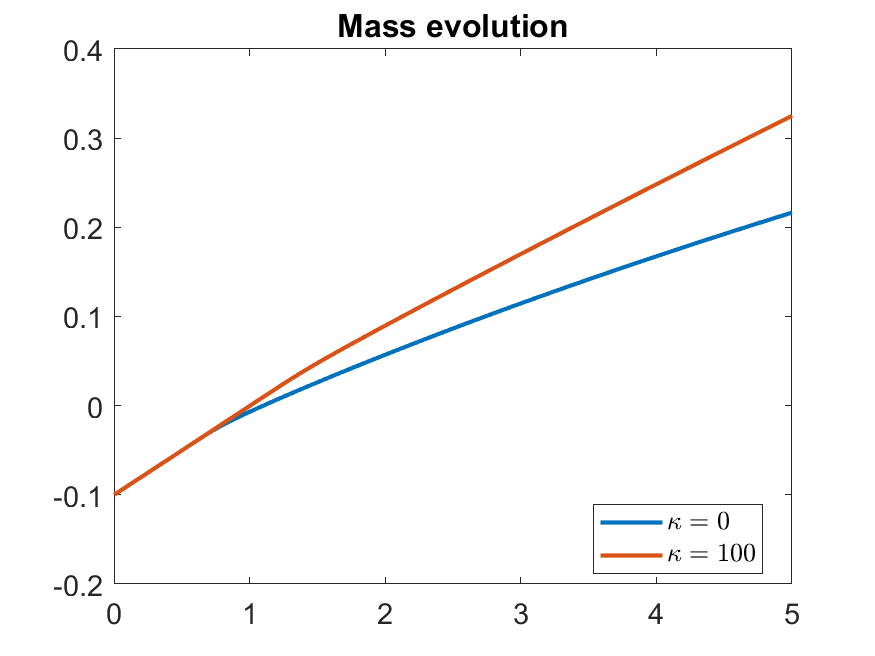}
\includegraphics[trim={0.0cm 0.0cm 0.0cm 0.0cm},clip,scale=0.45]{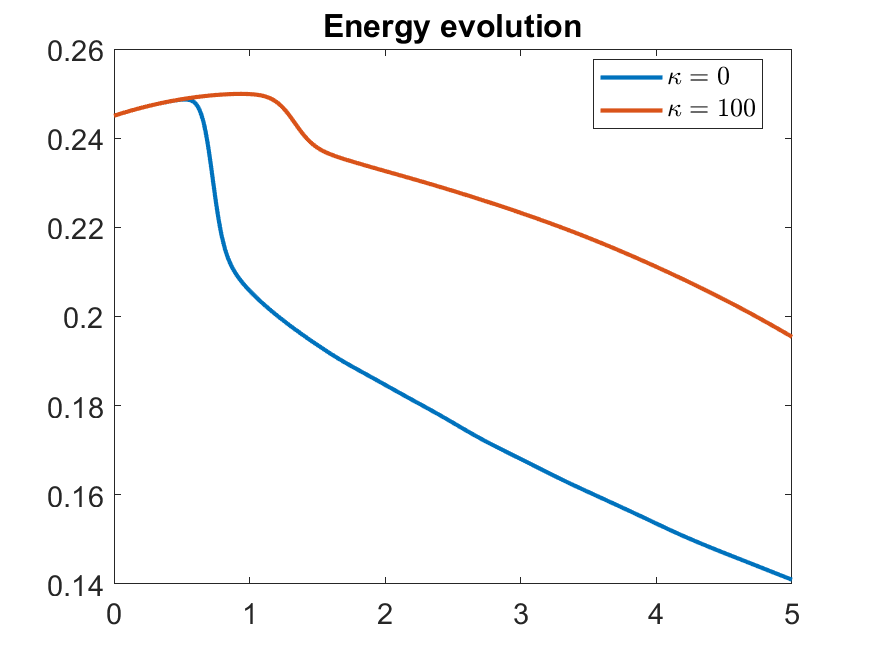}
\caption{\cref{exp:3} for the Ohta--Kawasaki model: Evolution of total mass (left) and the energy $\mathcal{E}$ (right).} \label{pic:exp3struc}
\end{figure}
\section{Conclusion} \label{Sec:Concl}

In this study, we investigated the Ohta--Kawasaki model for the microphase separation of diblock copolymers, focusing particularly on the effects of the repulsion parameter \(\kappa\) and the inclusion of degenerate mobility and external force. We proved the existence of a weak solution in the case of degenerate mobility and external force. This was done by an approximation scheme by lifting the mobility up, allowing us to first consider the case of non-degenerating mobility.  Our numerical findings reveal that \(\kappa\) plays a critical role in the dynamics and structure of phase separation, directly influencing both the rate and pattern of the resulting phases. For small values of \(\kappa\), the behavior of the system closely resembles that of the classical Cahn--Hilliard model, with rapid phase separation and a swift energy decay as the system minimizes interfacial energy. In contrast, larger values of \(\kappa\) slow down the phase separation process significantly, dampening energy dissipation and promoting a stable structure with a higher number of smaller phases, suggesting that \(\kappa\) effectively regulates the stability and morphology of separated phases.


\section*{Acknowledgements}
\noindent M.~Fritz is supported by the state of Upper Austria. A.Brunk~gratefully acknowledges the support of the German Science Foundation (DFG) via TRR~146 (project~C3) and SPP2256 Project Number 441153493 and the support by the Gutenberg Research College, JGU Mainz.

\bibliographystyle{abbrv}
\bibliography{literature.bib}

\begin{thebibliography}{10}

\bibitem{abels2013incompressible}
H.~Abels, D.~Depner, and H.~Garcke.
\newblock On an incompressible {N}avier--{S}tokes/{C}ahn--{H}illiard system
  with degenerate mobility.
\newblock {\em Annales de l'IHP Analyse Non Lin{\'e}aire}, 30:1175--1190, 2013.

\bibitem{abetz2005phase}
V.~Abetz and P.~F. Simon.
\newblock Phase behaviour and morphologies of block copolymers.
\newblock {\em Advances in Polymer Science}, 189:125--212, 2005.

\bibitem{agmon1959estimates}
S.~Agmon, A.~Douglis, and L.~Nirenberg.
\newblock Estimates near the boundary for solutions of elliptic partial
  differential equations satisfying general boundary conditions {I}.
\newblock {\em Communications on Pure and Applied Mathematics}, 12(4):623--727,
  1959.

\bibitem{Barrett1999}
J.~W. Barrett, J.~F. Blowey, and H.~Garcke.
\newblock Finite element approximation of the {C}ahn--{H}illiard equation with
  degenerate mobility.
\newblock {\em SIAM Journal on Numerical Analysis}, 37(1):286--318, 1999.

\bibitem{bates2000block}
F.~S. Bates and G.~Fredrickson.
\newblock Block copolymers-designer soft materials.
\newblock {\em Physics Today}, 52(2):32--38, 2000.

\bibitem{bretin2023multiphase}
E.~Bretin, R.~Denis, S.~Masnou, A.~Sengers, and G.~Terii.
\newblock A multiphase {C}ahn--{H}illiard system with mobilities and the
  numerical simulation of dewetting.
\newblock {\em ESAIM: Mathematical Modelling and Numerical Analysis},
  57(3):1473--1509, 2023.

\bibitem{brezis2010functional}
H.~Brezis.
\newblock {\em Functional Analysis, Sobolev Spaces and Partial Differential
  Equations}.
\newblock Springer, 2010.

\bibitem{brunk2023second}
A.~Brunk, H.~Egger, and O.~Habrich.
\newblock A second-order structure-preserving discretization for the
  cahn-hilliard/allen-cahn system with cross-kinetic coupling.
\newblock {\em Appl. Numer. Math.}, 206:12--28, 2024.

\bibitem{brunk2022second}
A.~Brunk, H.~Egger, O.~Habrich, and M.~Lukacova-Medvidova.
\newblock {A second-order fully-balanced structure-preserving variational
  discretization scheme for the Cahn--Hilliard--Navier--Stokes system}.
\newblock {\em Math. Models Methods Appl. Sci.}, 33(12):2587--2627, 2023.

\bibitem{brunk2023stability}
A.~Brunk, H.~Egger, O.~Habrich, and
  M.~Luk{\'a}{\v{c}}ov{\'a}-Medvi{\v{d}}ov{\'a}.
\newblock Stability and discretization error analysis for the
  {C}ahn--{H}illiard system via relative energy estimates.
\newblock {\em ESAIM: Mathematical Modelling and Numerical Analysis},
  57(3):1297--1322, 2023.

\bibitem{brunk2024structure}
A.~Brunk and M.~Fritz.
\newblock Structure-preserving approximation of the {C}ahn--{H}illiard--{B}iot
  system.
\newblock {\em Numerical Methods for Partial Differential Equations}, to
  appear, 2024.

\bibitem{brunk2023variational}
A.~Brunk, O.~Habrich, T.~D. Oyedeji, Y.~Yang, and B.-X. Xu.
\newblock Variational approximation for a non-isothermal coupled phase-field
  system: Structure-preservation \& nonlinear stability.
\newblock {\em Comput. Methods Appl. Math.}, 2024.

\bibitem{bubba2022nonnegativity}
F.~Bubba and A.~Poulain.
\newblock A nonnegativity preserving scheme for the relaxed {C}ahn--{H}illiard
  equation with single-well potential and degenerate mobility.
\newblock {\em ESAIM: Mathematical Modelling and Numerical Analysis},
  56(5):1741--1772, 2022.

\bibitem{Cao2022}
L.~Cao, O.~Ghattas, and J.~T. Oden.
\newblock A globally convergent modified {N}ewton method for the direct
  minimization of the {O}hta-{K}awasaki energy with application to the directed
  self-assembly of diblock copolymers.
\newblock {\em SIAM J. Sci. Comput.}, 44(1):B51--B79, 2022.

\bibitem{chen2021fully}
C.~Chen and X.~Yang.
\newblock Fully-discrete finite element numerical scheme with decoupling
  structure and energy stability for the {C}ahn--{H}illiard phase-field model
  of two-phase incompressible flow system with variable density and viscosity.
\newblock {\em ESAIM: Mathematical Modelling and Numerical Analysis},
  55(5):2323--2347, 2021.

\bibitem{chen2020directed}
Y.~Chen and S.~Xiong.
\newblock Directed self-assembly of block copolymers for sub-10 nm fabrication.
\newblock {\em International Journal of Extreme Manufacturing}, 2(3):032006,
  2020.

\bibitem{cherfils2021convergent}
L.~Cherfils, H.~Fakih, M.~Grasselli, and A.~Miranville.
\newblock A convergent convex splitting scheme for a nonlocal
  {C}ahn--{H}illiard--{O}ono type equation with a transport term.
\newblock {\em ESAIM: Mathematical Modelling and Numerical Analysis},
  55:S225--S250, 2021.

\bibitem{Choksi2012}
R.~Choksi.
\newblock On global minimizers for a variational problem with long-range
  interactions.
\newblock {\em Quart. Appl. Math.}, 70(3):517--537, 2012.

\bibitem{Choksi2011}
R.~Choksi, M.~Maras, and J.~F. Williams.
\newblock 2{D} phase diagram for minimizers of a {C}ahn-{H}illiard functional
  with long-range interactions.
\newblock {\em SIAM J. Appl. Dyn. Syst.}, 10(4):1344--1362, 2011.

\bibitem{Choksi2009}
R.~Choksi, M.~A. Peletier, and J.~F. Williams.
\newblock On the phase diagram for microphase separation of diblock copolymers:
  an approach via a nonlocal {C}ahn-{H}illiard functional.
\newblock {\em SIAM J. Appl. Math.}, 69(6):1712--1738, 2009.

\bibitem{Choksi2003}
R.~Choksi and X.~Ren.
\newblock On the derivation of a density functional theory for microphase
  separation of diblock copolymers.
\newblock {\em J. Statist. Phys.}, 113(1-2):151--176, 2003.

\bibitem{Ebenbeck21}
M.~Ebenbeck, H.~Garcke, and R.~Nürnberg.
\newblock Cahn--{H}illiard--{B}rinkman systems for tumour growth.
\newblock {\em Discrete and Continuous Dynamical Systems - S},
  14(11):3989--4033, 2021.

\bibitem{egger2019structure}
H.~Egger.
\newblock Structure preserving approximation of dissipative evolution problems.
\newblock {\em Numer. Math.}, 143:85--106, 2019.

\bibitem{elbar2024analysis}
C.~Elbar and A.~Poulain.
\newblock Analysis and numerical simulation of a generalized compressible
  {C}ahn--{H}illiard--{N}avier--{S}tokes model with friction effects.
\newblock {\em ESAIM: Mathematical Modelling and Numerical Analysis},
  58(5):1989--2034, 2024.

\bibitem{elliott1996cahn}
C.~M. Elliott and H.~Garcke.
\newblock On the {C}ahn--{H}illiard equation with degenerate mobility.
\newblock {\em SIAM Journal on Mathematical Analysis}, 27(2):404--423, 1996.

\bibitem{elliott1993global}
C.~M. Elliott and A.~Stuart.
\newblock The global dynamics of discrete semilinear parabolic equations.
\newblock {\em SIAM Journal on Numerical Analysis}, 30(6):1622--1663, 1993.

\bibitem{Fredrickson2005}
G.~H. Fredrickson.
\newblock {\em The Equilibrium Theory of Inhomogeneous Polymers}.
\newblock Oxford University Press, 2006.

\bibitem{fritz2023tumor}
M.~Fritz.
\newblock Tumor evolution models of phase-field type with nonlocal effects and
  angiogenesis.
\newblock {\em Bulletin of Mathematical Biology}, 85(6):44, 2023.

\bibitem{fritz2021equivalence}
M.~Fritz, U.~Khristenko, and B.~Wohlmuth.
\newblock Equivalence between a time-fractional and an integer-order gradient
  flow: {T}he memory effect reflected in the energy.
\newblock {\em Advances in Nonlinear Analysis}, 12:20220262, 2023.

\bibitem{fritz2018unsteady}
M.~Fritz, E.~Lima, J.~T. Oden, and B.~Wohlmuth.
\newblock On the unsteady {D}arcy--{F}orchheimer--{B}rinkman equation in local
  and nonlocal tumor growth models.
\newblock {\em Mathematical Models and Methods in Applied Sciences},
  29:1691--1731, 2019.

\bibitem{fritz2022time}
M.~Fritz, M.~L. Rajendran, and B.~Wohlmuth.
\newblock Time-fractional {C}ahn--{H}illiard equation: {W}ell-posedness,
  degeneracy, and numerical solutions.
\newblock {\em Computers \& Mathematics with Applications}, 108:66--87, 2022.

\bibitem{gomez2011provably}
H.~Gomez and T.~J. Hughes.
\newblock Provably unconditionally stable, second-order time-accurate, mixed
  variational methods for phase-field models.
\newblock {\em Journal of Computational Physics}, 230(13):5310--5327, 2011.

\bibitem{gurtin1996generalized}
M.~E. Gurtin.
\newblock Generalized {G}inzburg--{L}andau and {C}ahn--{H}illiard equations
  based on a microforce balance.
\newblock {\em Physica D: Nonlinear Phenomena}, 92(3-4):178--192, 1996.

\bibitem{hamley1998physics}
I.~W. Hamley.
\newblock {\em The Physics of Block Copolymers}.
\newblock Oxford Science Publications, 1998.

\bibitem{kim2010functional}
J.~K. Kim, S.~Y. Yang, Y.~Lee, and Y.~Kim.
\newblock Functional nanomaterials based on block copolymer self-assembly.
\newblock {\em Progress in Polymer Science}, 35(11):1325--1349, 2010.

\bibitem{leibler1980theory}
L.~Leibler.
\newblock Theory of microphase separation in block copolymers.
\newblock {\em Macromolecules}, 13(6):1602--1617, 1980.

\bibitem{muller2018continuum}
M.~M{\"u}ller and J.~C.~O. Rey.
\newblock Continuum models for directed self-assembly.
\newblock {\em Molecular Systems Design \& Engineering}, 3(2):295--313, 2018.

\bibitem{ohta1986equilibrium}
T.~Ohta and K.~Kawasaki.
\newblock Equilibrium morphology of block copolymer melts.
\newblock {\em Macromolecules}, 19(10):2621--2632, 1986.

\bibitem{qin2013evolutionary}
J.~Qin, G.~S. Khaira, Y.~Su, G.~P. Garner, M.~Miskin, H.~M. Jaeger, and J.~J.
  de~Pablo.
\newblock Evolutionary pattern design for copolymer directed self-assembly.
\newblock {\em Soft Matter}, 9(48):11467--11472, 2013.

\bibitem{robinson2001infinite}
J.~C. Robinson.
\newblock {\em Infinite-Dimensional Dynamical Systems: An Introduction to
  Dissipative Parabolic {PDE}s and the Theory of Global Attractors}.
\newblock Cambridge University Press, 2001.

\bibitem{schoberl2014c++}
J.~Sch{\"o}berl.
\newblock {C$++$11 implementation of finite elements in NGSolve}.
\newblock {\em ASC (Institute for Analysis and Scientific Computing) Report, TU
  Wien}, 30:1--23, 2014.

\bibitem{shimura2020error}
K.~Shimura and S.~Yoshikawa.
\newblock {Error estimate for structure-preserving finite difference schemes of
  the one-dimensional Cahn--Hilliard system coupled with viscoelasticity}.
\newblock In {\em Regularity and Asymptotic Analysis for Critical Cases of
  Partial Differential Equations}, pages 159--175. Research Institute for
  Mathematical Sciences, Kyoto University, 2020.

\bibitem{stefik2015block}
M.~Stefik, S.~Guldin, S.~Vignolini, U.~Wiesner, and U.~Steiner.
\newblock Block copolymer self-assembly for nanophotonics.
\newblock {\em Chemical Society Reviews}, 44(15):5076--5091, 2015.

\bibitem{temam2012infinite}
R.~Temam.
\newblock {\em Infinite-Dimensional Dynamical Systems in Mechanics and
  Physics}.
\newblock Springer, 2012.

\bibitem{van_den_Berg_2017}
J.~B. van~den Berg and J.~F. Williams.
\newblock Validation of the bifurcation diagram in the 2{D} {O}hta-{K}awasaki
  problem.
\newblock {\em Nonlinearity}, 30(4):1584--1638, 2017.

\bibitem{xia2007local}
Y.~Xia, Y.~Xu, and C.-W. Shu.
\newblock Local discontinuous galerkin methods for the {C}ahn--{H}illiard type
  equations.
\newblock {\em Journal of Computational Physics}, 227(1):472--491, 2007.

\bibitem{Xu2019}
X.~Xu and Y.~Zhao.
\newblock Energy stable semi-implicit schemes for
  {A}llen--{C}ahn--{O}hta--{K}awasaki model in binary system.
\newblock {\em Journal of Scientific Computing}, 80(3):1656–1680, June 2019.

\bibitem{yang2008virus}
S.~Y. Yang, J.~Park, J.~Yoon, M.~Ree, S.~K. Jang, and J.~K. Kim.
\newblock Virus filtration membranes prepared from nanoporous block copolymers
  with good dimensional stability under high pressures and excellent solvent
  resistance.
\newblock {\em Advanced Functional Materials}, 18(9):1371--1377, 2008.

\bibitem{Zeidler1}
E.~Zeidler.
\newblock {\em Nonlinear Functional Analysis and its Applications~{I}:
  Fixed-Point Theorems}.
\newblock Springer, 1986.

\bibitem{zou2022fully}
G.~Zou, B.~Wang, and X.~Yang.
\newblock A fully-decoupled discontinuous galerkin approximation of the
  {C}ahn--{H}illiard--{B}rinkman--{O}hta--{K}awasaki tumor growth model.
\newblock {\em ESAIM: Mathematical Modelling and Numerical Analysis},
  56(6):2141--2180, 2022.

\end{thebibliography}

\end{document}